\documentclass[aop, preprint]{imsart}
\usepackage[utf8]{inputenc}
\usepackage{amsmath,amsthm,amssymb}
\usepackage{graphicx}
\usepackage{url}
\usepackage[breaklinks=true,unicode]{hyperref}
\usepackage{enumerate}

\startlocaldefs

\newcommand{\overbar}[1]{\mkern 1.5mu\overline%
{\mkern-1.5mu#1\mkern-1.5mu}\mkern 1.5mu}

\newcommand{\E}{\ensuremath{\mathbb{E}}}
\newcommand{\N}{\ensuremath{\mathbb{N}}}
\newcommand{\R}{\ensuremath{\mathbb{R}}}
\newcommand{\Z}{\ensuremath{\mathbb{Z}}}
\newcommand{\BBM}{{\mathsf{BBM}}}
\newcommand{\ei}{\ensuremath{\mathsf{ei}}}
\newcommand{\es}{\ensuremath{\mathsf{es}}}
\newcommand{\floor}[1]{{\lfloor #1 \rfloor}}
\newcommand{\ind}{\boldsymbol{1}}
\newcommand{\ttE}{\ensuremath{\mathtt{E}}}
\newcommand{\ttP}{\ensuremath{\mathtt{P}}}
\renewcommand{\P}{\ensuremath{\mathbb{P}}}
\renewcommand{\d}{\ensuremath{\mathrm{d}}}

\DeclareMathOperator{\Cov}{Cov}
\DeclareMathOperator{\Var}{Var}

\DeclareMathOperator*{\essinf}{ess\,inf}
\DeclareMathOperator*{\esssup}{ess\,sup}

\DeclareMathOperator{\dist}{dist}

\newtheorem{theorem}{Theorem}[section]
\newtheorem{claim}[theorem]{Claim}

\newtheorem{corollary}[theorem]{Corollary}
\newtheorem{lemma}[theorem]{Lemma}
\newtheorem{proposition}[theorem]{Proposition}

\theoremstyle{remark}
\newtheorem{remark}[theorem]{Remark}

\theoremstyle{definition}

\newcommand{\ignore}[1]{{}}

\numberwithin{equation}{section}
\endlocaldefs

\begin{document} 

\begin{frontmatter}

\title{Quenched invariance principles for the maximal particle in
  branching random walk in random environment and the parabolic
  Anderson model}
\runtitle{Quenched invariance principles for BRWRE}

\begin{aug}
  \author{\fnms{Jiří}
    \snm{Černý}\corref{}\ead[label=e1]{jiri.cerny@unibas.ch}}
  \and
  \author{\fnms{Alexander}
    \snm{Drewitz}\ead[label=e2]{drewitz@math.uni-koeln.de}}

  \runauthor{J.~Černý and A.~Drewitz}

  \affiliation{Universität Basel and Universität zu Köln}

  \address{Department of mathematics and computer science\\
    University of Basel\\
    Spiegelgasse 1\\
    4051 Basel, Switzerland\\
           \printead{e1}}

  \address{Mathematisches Institut\\
  Universität zu Köln\\
    Weyertal 86--90\\
    50931 Köln, Germany\\
          \printead{e2}}

\end{aug}

\begin{abstract}
  We consider branching random walk in spatial random branching
  environment (\hbox{BRWRE}) in dimension one, as well as related
  differential equations: the Fisher-KPP equation with random branching
  and its linearized version, the parabolic Anderson model (PAM). When
  the random environment is bounded, we show that after recentering and
  scaling, the position of the maximal particle of the BRWRE, the front
  of the solution of the PAM, as well as the front of the solution of the
  randomized Fisher-KPP equation fulfill quenched invariance principles.
  In addition, we prove that at time $t$ the distance between the median
  of the maximal particle of the BRWRE and the front of the solution of
  the PAM is in $O(\ln t)$. This partially transfers classical results of
  Bramson \cite{Br-78} to the setting of BRWRE.
\end{abstract}



\end{frontmatter}

\date{\today}
\maketitle

\tableofcontents

\section{Introduction} 
\label{sec:intro}

Branching random walk as well as branching Brownian motion, and in particular the
position of their maximally displaced particles, have been the subject of
highly intensive research during the last couple of decades, see the
monographs \cite{Sh-15, Bo-16} as well as the references in these sources.

Indeed, in \cite{Ha-74}, \cite{Ki-75}, and  \cite{Bi-76} it has
successively been shown that under suitable assumptions  the position of
the maximal or rightmost particle $M(n)$ of the branching random walk at
time $n$ satisfies a law of large numbers; i.e., almost surely
\begin{equation}
  \label{eq:maxLLN}
  \lim_{n \to \infty} n^{-1}{M(n)} = \tilde v_0
\end{equation}
for some non-random $\tilde v_0 \in \R$. Subsequently, concentration
results for $M(n)$ around its median $m(n)$,
cf.~\cite{McD-95,Re-03,Dr-03,ChDr-06}, as well as corresponding results
on the  distributional convergence have been obtained, see
\cite{Ba-00,BrZe-09, BrZe-10,Ai-13}. In particular, in
\cite{AdRe-09,HuSh-09} the law of large numbers of \eqref{eq:maxLLN} has
been improved in that for a wide class of branching random walks the
position of the maximal particle $M(n)$ at time $n$ satisfies
\begin{equation}
  \label{eq:brwMin}
  M(n) = \tilde v_0 n - \frac{3}{2} c \ln n + O(1),
\end{equation}
where $c > 0$ is a parameter depending on the specifics of the branching
and displacement mechanisms.

In the continuum setting of branching Brownian motion (BBM) with binary
branching, replacing $n$ by $t$ in a suggestive way for the respective
quantities, even more precise asymptotics, namely
\begin{equation} \label{eq:bramson}
  m(t)= \sqrt 2 t - \frac{3}{2\sqrt 2} \ln t + o(1),
\end{equation}
has been proved much earlier in seminal works by Bramson
\cite{Br-78,Br-83} already.
Bramson made  use of the fact that the function
$w^\BBM(t,x):=\P(M(t) \ge x)$, $t\ge 0$, $x\in \mathbb R$, solves the
Fisher-KPP equation
\begin{equation}
  \label{eq:detKPP}
    \frac{\partial w^\BBM}{\partial t}
    =\frac 12 \Delta w^\BBM + w^\BBM (1-w^\BBM),
\end{equation}
with the initial condition $w^\BBM(0, \cdot) = \ind_{(-\infty,0]}$.
He then investigated the solution to this equation through an
impressively refined analysis of its Feynman-Kac representation.

While the above results for branching random walk have been derived in
the context of homogeneous branching mechanisms, there has recently been
an increased activity in the investigation of branching random walk with
non-homogeneous branching rates that depend on either time or space in
special \emph{deterministic} ways, see \cite{LaSe-88,LaSe-89, FaZe-12,
  FaZe-12b, BeBrHaHaRo-12, NoRoRy-13, MaZe-13, Ma-13, BoHa-14,BoHa-15}.
Among other things, as an interesting consequence of the inhomogeneous
branching rates, in these sources second order terms that differ from the
logarithmic correction of \cite{Br-78} and \eqref{eq:brwMin} have been
obtained.

While the above sources focus on the case of deterministic branching
environments, there are very compelling reasons for trying to achieve a
better understanding of the case of spatially random branching
environments. On the one hand, this is already interesting from a purely
mathematical point of view. On the other hand, when it comes to modeling
real world applications, though branching environments are not  random,
they oftentimes are locally irregular but exhibit certain spatial
averaging properties. One natural approach is then to model the
environment as random and try to understand the evolution of the process
either conditionally on a realization of the branching environment or
averaged over all such environments. In this setting, notable research
has been conducted over the past decades on a variety of aspects such as
survival and growth properties, transience vs.~recurrence, diffusivity,
as well as localization properties  (see
  e.g.~\cite{GrHo-92,MaPo-00,CoPo-07, CoPo-07a, HuYo-09, GaMuPoVa-10,
    HeNaYo-11, Na-11, OrRo-16} for a non-exhaustive list).

To the best of our knowledge, the only source that  in some sense focuses
on the maximal particle is Comets and Popov \cite{CoPo-07a}. They prove a
shape theorem for a BRWRE on $\mathbb Z^d$, $d\ge 1$, from which, as a
corollary, one can infer that the maximal particle has an asymptotic
velocity, that is \eqref{eq:maxLLN} holds.

Finally, branching random walk in an environment that is changing randomly
in \emph{time} was studied in \cite{HuLi-14,MaMi-15a} recently.
Among other results, Huang and Liu \cite{HuLi-14} proved a law of large
numbers for the maximal particle. Mallein and Miłoś \cite{MaMi-15a}
considered the backlog of the maximal particle behind what can be
interpreted as the breakpoint in their setting (cf.~\eqref{eq:origBP}
  below) and proved that it is strictly larger than in the setting of
constant branching rates. As a corollary, their results yield a central
limit theorem for the position of the maximally displaced particle. It
should be noted here that the time-dependent random environment seems to
be easier to handle since certain techniques of the theory of multi-type
branching processes apply in this case. We were not able to use them for
the model considered in this paper.

\section{Definition of the model and main results} 

Let us now introduce the model of branching random walk in random
(branching) environment considered in this paper. The random environment
is given by a family $\xi = (\xi(x))_{x \in \Z}$ of random variables
defined on some probability space $(\Omega, \mathcal F, \P)$. We assume
that the environment is i.i.d.~and bounded:
\begin{align}
  \label{eq:xiAssumptions} \tag{POT}
  \begin{split}
    &(\xi(x))_{x \in \Z} \text{ are i.i.d.~under $\P$},\\
    &0 < \ei := \essinf \xi(0) < \esssup \xi(0) =: \es < \infty,\\
  \end{split}
\end{align}

We presume that the i.i.d.~property can be relaxed with some additional
technical effort, but we prefer to work in this context for the sake of
simplicity. The same holds true for the condition $\ei >0$ which could be
relaxed to $\ei \ge c\in \mathbb R$, with negative branching rates being
interpreted as killing rates. On the other hand, some form of boundedness
of $\xi (x)$ from above is essential for our investigations.

We furthermore assume, again for reasons of simplicity, that the initial
configuration $u_0: \Z \to \N_0$ is such that
\begin{align}
  \label{eq:inCond} \tag{INI}
  \begin{split}
    C \ind_{-\N_0} \ge u_0 \ge
    \ind_{\{0\}} \quad \text{for some $C\in [1,\infty)$.}
  \end{split}
\end{align}

In particular, $u_0 = \ind_{\{0\}}$ and $u_0=\ind_{-\N_0}$ fulfill
\eqref{eq:inCond}. Later, as a consequence of Lemmas~\ref{lem:inCond}
and~\ref{lem:upperTail} below, we show that any  initial configuration
satisfying \eqref{eq:inCond} is comparable for our purposes to
$u_0 = \ind_{\{0\}}$ in the results that follow. Hence, the reader may
assume $u_0 = \ind_{\{0\}}$ from now onwards without loss of generality.

Let us now describe the dynamics of the BRWRE in detail. Given a realization
of $\xi $ and an initial condition $u_0 : \Z \to \N_0$, at
each $x \in \Z$ we place $u_0(x)$ particles at time $0$. As time evolves,
all particles move independently according to continuous time simple
random walk with jump rate $1$. In addition, and independently of
everything else, while at a site $x$, a particle splits into two at rate
$\xi(x)$, and if it does so, the two new particles evolve independently
according to the same diffusion mechanism as the remaining particles.
This defines \emph{branching random walk in the branching environment}
$\xi$ with binary branching, where again the latter is for simplicity but
not essential. Given a realization of $\xi $, we write
$\ttP_{u_0}^\xi$ for the quenched law of the process conditional on
starting with a particle configuration $u_0$ at time $0$, and
$\ttE^\xi_{u_0}$ for the corresponding expectation. We use
$\mathbb P\times \ttP^\xi_{u_0}$ to denote the averaged law of the
process. To simplify notation, we abbreviate
$\ttP_{x}^\xi = \ttP_{\ind_{\{x\}}}^\xi$.

We use $N(t)$ to denote the set of particles alive at time $t$ in this
BRWRE. For any particle $Y \in N(t)$, we denote by $(Y_s)_{s \in [0,t]}$
the trajectory of itself and its ancestors up to time $t$. We will also
call $(Y_s)_{s \in [0,t]}$ the \emph{genealogy} of $Y$. For $t\ge 0$ and
$x\in \Z$, we define
\begin{equation}
  \label{eq:expPart}
  \begin{split}
    &N(t,x)
    := \big\vert \{ Y \in N(t) \, : Y_t = x\} \big\vert
 \quad \text{ and }\\
    &N^\ge(t,x) :=\big\vert \{ Y \in N(t) \, : Y_t \ge x\} \big\vert =
    \sum_{y \ge x} N(t,y)
  \end{split}
\end{equation}
as the number of particles in the process at time $t$ which are located
at or to the right of $x$.

To state our last assumption, we recall that it is well known from the
studies on the parabolic Anderson model (cf.~Section~\ref{ssec:PAM}
  below) that there is a deterministic function
$\lambda :\mathbb R\to \mathbb R$, the \emph{Lyapunov exponent}, such
that for a.e.~realization of $\xi $ the quenched expectation of $N(t,x)$
satisfies
\begin{equation}
  \label{eq:lyapunov}
  \lambda(v) =
  \lim_{t \to \infty} \frac1t \ln \ttE^\xi_{0} \big[  N(t,\lfloor
      tv \rfloor)  \big], \qquad v\in \mathbb R.
\end{equation}
Under \eqref{eq:xiAssumptions}, one can show that  $\lambda $ is even,
concave everywhere and strictly concave exactly on $(v_c,\infty)$ for
some non-trivial critical velocity $v_c\in (0,\infty)$,
see~Figure~\ref{fig:lambda} for the illustration and
Proposition~\ref{prop:lyapExp} in the Appendix for the proof. Furthermore, the
asymptotic velocity of the maximally displaced particle
(cf.~\eqref{eq:maxLLN}) is given by the unique $v_0\in (0,\infty)$ such that
\begin{equation}
  \label{eq:vzero}
  \lambda({v_0}) = 0.
\end{equation}
Throughout the paper we will assume that the maximally displaced particle is
faster than $v_c$, that is
\begin{equation}
  \label{eq:vAssumptions} \tag{VEL}
  v_0> v_c.
\end{equation}
\begin{figure}
  \includegraphics{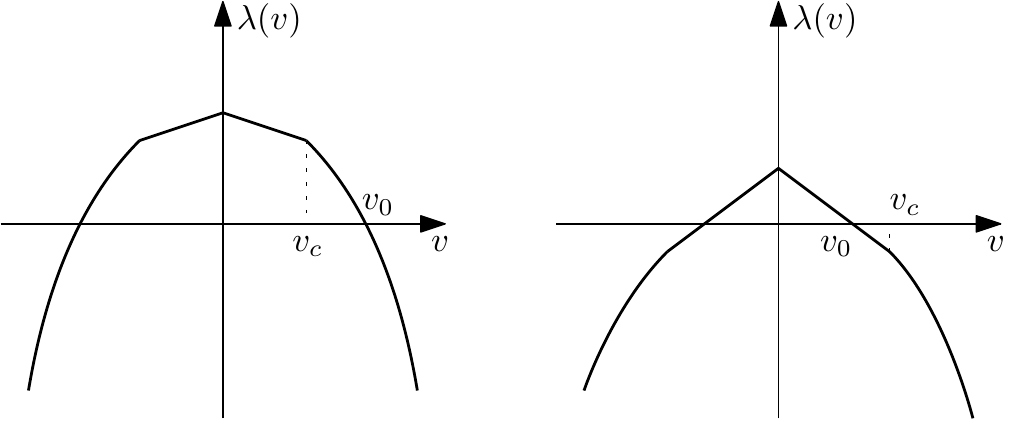}
  \caption{\label{fig:lambda}
    Qualitative illustration of the behavior of the Lyapunov exponent
    $\lambda (v)$ with \eqref{eq:vAssumptions} satisfied (left) or not
    (right). In particular, the Lyapunov exponent is a linear function on
    two non-degenerate symmetric intervals adjacent to the origin, and
    strictly convex otherwise; see Proposition~\ref{prop:lyapExp} for details.}
\end{figure}%
This assumption will ensure that there is certain tilted Gibbs measure
related to BRWRE (cf.~\eqref{eq:tiltedVarTime} and below) under which the
particles have the speed $v_0;$ the existence of such a measure is
crucial for the techniques employed in this paper. While condition
\eqref{eq:vAssumptions} is not easy to check in general, in
Lemma~\ref{lem:qualLyap} of the Appendix we show that it is satisfied for
a rich family of random environments. Moreover, so far we have found no
examples where \eqref{eq:vAssumptions} fails to hold, but a proof that
\eqref{eq:vAssumptions} is always fulfilled eludes us so far. Figure
\ref{fig:lambda} covers possible shapes of the Lyapunov exponent in terms
of convexity and the locations of  $v_0$ and~$v_c$.

\subsection{Behavior of the maximally displaced particle} 

From a probabilistic point of view, in this article we are mainly
interested in the behavior of the position of the maximally displaced
particle at time $t$,
\begin{equation*}
  M(t) := \max \{ Y_t  :  Y \in N(t) \},
\end{equation*}
for which we prove the following functional central limit theorem.

\begin{theorem}
  \label{thm:maxFCLT}
  Assume \eqref{eq:xiAssumptions}, \eqref{eq:inCond} and
  \eqref{eq:vAssumptions}. Then there is
  $\overbar \sigma_{v_0}\in (0,\infty)$ given explicitly in
  \eqref{eq:barsigmav} below, such that the
  sequence of processes
  \begin{equation*}
    [0,\infty) \ni t \mapsto
    \frac{ M(n t) - v_0 n t}{\overbar\sigma_{v_0} \sqrt n},
    \quad n \in \N,
  \end{equation*}
  converges  as $n\to\infty$ in $\mathbb P\times \ttP^\xi_{u_0}$-distribution
  to standard Brownian motion.
\end{theorem}

\begin{remark}
  Without further mentioning, in the functional central limit theorems we
  prove, we consider the space of càdlàg functions endowed with the
  Skorokhod topology as the underlying space.
\end{remark}

Theorem~\ref{thm:maxFCLT} will directly follow from three intermediate
results (Proposition~\ref{prop:Mclosetom} and
  Theorems~\ref{thm:logCorrection}, \ref{thm:bpFCLT}  below) which are of
independent interest. To state these results, we define $m(t)$ as the
quenched median of the distribution of $M(t)$,
\begin{equation}
  \label{eq:median}
  m(t):=\sup \big\{x \in \Z  :  \ttP_{u_0}^\xi (M(t)\ge x)\ge 1/2 \big\}.
\end{equation}
Note here that $m(t)$ is a random variable on $(\Omega, \mathcal F, \P)$.

We further introduce a quantity $\overbar m(t)$ which is sometimes
referred to as the  \emph{breakpoint} in the case of homogeneous
branching rates; incidentally, we already remark at this point that in
our setting it is also instructive to interpret it as the front of the
solution to the parabolic
Anderson model, cf.~Section~\ref{ssec:PAM} below. It is defined~as
\begin{equation}
  \label{eq:origBP}
  \overbar m(t)
  := \sup \big \{ x \in \Z \, : \,
    \ttE^\xi_{u_0} \big[  N^\ge(t,x) \big]  \ge 1/2
    \big\}.
\end{equation}

As the first ingredient of the proof of Theorem~\ref{thm:maxFCLT}, we show
that $M(t)$ is sufficiently close to its median so that, for the sake of the
functional central limit theorem, $M(t)$ can effectively be replaced by
$m(t)$.
\begin{proposition}
  \label{prop:Mclosetom}
  Under assumptions \eqref{eq:xiAssumptions}, \eqref{eq:inCond} and
  \eqref{eq:vAssumptions}, there is a constant $C \in (0,\infty)$ such
  that $\mathbb P\times \mathtt P^\xi_{u_0}$-a.s.,
  \begin{equation*}
    \limsup_{t\to\infty} \frac{|M(t)-m(t)|}{\ln t} \le C.
  \end{equation*}
\end{proposition}

The second substantial step to show Theorem~\ref{thm:maxFCLT} is the
following approximation result. It is one of the main results of this
article and it is interesting in its own right.
\begin{theorem}
  \label{thm:logCorrection}
  Assume \eqref{eq:xiAssumptions}, \eqref{eq:inCond} and
  \eqref{eq:vAssumptions} to hold. Then $m(t)\le \overbar m(t)$,
  and there exists a constant $C \in (0,\infty)$ such that for
  $\P$-a.e.~realization of $\xi$,
  \begin{equation} \label{eq:logCorrection}
    \limsup_{t\to\infty} \frac{\overbar m(t)- m(t) }{\ln t}  \le C.
  \end{equation}
\end{theorem}

\begin{remark}
  This result should be compared to the classical results of Bramson
  \cite{Br-78, Br-83} for homogeneous BBM (and to corresponding results
    for branching random walk \cite{AdRe-09,HuSh-09}). In the case of BBM
  the breakpoint satisfies
  $\overbar m(t) = \sqrt 2 t- \frac{1}{2\sqrt 2} \ln t + o(1)$ which can
  be proved easily using Gaussian tail estimates. Together with
  \eqref{eq:bramson},  this yields that for BBM,
  \begin{equation*}
    \lim_{t\to\infty} \frac{\overbar m(t) - m(t) }{\ln t}
    = \frac 1 {\sqrt 2}.
  \end{equation*}
  Our result thus shows that in the case of random branching rates we
  can recover an upper bound whose order matches that of the
  homogeneous branching setting.  The question of whether
  there is a limit in \eqref{eq:logCorrection} remains open.
\end{remark}

The third and last ingredient of the proof of Theorem~\ref{thm:maxFCLT}
is the functional central limit theorem for a suitably rescaled and
centered version of the process $\overbar m(t)$. In fact, we prove a
slightly more general statement: As a generalization to
\eqref{eq:origBP}, we define
\begin{equation}
  \label{eq:mvt}
  \overbar m_v(t) := \sup \Big \{ x \in \N \, : \,
    \ttE^\xi_{u_0} \big[  N^\ge(t,x) \big]
      \ge \frac 12
    e^{t\lambda(v)}  \Big \},\qquad v> 0,\ t>0,
\end{equation}
where $\lambda $ is the Lyapunov exponent defined in \eqref{eq:lyapunov}.
Note that due to the definition \eqref{eq:vzero} of $v_0$ we have
$\overbar m(t)=\overbar m_{v_0}(t)$.

\begin{theorem}
  \label{thm:bpFCLT}
  Under assumptions \eqref{eq:xiAssumptions} and \eqref{eq:inCond}, for
  every $v>v_c$, the sequence of processes
  \begin{equation*}
    [0,\infty) \ni t \mapsto
    \frac{\overbar m_v(n t) -  v n t}{\overbar\sigma_v \sqrt n},
    \quad n \in \N,
  \end{equation*}
  converges  as $n\to\infty$ in $\mathbb P$-distribution to standard
  Brownian motion. The value of $\overbar\sigma_v \in (0,\infty)$ is
  given in \eqref{eq:barsigmav} below.
\end{theorem}

Theorem~\ref{thm:maxFCLT} follows directly from
Proposition~\ref{prop:Mclosetom} and Theorems~\ref{thm:logCorrection},
\ref{thm:bpFCLT}. Combining these results we  also immediately obtain a
functional limit theorem for the median~$m(t)$:
\begin{corollary}
  Assuming \eqref{eq:xiAssumptions}, \eqref{eq:inCond} and
  \eqref{eq:vAssumptions}, with
  $\overbar\sigma_{v_0}$ as in Theorem~\ref{thm:bpFCLT}, the
  sequence of processes
  \begin{equation*}
    [0,\infty) \ni t
    \mapsto  \frac{ m(n t) - v_0 n t}{\overbar\sigma_{v_0} \sqrt n},
    \quad n \in \N,
  \end{equation*}
  converges  as $n\to\infty$ in $\mathbb P$-distribution to standard
  Brownian motion.
\end{corollary}

\subsection{Implications for the PAM and randomized Fisher-KPP equation}
\label{ssec:PAM}

As we have touched upon previously in the introduction, there is a close
connection between certain partial differential equations and branching
processes: In the case of BBM, it is easy to see that the
density $u^\BBM(t,x)$ of the expected number of particles satisfies
\begin{equation}
  \label{eq:ubbm}
  \frac{\partial}{\partial t} u^\BBM
  =\frac 12\Delta u^\BBM+u^\BBM.
\end{equation}
As this equation is essentially the heat equation (write
  $u^\BBM=e^t \tilde u$), this allows to estimate the corresponding
breakpoint $\sup\{x>0:u(t,x)\ge 1/2\}$ with high accuracy using Gaussian
tail estimates. Moreover, as we already mentioned,
$w^\BBM(t,x)=\mathbb P(M(t)\ge x)$ satisfies the Fisher-KPP
equation~\eqref{eq:detKPP}. In particular, the front of the solution to
\eqref{eq:detKPP}, defined as $\sup\{ x \in \R  :  w^\BBM(t,x) \ge 1/2\}$,
coincides with the median $m(t)$ of the distribution of the maximal
particle of the BBM. Therefore, Bramson's result \eqref{eq:bramson}
immediately gives equally precise information on the position of the
front of the solution to \eqref{eq:detKPP} as well.

In our setting of inhomogeneous branching rates the situation is both
more complicated but also more interesting.  The breakpoint in the case
of heterogeneous branching rates corresponds to the front of the solution
to the \emph{parabolic Anderson model} (PAM), a discrete randomized
version of~\eqref{eq:ubbm},
\begin{equation}
  \begin{split}
    \label{eq:PAM}
    \frac{\partial u }{\partial t} (t,x)
    &= \Delta_{\mathrm d} u  (t,x) + \xi (x ) u(t,x),
    \qquad t\ge 0, x\in \mathbb Z, \\
    u (0,x)&=u_0(x), \qquad x\in \mathbb Z,
  \end{split}
\end{equation}
Here, $\Delta_{\mathrm d}  f(x)=  \frac12 (f(x+1)+f(x-1) -2f(x))$ stands for the
discrete Laplace operator.

It is well-known that  conditionally on $\xi$, the expected number of
particles at time $t$ and position $x$
\begin{equation}
  \label{eq:uxiN}
  u (t,x) := \ttE_{u_0}^\xi [N(t,x)]
\end{equation}
solves \eqref{eq:PAM} (cf.~the original source \cite{GaMo-90} as well as
  \cite{GaKo-05} and \cite{Ko-16} for more recent surveys).
Hence, due to \eqref{eq:origBP} and \eqref{eq:uxiN}, the process
$\overbar m(t)$ can be viewed as the \emph{front} of the solution to the
PAM, which, according to Theorem \ref{thm:bpFCLT},
fulfills a corresponding functional central limit theorem.

This functional central limit theorem can be supplied with another one,
for the logarithm of the function $u(t,x)$ itself: Since statement
\eqref{eq:lyapunov} can  be read as a law of large numbers for
$t^{-1}\ln u (t,\floor{tv})$, it is natural to inquire about the
fluctuations. Our investigations  lead to a corresponding invariance
principle  which is of independent interest.
\begin{theorem}
  \label{thm:PAMFCLT}
  Under assumptions \eqref{eq:xiAssumptions} and \eqref{eq:inCond}, for
  every $v>v_c$ there exists $\sigma_v\in (0,\infty)$ given explicitly in
  \eqref{eq:limitSigma} below, such that the sequence of processes
  \begin{equation*}
    [0,\infty) \ni t \mapsto \frac{1}{\sigma_v\sqrt {vn}}
    \big(\ln u (nt,\floor{vnt}) - nt \lambda (v)\big),
    \qquad n\in \N,
  \end{equation*}
  converges as $n\to\infty$ in $\P$-distribution to standard Brownian
  motion.
\end{theorem}

While this result for the front of the solution to the PAM is interesting
in its own right, the question naturally arises of what one can say about
the front of the solution to its non-linear version, the randomized
discrete Fisher-KPP equation
\begin{equation} \label{eq:dKPP}
  \frac{\partial w}{\partial t}(t,x)
  = \Delta_{\mathrm d} w(t,x) + \xi (x)w(t,x) (1-w(t,x)), \qquad
  t\ge 0,\, x\in \Z,
\end{equation}

Previous results (in continuum space) on the front of the solution to
\eqref{eq:dKPP} have been obtained in \cite{GaFr-79} (see also
  \cite{Fr-85}), \cite{No-11}, and \cite{HaNoRoRy-12}. First, under
suitable regularity and mixing assumptions, and a Heaviside type initial
condition, $w(0,\cdot)=\ind_{-\N_0}$, as in \eqref{eq:detKPP}, the
existence  of the speed of the front
\begin{equation} \label{eq:frontFKPP}
  \widehat m(t) := \sup \{ x \in \R : w(t,x) = 1/2\}
\end{equation}
of the solution to the randomized Fisher-KPP equation \eqref{eq:dKPP} is
known \cite[Theorem 7.6.1]{Fr-85}: For  $\P$-a.e.~realization
of $\xi$,
\begin{equation} \label{eq:LLNcondMed}
  \lim_{t \to \infty} t^{-1}{  \widehat m(t)} = \hat v_0,
\end{equation}
where $\hat v_0$ is non-random and corresponds to the speed of the front of
the linearized equation, which is a ``continuum space PAM''. Here, as in
Bramson's work \cite{Br-83}, a precise analysis of the Feynman-Kac
formula plays an important role in the proofs.

In the case of $\xi$ periodic instead of random, in \cite{HaNoRoRy-12}
it has been shown that there is a logarithmic correction term between
$m(t)$ and $\widehat m(t)$, and the authors were able to characterize the
constant in front of the logarithmic correction as a certain minimizer.

To the best of our knowledge, nothing is known about the fluctuations of
$\widehat m(t)$  for the Heaviside-type initial conditions in the case of
random branching rates. For a different and due technical reasons
restricted set of initial conditions, Nolen \cite{No-11} has derived a
central limit theorem for the position of the front of the solution to
\eqref{eq:dKPP} by analytic means. To put our results into context, let
us describe the assumptions of \cite{No-11} more precisely: The initial
condition $w_0(x,\xi)$ of \cite{No-11} is required to depend on the
randomness of the environment. It should satisfy
$\lim_{x \to -\infty} w_0(x,\xi) = 1$ (which roughly corresponds to our
  assumption~\eqref{eq:inCond}), and, more importantly,
\begin{equation}
  \label{eq:nolic}
  c(\xi ) \widetilde w(x,\xi , \gamma) \le w_0(x, \xi)
  \le C(\xi) \widetilde w(x,\xi, \gamma)
  \quad \text{for all } x > 0.
\end{equation}
Here $\widetilde w = \widetilde w(x,\xi, \gamma)$, $t\ge 0$, $x\ge 0$, is
a non-negative solution to the ordinary differential equation
$\frac 12\Delta \widetilde w = (\xi -\gamma )\widetilde w$ satisfying
$\widetilde w(0,\xi ,\gamma )=1$ and which decays to $0$ as $x\to\infty$.
It was known previously that $\widetilde w$ exists whenever $\gamma$ is
larger than a certain $\overbar \gamma $. In addition, there is another
$\gamma^* > \overbar \gamma $ such that whenever $\gamma \ge \gamma^*$
and the initial condition satisfies \eqref{eq:nolic}, then the law of
large numbers for the velocity of the traveling wave, that is
\eqref{eq:LLNcondMed}, holds with the same speed $\hat v_0$. In order to
prove his central limit theorem, Nolen needs to assume that
$\gamma \in (\overbar \gamma , \gamma^*)$, which leads to traveling waves
with a larger velocity $v(\gamma )> \hat v_0$. The initial conditions
corresponding to such $\gamma $ decay to $0$ exponentially as
$x\to\infty$, but the rate of decay is slow.

It is worthwhile to remark that such a distinction between the waves with
the minimal (or `critical') velocity, and the waves with strictly
larger velocity is present already in the paper of Bramson~\cite{Br-83}.
Already there it turns out that the `supercritical' is easier to
handle.

One of our main motivations for writing this paper was to understand the
behavior of the front of the traveling wave solution to randomized
Fisher-KPP equation in the `critical' case, in particular for initial
conditions of the form $w_0 = \ind_{-\mathbb N_0}$, that are, from the
point of view of the BRWRE as well as of the PAM, more natural.

\begin{theorem}
  \label{thm:FKPPCLT}
  Let $\widehat m(t)$ be the front of the solution to discrete randomized
  Fisher-KPP equation~\eqref{eq:dKPP} with initial condition
  $w_0=\ind_{-\N_0}$ defined similarly as in \eqref{eq:frontFKPP} by
  \begin{equation} \label{eq:dKPPfront}
    \widehat m(t) := \sup \{ x \in \Z : w(t,x) \ge 1/2\}.
  \end{equation}
  Then, assuming that \eqref{eq:xiAssumptions} and \eqref{eq:vAssumptions} hold true,
    ${(\widehat m(t)-v_0 t)}/{(\overbar \sigma_{v_0}\sqrt t)}$
  converges as $t\to\infty$ in $\mathbb P$-distribution to a standard
  normal random variable.
\end{theorem}

The previous theorem is a \emph{non-functional} central limit theorem
only, which might look surprising in view of our previous results. The
reason for this is the fact that the connection between the BRWRE and the
corresponding randomized Fisher-KPP equation is slightly more complicated
than in the homogeneous case, due to the fact that the BRWRE is not
translation and reflection invariant (given $\xi $): We will prove in
Proposition~\ref{prop:KPPBRW} that
\begin{equation*}
  w(t,x) = \mathtt P_x^\xi(M(t)\ge 0)
\end{equation*}
solves the randomized Fisher-KPP equation \eqref{eq:dKPP} with initial
condition $w(0,\cdot)=\ind_{\mathbb N_0}$. This should be contrasted with
the definition of $w^{\BBM}(t,x)=\mathbb P({M(t)\ge x})$ used in
\eqref{eq:detKPP}.

\paragraph{Acknowledgment} The authors would like to thank the Columbia
University Mathematics Department as well as the Mathematics Department
of the University of Vienna for their hospitality and financial support.
Part of this work has been completed during a stay at the \textit{Random
  Geometry} semester programme of the Isaac Newton Institute for
Mathematical Sciences in Cambridge, UK, in 2015. AD has been supported by
the  UoC  Forum \textit{Classical and quantum dynamics of interacting
  particle systems} of the Institutional Strategy of the University
  of Cologne within the German Excellence Initiative. The authors would like
to thank an anonymous referee as well as Lars Schmitz for carefully
reading the manuscript and for valuable comments.

\section{Strategy of the proof}
\label{sec:strategy}

We now roughly explain the strategy of the proof of our main results, and
describe the organization of the paper. As it is common in the branching
random walk literature, a first moment method is used to provide an upper
bound on the maximum of the BRWRE;  a complementary truncated second
moment computation gives a lower bound.

Luckily, similarly to the homogeneous case, the moments of the number
of particles in the BRWRE (possibly satisfying certain additional
  restrictions) have an explicit representation.  This representation, in
terms of expectations of certain functionals of simple random walk, is
called Feynman-Kac formula, `many-to-one lemma' or `many-to-few lemma',
depending on the source and context. To introduce it, for $x \in \Z$, let
$P_x$ denote the law of the continuous-time simple random walk
$(X_t)_{t \ge 0}$ on $\mathbb Z$ with jump rate $1$ and denote by $E_x$
the corresponding expectation. The following proposition, which is an
adaptation  of Section 4.2 of \cite{HaRo-17} or Theorem~2.1 of
\cite{GuKoSe-13}, gives the representation for first and second moments.
Its proof is an easy modification of the proofs of these results, and it
is therefore omitted.

\begin{proposition}[Feynman-Kac formula]
  \label{prop:FK}
  Let $\varphi_1, \varphi_2$ be càdlàg functions from $[0,t]$ to
  $ [-\infty, \infty]$ satisfying  $\varphi_1 \le \varphi_2$. Then the first and second
  moments of the number of particles in $N(t)$ whose genealogy stays
  between $\varphi_1$ and $\varphi_2$ are given by
  \begin{equation}
    \label{eq:firMomFormula}
    \begin{split}
      &\ttE_0^\xi \big [ \big| \big\{ Y \in N(t)  :
          \varphi_1(r) \le Y_r \le \varphi_2 (r) \ \forall r \in [0,t]
          \big\} \big|  \big] \\
      &\qquad =
      E_0 \bigg[\exp \Big\{ \int_0^t \xi(X_r) \, \d r \Big\};
        \varphi_1 (r)\le X_r \le \varphi_2 (r) \ \forall r \in [0,t]
        \bigg]
    \end{split}
  \end{equation}
  \begin{equation}
    \label{eq:secMomFormula}
    \begin{split}
     & \ttE_0^\xi \big [\big| \big\{ Y \in N(t) :
            \varphi_1(r) \le Y_r \le \varphi_2(r) \ \forall r \in [0,t]
            \big\} \big|^2 \big] \\
        & =
      E_0 \bigg[\exp \Big\{ \int_0^t \xi(X_r) \, \d r \Big\};
         \varphi_1(r) \le X_r \le \varphi_2(r)  \forall r \in [0,t]
        \bigg] \\
      &
      \begin{aligned}
        & + 2 \int_0^t   E_0 \bigg[
          \exp \Big\{ \int_0^s \xi(X_r) \, \d r \Big\} \xi(X_s)
          \ind_{ \varphi_1(r) \le X_r \le \varphi_2(r)
            \forall r \in [0,s] }\\
          & \times \Big( E_{X_s} \Big[
              \exp \Big\{ \int_0^{t-s} \xi(X_r) \, \d r \Big\}
              \ind_{\varphi_1(r+s) \le X_r \le \varphi_2(r+s)
                \forall r \in [0,t-s] } \Big] \Big)^2
          \bigg]\d s.
      \end{aligned}
    \end{split}
  \end{equation}
  In particular, \eqref{eq:firMomFormula} implies that
  \begin{equation}
    \label{eq:FKforN}
    \begin{split}
      \ttE_{u_0}^\xi [N^\ge(t,n)]
      &= \sum_{i\in \mathbb Z}u_0(i)
      E_i\Big[ \exp \Big\{ \int_0^t \xi(X_r) \, \d r \Big\}; X_t\ge n\Big].
    \end{split}
  \end{equation}
\end{proposition}

In the first principal step of the proof we analyze the first moment
formula~\eqref{eq:FKforN} for $n=vt$ with $v>0$. To understand this
analysis, it is useful to recall the corresponding representation from
the homogeneous case (cf.~\cite{AdRe-09}). In that setting it is almost
trivial that $\mathtt E_0^{\xi \equiv 1}[ N^{\ge}(t,vt)] = e^t P_0(X_t\ge vt)$.
The probability on the right-hand side can then be analyzed
using exact large deviation results (see e.g.~\cite{DeZe-98},
  Theorem~3.7.4) to obtain a precise asymptotic formula.

While \eqref{eq:FKforN} has a different structure, its asymptotics can
be understood, at least at a heuristic level, by the same ingredients
that are usually used in the proof of exact large deviation theorems: a
tilting and a local central limit theorem. Slightly more in detail, by
introducing a tilting parameter $\eta $, using~\eqref{eq:FKforN}, we can write
\begin{equation}
  \label{eq:heuraa}
  \mathtt E^\xi_0 [N(t,vt)]
  = e^{-t\eta }
  E_0\Big[ \exp \Big\{ \int_0^t (\xi(X_r)+\eta ) \, \d r \Big\}; X_t= vt\Big].
\end{equation}
This suggests to introduce new ``Gibbs measures'' on the space of
random walk trajectories, whose density with respect to simple random
walk is the exponential factor in \eqref{eq:heuraa}
(cf.~Section~\ref{sec:tiltedmeasures}). We then adjust $\eta$ so that the
event $X_t=vt$ is typical under such a Gibbs measure. Next, using a
suitable local central limit theorem, the right-hand side of
\eqref{eq:heuraa} can be approximated by
(cf.~Proposition~\ref{prop:Donsker})
\begin{equation*}
  \sim \frac{c}{\sqrt t}e^{-t\eta }
  E_0\Big[ \exp \Big\{ \int_0^t (\xi(X_r)+\eta ) \, \d r \Big\};H_{vt}\le
    t \Big],
\end{equation*}
where $H_x$ stands for the hitting time of $x$ by the simple random walk
$X$, see \eqref{eq:Hi} below. This can further be rewritten as
\begin{equation*}
  \sim \frac{c}{\sqrt t}e^{-t\eta }
  E_0\bigg[\prod_{x=1}^{vt}
    \exp \Big\{ \int_{H_{x-1}}^{H_x} (\xi(X_r)+\eta ) \, \d r \Big\}
    \times
    \exp \Big\{ \int_{H_{vt}}^{t} (\xi(X_r)+\eta ) \, \d r \Big\}\bigg].
\end{equation*}
If one ignores the last factor in the expectation, which can be justified
using the concentration of the hitting times of the random walk under the
Gibbs measure (see Section \ref{ssec:comparisonlemma}), then by the Markov
property
\begin{equation*}
  \begin{split}
    &\sim  \frac{c}{\sqrt t}e^{-t\eta }
    \prod_{x=1}^{vt} E_{x-1}\bigg[ \exp \Big\{ \int_{0}^{H_x}
        (\xi(X_r)+\eta ) \, \d r \Big\}\bigg]
    \\&=
    \frac{c}{\sqrt t}e^{-t\eta }
    \exp\bigg\{\sum_{x=1}^{vt} \ln E_{x-1}\bigg[ \exp \Big\{ \int_{0}^{H_x}
          (\xi(X_r)+\eta ) \, \d r \Big\}\bigg]\bigg\}.
  \end{split}
\end{equation*}
The application of a suitable  central limit theorem  to the above sum
then suggests the central limit theorem behavior of the PAM,
Theorem~\ref{thm:PAMFCLT}.

Making the above heuristics rigorous requires a non-negligible effort. In
particular, it turns out that the tilting parameter $\eta $ making the
event $X_t=vt$ typical under the Gibbs measure is \emph{random} (i.e.,
  $\xi $-dependent). This disallows a straightforward application of a
central limit theorem in the last formula above.
Section~\ref{sec:hittingFCLT} deals with this problem, building on a
preparatory Section~\ref{ssec:FCLT_LT}. Other approximations appearing in
the previous heuristic computation are treated in
Section~\ref{ssec:comparisonlemma}; Section~\ref{ssec:initialcondition}
controls the influence of the initial conditions. The functional central
limit theorem for the PAM, Theorem~\ref{thm:PAMFCLT}, then follows
easily, cf.~Section~\ref{ssec:PAMFCLT}.

In order to show the functional central limit theorem for the breakpoint,
Theorem~\ref{thm:bpFCLT}, we essentially need to find the largest root of
the function $x\mapsto \ln {\tt E}^\xi_0[N^{\ge}(t,x)]$, which requires, in a
certain sense, to invert the functional central limit theorem for the
PAM, cf.~Section~\ref{ssec:bpFCLT}. In oder to perform this inversion, we
study how sensitive $ {\tt E}^\xi_0[N^{\ge}(t,x)]$ is to perturbations in
the space and time direction, cf.~Section~\ref{ssec:perturbationestimates}.

Let us now comment on the second moment computation required to prove the
remaining main results of this paper. Similarly to the homogeneous case,
the second moment of $N^{\ge}(t,vt)$ explodes too quickly to yield any
useful estimates. This explosion is, essentially, due to particles that
are much faster than the breakpoint at  times in the bulk of the interval
$[0,t]$. In the case of homogeneous branching rates this is solved by a
truncation which involves considering only so-called \emph{leading}
particles, that is the particles that are slower than the breakpoint,
$X_s\le v_0 s$ for all $s\in [0,t]$ (here, $v_0 t$ is a first order
  breakpoint asymptotics). The principal ingredient for the computation
of the moments for the leading particles is then a `ballot theorem' for
the random walk bridge, which gives the probability that a random walk
bridge from $(0,0)$ to $(t,0)$ stays positive for all intermediate times.

Following the above strategy in the case of BRWRE suggests to call a
particle $Y\in N(t)$ \emph{leading at time $t$} if (a) $Y_t$ is close to
the breakpoint $\overbar m(t)$, and (b) $Y$ is slower than breakpoint at
intermediate times, $Y_s\le \overbar m(s)$ for all $s\in [0,t]$%
\footnote{The actual definition of leading particles in
  Section~\ref{sec:maximumproofs} is slightly different for technical
  reasons.}. Since $\overbar m(t)$ satisfies a functional central limit
theorem itself, it naturally leads to a ballot estimate of the following
form: Let $B,W$ be two independent
Brownian motions (or centered random walks, possibly not identically
  distributed). What is the behavior of
\begin{equation*}
  \mathbb P\big ( B(t)\ge W(t), B(s)\le 1+W(s)\,
    \forall s\in [0,t] \, \big| \, \sigma (W)\big)?
\end{equation*}
Observe that the process $W$ is `quenched' in this computation as
we condition on the $\sigma $-field $\sigma (W)$ generated by $W$. This
modified ballot problem was recently studied by Mallein and Miłoś
\cite{MaMi-15}. We were however not able to use their results directly
due to the lack of the independence that we encounter in our model.

The first and second moment of the number of leading particles is
computed in Section~\ref{ssec:leadingparticles}. In particular, a lengthy
proof of a (relatively weak version of) a ballot estimate can be found in
Section~\ref{ssec:leadingparticlesfirst}. Theorem~\ref{thm:logCorrection}
and thus Theorem~\ref{thm:maxFCLT} are then shown in
Section~\ref{ssec:maxproofs}. Section~\ref{sec:FKPPproofs} proves the
functional central limit theorem for the Fisher-KPP equation,
Theorem~\ref{thm:FKPPCLT}. Finally, Section~\ref{sec:openquestions}
discusses some open problems.

\paragraph{Notational conventions.} 

For two functions $f, g :[0,\infty) \to (0,\infty)$, we write $f \sim g$
when $\lim_{t \to \infty} f(t)/g(t) = 1$, and $f \asymp g$ when
$0 < \inf_{t \in [0,\infty) } f(t)/g(t) \le
\sup_{ t\in [0,\infty)} f(t)/g(t) < \infty$.
We use $c$ and $C$ to denote positive finite constants whose value may
change during computations, and sometimes write $c(\xi)$ etc.~in order
to emphasize their dependence on realizations of the branching rates.
Indexed constants
such as $c_1$ keep their value from their first time of occurrence. We
use $E[f;A]$ as an abbreviation for $E[f \ind_A]$.

For $x\in \mathbb R\setminus \mathbb Z$ we define $P_x$ by linear
interpolation. More precisely, for $x=\floor{x} + \lambda $ we define
$P_x:= (1-\lambda )P_{\floor{x}} +  \lambda  P_{\floor{x}+1}$. Similarly,
other quantities which are only defined for integers a priori are to be
interpreted as the linear interpolation of the evaluations at their
integer neighbors, which will usually be clear from the context.

While we have stated the precise assumptions needed in the main results
given above,
\begin{equation*}
  \text{{we will from now on assume  \eqref{eq:xiAssumptions},
      \eqref{eq:inCond} and \eqref{eq:vAssumptions} to be fulfilled}}
\end{equation*}
as standing assumptions without further mentioning. This helps in
keeping notation lighter compared to mentioning a suitable subset of
these assumptions at each of the numerous subsequent auxiliary results.

\section{Expected number of particles of given velocity}
\label{sec:firstmoment}

In this section we study the asymptotic behavior of
$\ttE^\xi_{u_0}[N(t,vt)]$ and related quantities, following the strategy
described in Section~\ref{sec:strategy}.

\subsection{Tilted random walk measures} 
\label{sec:tiltedmeasures}

We introduce the tilted distributions of random walk in random
potential, and show that one can tilt the random walk in a suitable way
to make the extremal behavior typical.

Recall that $(X_t)_{t\ge  0}$ denotes continuous-time simple random
walk on $\mathbb Z$ with jump rate $1$. For $i \in \Z$ we define the
hitting time of $i$ as
\begin{equation}
  \label{eq:Hi}
  H_i := \inf\{s \in [0,\infty)  : X_s = i\},
\end{equation}
and set $ \tau_i :=  H_{i} - H_{i-1}$. Recalling \eqref{eq:xiAssumptions}
and writing
\begin{equation}
  \label{eq:zetaDef}
  \zeta(x) := \xi(x) - \es, \qquad x\in \Z,
\end{equation}
we infer
\begin{equation}
  \label{eq:zetaBd}
  -\infty < \essinf \zeta < \esssup\zeta  = 0.
\end{equation}

For $n\ge 1$, $A \in \sigma ( X_{s \wedge H_n}, \, s \in [0,\infty))$
and $\eta \in \mathbb R$, we define
\begin{equation}
  \label{eq:tiltedVarTime}
  P_{(n)}^{\zeta,\eta }(A) := (Z_{(n)}^{\zeta,\eta })^{-1}
  E_0 \Big[ \exp \Big \{
      \int_0^{H_n} (\zeta(X_s)+\eta ) \, \d s \Big\};A \Big],
\end{equation}
where
\begin{equation*}
  Z_{(n)}^{\zeta,\eta} := E_0 \Big[ \exp \Big \{ \int_0^{H_n}
      (\zeta(X_s)+\eta ) \,  \d s
      \Big\}\Big].
\end{equation*}
We will see below, cf.~Lemma \ref{lem:finite}, that these quantities are
finite if and only if $\eta \le 0$.

It can be seen easily, using the strong Markov property, that
$P^{\zeta ,\eta }_{(n)}(A)=P^{\zeta ,\eta }_{(m)}(A)$ for every $m>n$ and
$A\in \sigma ( X_{s \wedge H_n}, \, s \in [0,\infty))$. We may thus use
Kolmogorov's extension theorem to extend $P_{(n)}^{\zeta,\eta }$ to a
measure $P^{\zeta,\eta }$ on $\sigma (X_s,s\ge 0)$. We write $P^\zeta $
for $P^{\zeta ,0}$.

It will be suitable to introduce the following logarithmic moment
generating functions
\begin{align}
  \label{eq:empL}
  L_i^\zeta(\eta)
  &:= \ln E_{i-1} \Big[ \exp \Big \{ \int_0^{H_{i}}
      (\zeta(X_s) +   \eta ) \, \d s \Big\} \Big],\\
        \label{eq:empLav}
  \overbar L_n^{\zeta}(\eta)
  &:= \frac{1}{n} \sum_{i=1}^n L_i^\zeta(\eta),\\
  \label{eq:Ldef}
  L(\eta) &:= \E \big[ L_1^\zeta(\eta) \big].
\end{align}
By the strong Markov property again,
\begin{equation}
  \label{eq:Zzetaeta}
  Z_{(n)}^{\zeta,\eta} = \exp\Big\{\sum_{i=1}^n L_i^\zeta (\eta )\Big\}
  =\exp\{n \overbar L_n^\zeta (\eta )\}.
\end{equation}
We now discuss the finiteness of the above objects.
\begin{lemma}
  \label{lem:finite}
  Under \eqref{eq:xiAssumptions}
  the quantities defined in
  \eqref{eq:empL}--\eqref{eq:Ldef} are finite if and only if $\eta \le 0$.
\end{lemma}

\begin{proof}
  Since $\esssup \zeta(x)\le 0$, the `if' part of the lemma is trivial.

  The `only if' part can be proved via the following strategy: For
  $i \in \Z$ and $\eta >0,$ using the independence assumption of the
  potential in \eqref{eq:xiAssumptions}, the random walk starting in $i$
  can find arbitrarily large islands to the left of $i$, where the
  potential $\zeta + \eta$ takes values larger than $\eta/2$. Once such
  an island is large enough so that the cost of the random walk to stay
  inside this island is offset by the exponential gain of a potential
  value larger than $\eta/2$ in the Feynman-Kac formula, one infers that
  $L_i^\zeta(\eta)$ is infinite, and then the same applies to the
  remaining quantities in question.

  Since in the case of random walk with a drift, the `only if' statement
  is a direct consequence of Proposition~3.1 in \cite{Dr-08}, we omit
  making the above proof rigorous. The lengthy proof of \cite{Dr-08},
  however, can directly be transferred to the case of simple random walk
  without drift.
\end{proof}

Recalling that $\tau_i=H_i-H_{i-1}$, as an easy corollary of
Lemma~\ref{lem:finite} we obtain
$\ln E^{\zeta,\eta} \big[e^{ \lambda \tau_i} \big]
= L_i^\zeta(\eta+\lambda )-L_i^\zeta (\eta )$
for every $ \eta \le 0$ and $\lambda \in \R$, as well as
\begin{equation}
  \label{eq:tauexpmoments}
  E^{\zeta ,\eta }[e^{\lambda \tau_i}]<\infty\quad \text{for every
    $\eta\le 0$ and $\lambda \le |\eta |$.}
\end{equation}
Finally, Birkhoff's ergodic theorem implies that
\begin{equation}
  \label{eq:Birkhoff}
  L(\eta)
  \overset{\P\text{-a.s.}}=
  \lim_{n \to \infty} \overbar L^\zeta_n(\eta).
\end{equation}
Other simple properties of functions $L^\zeta $ and $L$ are given in
the Appendix.

We will primarily be interested in those values
$\eta = \overbar{\eta}_n^{\zeta}(v)$ which make certain large deviations
events typical, more precisely for which
\begin{equation}
  \label{eq:baretazeta}
  {E}^{\zeta, \overbar{\eta}_n^{\zeta}(v)} [H_n]  = \frac{n}{v},
  \qquad v > 0.
\end{equation}
In order to discuss the existence of such $\eta $, which is random, we
introduce, in the next lemma, its ``typical value'' $\overbar \eta (v)$.
We recall the
\emph{critical velocity} $v_c$ introduced  below \eqref{eq:lyapunov}. It
will be shown in Proposition~\ref{prop:lyapExp}, partially using the results of
this paper, that the identity
\begin{equation}
  \label{eq:critVel}
  (v_c)^{-1} = {L'(0)}
\end{equation}
holds true, where the derivative is taken from the left only. Throughout
the paper we use \eqref{eq:critVel} as the primary definition of $v_c$.
\begin{lemma}
  \label{lem:baretaprops}
  For every $v>v_c$ there exists a unique
  $\overbar \eta(v) \in (-\infty, 0)$ such that
  \begin{equation}
    \label{eq:etaBarDef}
    L^*(1/v)
    = \sup_{\eta \in \mathbb R} (\eta /v - L (\eta) )
    = \overbar \eta(v) /v - L (\overbar \eta(v)).
  \end{equation}
  Furthermore, $\overbar \eta(v)$ is characterized by
  \begin{equation}
    \label{eq:etabar}
    L'(\overbar \eta (v))= v^{-1}.
  \end{equation}
  Moreover, $(v_c,\infty) \ni v \mapsto \overbar \eta (v)$ is a smooth
  strictly decreasing function.
\end{lemma}

\begin{proof}
  Due to Lemma~\ref{lem:lambdaProps}, $L$ is smooth, strictly increasing
  and strictly convex on $(-\infty,0)$, finite on $(-\infty,0]$,  and
  infinite on $(0,\infty)$, cf.~Lemma \ref{lem:finite}. In addition, it
  can be seen easily that $\lim_{\eta \to -\infty} L' (\eta )=0$ (see
    \cite[Lemma~3.5]{Dr-08} for the corresponding statement in the case
    of a random walk with drift; the proof for simple random walk
    proceeds in the same way and is omitted here). Therefore, recalling
  also \eqref{eq:critVel}, we see that the solution to
  \eqref{eq:etaBarDef} exists for every $v>v_c$. Furthermore, due to
  usual properties of the Legendre transform, it is characterized by
  \eqref{eq:etabar}. The last statement follows directly from the
  previously mentioned properties.
\end{proof}

We now show that $\overbar \eta^\zeta_n(v)$ fulfilling
\eqref{eq:baretazeta} exists $\mathbb P$-a.s.~for $v>v_c$ and $n$
large enough and, in fact, concentrates around $\overbar \eta (v)$.

\begin{proposition}
  \label{prop:etaExist}
  For each $v > v_c$ there exists a $\P$-a.s.~finite random variable
  $\mathcal N = \mathcal N(v)$ such that for all $n \ge \mathcal N$ there
  exists $\overbar{\eta}_n^\zeta (v) \in (-\infty,0)$ satisfying
  \eqref{eq:baretazeta}. Moreover, for every  $q \in \N$ and
  $V\subset (v_c,\infty)$ compact there exists a constant
  $C=C(q,V)<\infty$ such that for all $n\in \mathbb N$,
  \begin{equation}
    \label{eq:etaEmpEta}
    \P \bigg (  \sup_{v\in V }
     \vert \overbar \eta (v)
      - \overbar \eta_n^\zeta (v) \vert \ge
      C \sqrt{\frac{\ln n}{n}}
      \bigg)
    \le C n^{-q}
  \end{equation}
  (defining, arbitrarily, $\overbar \eta_n^\zeta (v)=0$ if the
    solution to \eqref{eq:baretazeta} does not exist).
\end{proposition}

\begin{proof}
  By Lemma~\ref{lem:lambdaProps}, for $\eta <0$,
  $E^{\zeta ,\eta }[H_n] = n (\overbar L_n^\zeta)' (\eta )$. Hence, in
  combination with \eqref{eq:baretazeta}, we may define
  $\overbar \eta_n^\zeta (v)$ as the solution to
  \begin{equation}
    \label{eq:qqqqa}
    \big(\overbar L_n^\zeta\big)' (\overbar \eta_n^\zeta(v)) = 1/v,
  \end{equation}
  if this solution exists, and by $\overbar \eta_n^\zeta (v)=0$
  otherwise. If we show that this $\overbar\eta_n^\zeta (v)$ satisfies
  \eqref{eq:etaEmpEta}, then the fact
  $\overbar \eta_n^\zeta (v)\in (-\infty,0)$ for all $n\ge \mathcal N$
  follows by a Borel-Cantelli argument, using also that
  $\sup_{v \in V} \overbar \eta(v) < 0$ by Lemma~\ref{lem:baretaprops},
  as well as the compactness of $V$.

  Comparing \eqref{eq:etabar} and \eqref{eq:qqqqa}, we see that we
  need to understand the concentration properties of
  $(\overbar L^\zeta_n)'$ first. We claim the following.

  \begin{claim}
    \label{cl:Lprimeconc}
    For every $q \in \N$ and $\Delta\subset(-\infty,0)$ compact, there exists
    $C=C(q,\Delta)<\infty$ such that for all $n\in \N$,
    \begin{equation}
      \label{eq:LprimeIneq}
      \P \bigg(\sup_{\eta \in \Delta}
        \Big\vert (\overbar L_n^\zeta)'(\eta) - L'(\eta) \Big\vert
        \ge C\sqrt{\frac{ \ln n}{n}} \bigg)
      \le  C n^{-q}.
    \end{equation}
  \end{claim}

  \begin{proof}
    We apply a Hoeffding type bound for mixing sequences which we recall in
    Lemma~\ref{lem:hoef}. Define the $\sigma$-algebras
    $\mathcal F_k:= \sigma(\xi(i)  :  i \le k )$, $k\in \mathbb Z$. By
    Lemma \ref{lem:lambdaProps},
    $\big((L_i^\zeta)'(\eta )-L'(\eta ) \big)_{i\in \mathbb Z}$ is a
    stationary sequence of bounded  random variables. By
    Lemma~\ref{lem:covBd}, there is $c<\infty$ such that
    $\big |\mathbb E\big[(L_i^\zeta )'(\eta ) \, |\,
      \mathcal F_k\big]- L'(\eta )\big|
    \le c e^{-(i-k)/c}$
    for all $i\ge k$ and $\eta \in \Delta $.
    Hence, the assumptions of Lemma~\ref{lem:hoef} are satisfied with
    $m_i=c$, and thus uniformly over $\eta \in \Delta $, for $C$ large
    enough,
   \begin{equation*}
     \P \bigg( \Big\vert (\overbar L_n^\zeta)'(\eta) - L'(\eta) \Big\vert
       \ge C\sqrt{\frac{\ln n}{n}} \bigg)
     \le C e^{ -C  \ln n} \le C n^{-q-1}.
   \end{equation*}
   Hence, by a union bound,
   \begin{equation}
     \label{eq:firstProvea}
     \P \bigg(\sup_{\eta \in \frac 1n \mathbb Z \cap \Delta}
       \Big\vert (\overbar L_n^\zeta)'(\eta) - L'(\eta) \Big\vert
       \ge C\sqrt{\frac{ \ln n}{n}} \bigg)
     \le  C n^{-q}.
   \end{equation}
   Moreover, by Lemma~\ref{lem:lambdaProps}, $L'$ and
   $(\overbar L_n^\zeta)'$ are both increasing on $(-\infty,0)$ with
   continuous and positive derivatives.  Hence, for any
   $\Delta\subset (-\infty,0)$ compact, there is $c<\infty$ such that
   \begin{equation}
     \label{eq:Lsecondbound}
     c^{-1}<\inf_\Delta L''\le \sup_\Delta L'' <c.
   \end{equation}
   Combining this with \eqref{eq:firstProvea} and the fact that $L'$ and
   $(\overbar L_n^\zeta)'$ are increasing again, this implies the claim.
  \end{proof}

  To prove \eqref{eq:etaEmpEta}, fix a compact $\Delta\subset (-\infty,0)$
  such that  $\overbar \eta (V)$ is contained in the interior of $\Delta$, which
  is possible by Lemma~\ref{lem:baretaprops}, and set
  $\delta = \dist(\overbar\eta (V),\Delta^c)>0$. By \eqref{eq:etabar} and
  \eqref{eq:qqqqa}, $\overbar \eta (v)$ and $\overbar \eta_n^\zeta(v)$
  are the respective solutions to $L'(\overbar \eta (v))=v^{-1}$ and
  $(\overbar L_n^\zeta )'(\overbar \eta_n^\zeta (v))=v^{-1}$ (if the
    solution to the second equation exists). Moreover, by
  \eqref{eq:Lsecondbound}, the slope of $L'$ on $\Delta $ is at least
  $c^{-1}$. Therefore, on the complement of the event in the probability
  on the left-hand side of \eqref{eq:LprimeIneq}, for $n$ large enough so
  that $C\sqrt{\ln(n)/n}< c^{-1} \delta $, we know that for all $v\in V$
  the equation \eqref{eq:qqqqa}
  has a solution $\overbar \eta_n^\zeta (v)$ which satisfies
  $|\overbar \eta_n^\zeta(v)-\overbar \eta (v)|\le c C\sqrt{\ln(n)/n}<\delta$.
  Hence, \eqref{eq:etaEmpEta} follows from \eqref{eq:LprimeIneq} by
  adjusting constants.
\end{proof}

For future reference we recall that whenever $\overbar \eta_n^\zeta (v)$
exists, then it is characterized, due to the usual properties of the
  Legendre transform, by
\begin{equation}
  \label{eq:etanxiDef}
  (\overbar{L}_n^\zeta)^* (1/v)
  := \sup_{\eta \in \R}
  \Big( \frac{\eta}{ v} - \overbar{L}_n^\zeta (\eta) \Big)
  = \frac{\overbar{\eta}_n^\zeta(v)}{ v}
  - \overbar{L}_n^\zeta (\overbar{\eta}_n^\zeta(v)).
\end{equation}

\paragraph{Technical assumption}
In order to keep the constants in the paper independent of the
velocity~$v$,  for the rest of this paper we assume
that
\begin{align} \label{eq:Vinterval}
  \parbox{0.85\textwidth}{the velocities $v$ that we are considering
    are contained in a fixed compact interval $V\subset(v_c,\infty)$
    which has $v_0$ in its interior.}
\end{align}
Such $V$ exists due to \eqref{eq:vAssumptions}. The constants appearing
in the results below may depend on $V$. Using
Proposition~\ref{prop:etaExist} and the monotonicity of $\overbar \eta $
and $\overbar \eta^\zeta_n$ in $v$ and $\zeta$,  it is then possible to
fix a compact interval $\Delta \subset (-\infty,0)$ such that there is a
$\mathbb P$-a.s.~finite random variable $\mathcal N_1$ such that the event
\begin{equation}
  \label{eq:Hn}
  \mathcal H_n := \mathcal H_n(V)
  := \{\overbar\eta_n^\zeta (v)\in \Delta  \text{ for all }
    v\in V\}\quad
  \text{occurs  for all $n\ge \mathcal N_1$}.
\end{equation}
We also recall that we arbitrarily
set $\overbar \eta_n^\zeta (v)=0 $ in the case when
\eqref{eq:baretazeta} does not have any solution. This occurs on
$\mathcal H_n^c$  only.

\medskip

For future use we state the following easy estimate.
\begin{lemma}
  \label{lem:etandep}
  For each $\delta \in (0,1)$ there exists a constant
  $C=C(\delta )$ such that $\mathbb P$-a.s.~for all $n$ large enough,
  uniformly for $v \in V$ and $h\le n^{1-\delta}$,
  \begin{equation*}
    \big \vert \overbar \eta^\zeta_{n} (v)
    - \overbar \eta^\zeta_{n+h}  (v) \big \vert \le
    \frac{Ch}{n}.
  \end{equation*}
\end{lemma}

\begin{proof}
  Let $\Delta $ be as in \eqref{eq:Hn}.
  We claim that there exists a constant
  $C <\infty $ such that for all $n\ge 1$,  $h\le n$ and
  $\eta \in \Delta$,
  \begin{equation}
    \label{eq:secondProve}
    \big \vert (\overbar L_{n+h}^\zeta)'( \eta) -
    (\overbar L_{n}^\zeta)'(\eta) \big \vert \le  \frac{C h}{n}.
  \end{equation}
  Indeed, plugging in the definitions we obtain
  \begin{align*}
    (\overbar L_{n+h}^\zeta)'(\eta) -
    (\overbar L_{n}^\zeta)'( \eta)
    = - \frac{h}{n(n + h)} \sum_{i=1}^{n} (L_i^\zeta)'( \eta)
    + \Big( \frac{1}{n + h} \Big) \sum_{i= n +1}^{n + h} (L_i^\zeta)'( \eta),
  \end{align*}
  from which we can then deduce \eqref{eq:secondProve} by observing that
  $(L^\zeta_i)'(\eta )$ can be bounded uniformly over $\P$-a.a.~realizations
  of $\zeta $ and $\eta \in \Delta$, by Lemma~\ref{lem:lambdaProps}. The
  claim of the lemma then follows from \eqref{eq:etabar},
  \eqref{eq:qqqqa}, \eqref{eq:Lsecondbound} and \eqref{eq:secondProve} by
  the same arguments as at the end of the proof of
  Proposition~\ref{prop:etaExist}.
\end{proof}

\subsection{An invariance principle for the empirical Legendre transforms}
\label{ssec:FCLT_LT}

In this section we show an invariance principle for the suitably centered
and rescaled Legendre transforms of the functions $\overbar L_{n}^\zeta$
defined in \eqref{eq:empLav}.
In order to state them we introduce
\begin{align}
  \label{eq:Vdef}
  & V^{\zeta,v}_i(\eta) :=  \eta / v - L_i^\zeta(\eta),\\
  \label{eq:limitSigma}
  &\sigma_v^2 :=  \Var_\P\big(V^{\zeta,v}_1(\overbar \eta (v))\big)
 + 2\sum_{j \ge 2} \Cov_\P\big(V^{\zeta,v}_1(\overbar \eta (v)),
    V^{\zeta,v}_j(\overbar \eta (v))\big).
\end{align}
Using the non-degeneracy part of assumption \eqref{eq:xiAssumptions}, and
the exponential decay of correlations of the $L^\zeta_i$ proved in
Lemma~\ref{lem:covBd}, we see that $\sigma_v^2 \in (0,\infty)$.

\begin{proposition}
  \label{prop:IEst}
  For each $v\in V$, the sequence of processes
  \begin{equation}
    \label{eq:defWnt}
    t\mapsto W_n(t) := \frac{1}{\sigma_v}
    t \sqrt {n} \big( (\overbar{L}_{nt}^\zeta)^*(1/v)
        - {L}^* (1/v) \big), \qquad n\in \N,
  \end{equation}
  converges as $n\to\infty$, in $\mathbb P$-distribution to standard
  Brownian motion.
\end{proposition}

Heuristically, the proof of this proposition is based on the fact that
the fluctuations of the  Legendre transforms $(\overbar{L}_{n}^\zeta)^*$
are essentially given by the fluctuations of the functions
$\overbar L^\zeta_n$, whereas the influence of the fluctuations of the
maximizing argument at which the supremum is attained in the definition
\eqref{eq:etanxiDef} of the Legendre transform is negligible.

\begin{proof}[Proof of Proposition \ref{prop:IEst}]
  Recall that, due to \eqref{eq:etanxiDef}, on $\mathcal H_n$,
  \begin{equation}
    \label{eq:barLS}
    (\overbar{L}_{n}^\zeta)^* (1/v)
    = \frac{  \overbar{\eta}_n^\zeta (v)}{ v}
    -\overbar{L}_n^\zeta (\overbar{\eta}_n^\zeta (v))
    = \frac 1n \sum_{i=1}^n V_i^{\zeta,v} (\overbar \eta^\zeta_n(v))
    = \frac 1n S^{\zeta,v}_n(\overbar \eta^\zeta_n(v)),
  \end{equation}
  where we set
  \begin{equation} \label{eq:SnDef}
  S_n^{\zeta,v} (\eta) := \sum_{i=1}^n V^{\zeta,v}_i(\eta)
  \end{equation}
  as a shorthand.
  Using this notation, we expand the quantity of interest as
  \begin{equation}
    \begin{split}
      \label{eq:summands}
      tn  \big((\overbar{L}_{tn}^\zeta)^* (1/v) -  {L}^* (1/v)\big)
      = {}&
      \big( tn(\overbar{L}_{tn}^\zeta)^* (1/v)
        -S_{t n}^{\zeta,v} (\overbar \eta (v))\big)
      \\&+  \big( S_{tn}^{\zeta,v} (\overbar \eta (v)) - \E [ S_{t
            n}^{\zeta,v} (\overbar \eta (v))] \big)
      \\&+  \big( \E [ S_{t n}^{\zeta,v} (\overbar \eta (v)) ]
        - tnL^*(1/v) \big).
    \end{split}
  \end{equation}
  We will show that the first and the third summand on the right-hand
  side are negligible in a suitable sense, and that the second summand
  converges in distribution after rescaling by $\sigma_v \sqrt{n}$ to
  standard Brownian motion under~$\P$.

  The third summand in \eqref{eq:summands} is the easiest since it
  vanishes. Indeed, by \eqref{eq:Ldef} and \eqref{eq:etaBarDef},
  \begin{equation*}
    tn L^*(1/v)
    =t n \Big(\frac{\overbar \eta (v)}{ v} - L(\overbar \eta(v) ) \Big)
    = t n\E\Big[ \frac{\overbar \eta (v)}{ v}
      - \overbar L_{t n}^\zeta(\overbar \eta(v))\Big]
    = \E \big[ S_{tn}^{\zeta,v} (\overbar \eta (v))\big].
  \end{equation*}

  The next lemma deals with the second summand in \eqref{eq:summands}.
  \begin{lemma} \label{lem:secCLT}
    The sequence of processes
    \begin{equation*}
     [0,\infty) \ni t\mapsto \widetilde W_n(t):=\frac 1{\sigma_v \sqrt{n}}
      \big( S_{t  n}^{\zeta,v} (\overbar \eta (v))
        - \E [ S_{t n}^{\zeta,v} (\overbar \eta (v)) ] \big),
      \qquad n \in \N,
    \end{equation*}
    converges as $n\to\infty$ in  $\mathbb P$-distribution to standard
    Brownian motion.
  \end{lemma}

  \begin{proof}
    By the definition of $S_n^{\zeta,v}$,
    \begin{equation*}
      \frac 1{\sigma_v\sqrt n}\big(S_{tn}^{\zeta,v} (\overbar \eta (v))
        - \E [ S_{t  n}^{\zeta,v} (\overbar \eta (v)) ]\big)
      = \frac{1}{\sigma_v\sqrt n}
      \sum_{i=1}^{t n} V_i^{\zeta ,v}(\overbar \eta (v))
      - \E[V_i^{\zeta ,v}(\overbar \eta (v))].
    \end{equation*}
    The $V^{\zeta,v}_i(\overbar \eta (v))$  form a non-degenerate
    stationary sequence of random variables, which are coordinatewise
    decreasing in the $\zeta$'s. Therefore, by the FKG-inequality, they
    also form an associated sequence in the sense that any two
    coordinatewise decreasing functions of the
    $V^{\zeta,v}_i(\overbar \eta (v))$'s of finite variance are
    non-negatively correlated. Hence, the functional central limit
    theorem for associated random variables proved in
    \cite[Theorem~3]{NeWr-81} supplies us with convergence in $C([0,M])$
    for each $M \in (0,\infty)$, and the result is then extended to
    $C([0,\infty))$ in the standard fashion.
  \end{proof}
  Finally, for the first summand in \eqref{eq:summands}, we have the
  following estimate.
  \begin{lemma}
    \label{lem:concentration}
    There is $C<\infty$ such that $\mathbb P$-a.s.~for every
    $M\in (1,\infty)$ and $v\in V$,
    \begin{equation*}
      \limsup_{n\to\infty} \frac 1 {\ln n}
      \sup_{t \in [0,M]} \big|
      tn(\overbar{L}_{tn}^\zeta)^* (1/v)
      - S_{t  n}^{\zeta,v} (\overbar \eta (v)) \big |  \le C.
    \end{equation*}
  \end{lemma}

  \begin{proof}
    By Proposition~\ref{prop:etaExist} and \eqref{eq:Hn},  the
    representation \eqref{eq:barLS} holds for all $n\ge \mathcal N_1$,
    with $\mathcal N_1$ a $\P$-a.s.~finite random variable.  As a
    consequence, it is sufficient to show that $\P$-a.s.,
    \begin{equation}
      \label{eq:contoshow}
      \limsup_{n\to\infty}
      \frac 1 {\ln n}\max_{\mathcal N_1 \le k \le Mn}
        \Big \vert S_{k}^{\zeta,v}(\overbar \eta_{k}^\zeta(v))
        - S_{k}^{\zeta,v} (\overbar \eta(v)) \Big \vert
        \le C.
    \end{equation}
    Assuming $k \ge \mathcal N_1$ in what follows, using a Taylor
    expansion of the smooth function $S^{\zeta ,v}_k$ around
    $ \overbar \eta_{k}^\zeta (v)$ we get
    \begin{equation}
      \label{eq:Texpand}
      \begin{split}
        S_{k}^{\zeta,v} (\overbar \eta (v))
        - S_{k}^{\zeta,v} (\overbar \eta_{k}^\zeta (v))
        &= (S_{k}^{\zeta,v})' (\overbar \eta_{k}^\zeta (v))
        (\overbar \eta (v) - \overbar \eta_{k}^\zeta (v))
        \\&+ (S_{k}^{\zeta,v})'' (\widetilde \eta_{k}^\zeta )
        \frac{(\overbar \eta (v) - \overbar \eta_{k}^\zeta (v))^2}{2},
      \end{split}
    \end{equation}
    for some $\widetilde \eta_{k}^\zeta \in \Delta$ with
    $\vert \widetilde \eta_{k}^\zeta - \overbar \eta_{k}^\zeta  (v)\vert
    \le \vert \overbar \eta (v) - \overbar \eta_{k}^\zeta (v) \vert $.

    By \eqref{eq:etanxiDef}, $S_k^{\zeta,v} (\eta )$ is maximized for
    $\eta = \overbar \eta^\zeta_k(v)$, so
    $(S_{k}^{\zeta,v})' (\overbar \eta_{k}^\zeta (v)) = 0$ and the first
    term on the right-hand side of \eqref{eq:Texpand} vanishes.

    To bound the second term, observe that
    $(S_{k}^{\zeta,v})''(\widetilde \eta_{k}^\zeta)
    = -k (\overbar L_k^\zeta)''(\widetilde \eta_k^\zeta)$.
    By Lemma \ref{lem:lambdaProps}, $\P$-a.s.,
    $(L_1^\zeta)''(\eta )$ is bounded from above, uniformly over
    $\eta \in \Delta $ (cf.~\eqref{eq:Hn}). Hence, $\P$-a.s.,
    \begin{equation}
      \label{eq:Ssecondbound}
      (S_k^{\zeta ,v})''(\widetilde \eta_k^\zeta )
      \in[- Ck,0]\qquad
      \text{for all $k\ge \mathcal N_1$, $v\in V$}.
    \end{equation}
    Going back to
    \eqref{eq:Texpand}, $\P$-a.s.~for all $k\ge \mathcal N_1$,
    \begin{equation*}
      \big\vert  S_k^{\zeta,v} (\overbar \eta (v))
      - S_k^{\zeta,v} (\overbar \eta_k^\zeta (v)) \big\vert
      \le c k
      \big\vert \overbar \eta (v) - \overbar \eta_k^\zeta (v) \big\vert^2.
    \end{equation*}
    Using the concentration estimates for $\overbar\eta_k^\zeta (v)$
    from Proposition~\ref{prop:etaExist}, it is
    possible to fix a constant $C<\infty$ and a $\P$-a.s.~finite
    random variable $\mathcal N_2\ge\mathcal N_1$ such that for all
    $k\ge \mathcal N_2$,
    $\big |\overbar \eta_k^\zeta(v) - \overbar \eta(v)\big|
    \le C \sqrt {{\ln k}/k}$.
    Putting all together, this implies that $\P$-a.s.~the left-hand side
    in \eqref{eq:contoshow} is bounded by
    \begin{equation*}
      \begin{split}
        \limsup_{n\to\infty}
        \frac 1{\ln n}\Big\{
        \max_{\mathcal N_1 \le k \le \mathcal N_2}
          \vert S_{k}^{\zeta,v}(\overbar \eta_{k}^\zeta(v))
          - S_{k}^{\zeta,v} (\overbar \eta(v)) \vert
          +\max_{\mathcal N_2 \le k \le Mn}
          C\ln k\Big\}\le C.
      \end{split}
    \end{equation*}
    This completes the proof.
  \end{proof}
  Proposition~\ref{prop:IEst} now follows from \eqref{eq:summands} and
  Lemmas \ref{lem:secCLT} and \ref{lem:concentration}.
\end{proof}

The proof of Proposition~\ref{prop:IEst} has the following  corollary
which provides a useful explicit approximation to $W_n(t)$.
\begin{corollary}
  \label{cl:explicitW}
  There is a constant $C<\infty$ such that $\P$-a.s.~for every
  $M\in (0,\infty)$ and $v\in V$,
  \begin{equation*}
    \limsup_{n\to\infty}\frac 1{\ln n}\sup_{t \in [0, M]} \Big|
      \sigma_v\sqrt n W_n(t) -  \sum_{i=1}^{nt}
      \big(L(\overbar \eta (v))-L_i^\zeta (\overbar \eta
          (v))\big)\Big|\le C.
  \end{equation*}
\end{corollary}

\begin{proof}
  It suffices to use the definition \eqref{eq:defWnt} of $W_n(t)$ together
  with Lemma~\ref{lem:concentration}. The claim then follows after a
  straightforward computation  by inserting
  the definition of $S_{tn}^{\zeta,v}(\overbar \eta (v))$ and using that
  $L^*(1/v)=\overbar \eta (v)/v - L(\overbar \eta (v))$.
\end{proof}

\subsection{An auxiliary invariance principle}
\label{sec:hittingFCLT}

We now prove an invariance principle for the logarithm of
the auxiliary process
\begin{equation*}
  Y_{v}(n):=  E_{0} \Big [ \exp \Big
    \{ \int_0^{H_n} \zeta(X_s) \, \d s \Big\}; H_n \le \frac{n}{v}
    \Big], \qquad n\in\N,v\in V,
\end{equation*}
which we will relate to quantities considered in the Feynman-Kac
representation~\eqref{eq:firMomFormula} later on. Observe that this
invariance principle can be seen as a first step to exact large
deviation estimates, as explained in Section~\ref{sec:strategy} above.

For convenience we split the process $Y_v$ into the two summands
\begin{align}
  \label{eq:Ysummands}
  \begin{split}
    Y^\approx_{v}(n) &:= E_{0} \Big [ \exp \Big
      \{ \int_0^{H_n} \zeta(X_s) \, \d s \Big\}; H_n\in\Big[\frac
        {n}{v}-K,\frac{n}{v}\Big] \Big] \quad \text{ and }
    \\ Y^<_{v}(n) &:=  E_{0} \Big [ \exp \Big
      \{ \int_0^{H_n} \zeta(X_s) \, \d s \Big\}; H_n<\frac {n}{v}-K \Big],
  \end{split}
\end{align}
where $K>0$ is a large constant which will be fixed later on.

For $n\in \mathbb N$ and $v\in V$ we define random variables
$\sigma_n^\zeta(v)$
\begin{equation}
  \label{eq:sigmaDef}
  \sigma_{n}^\zeta(v) :=
  \begin{cases}
    |\overbar \eta_{n}^\zeta(v)|
    \sqrt{ \vphantom{\sum}\Var_{P^{\zeta, \overbar \eta_{n}^\zeta(v)}}
      [H_n] }, \qquad&
    \text{on } \mathcal H_n,\\
    \max \Delta
    \sqrt{ \vphantom{\sum}\Var_{P^{\zeta, \max \Delta  }}
      [H_n] }, \qquad&
    \text{on } \mathcal H_n^c.
  \end{cases}
\end{equation}
Under every $P^{\zeta,\eta}$ we can write $H_n=\sum_{i=1}^{n}{\tau_i}$ as
a sum of independent random variables (see \eqref{eq:Hi} and below).
Moreover, by Lemma~\ref{lem:lambdaProps}, there is a constant $c<\infty$
such that $c^{-1}\le \Var_{P^{\zeta,\eta }}[\tau_i]\le c$ for all
$n \in \mathbb N$, $\eta\in \Delta$ and $\P$-a.e.~$\zeta$, and thus
\begin{equation}
  \label{eq:sigmabound}
  c^{-1}\sqrt n\le \sigma_n^\zeta (v) \le c \sqrt n \qquad
  \text{for all $n\in \mathbb N$, $v\in V$  and $\P$-a.e.~$\zeta$}.
\end{equation}

\begin{proposition}
  \label{prop:Donsker}
  Let $V$ be as in \eqref{eq:Vinterval},  and let $K$ from
  \eqref{eq:Ysummands} be a large enough fixed constant. Then there
  exists a constant $C < \infty$ such that
  \begin{equation}
    \label{eq:expAsymptotics}
      Y^\approx_{v}(n) \sigma_{n}^\zeta(v) \exp \big\{ n {L}^* (1/v) +
        \sigma_v \sqrt n W_n(1)   \big\}
      \in [C^{-1}, C]
  \end{equation}
  for all $v \in V,  n \in \N $ on $\mathcal H_n$,
  where $W_n$ is given in \eqref{eq:defWnt} of Proposition
  \ref{prop:IEst} and  $\sigma_v \in (0,\infty)$ is as in
  \eqref{eq:limitSigma}. In addition, for some $\widetilde C<\infty$,
  \begin{equation}
    \label{eq:comparability}
    \frac{Y_{v}^\approx(n)}{Y_{v}^<(n)}
    \in [\widetilde C^{-1}, \widetilde C]
    \quad \text{for all } v \in V, n\in \mathbb N,
    \text{ on } \mathcal H_n.
  \end{equation}
  In particular, each of the three sequences of processes
  \begin{align} \label{eq:3seq}
    \begin{split}
      t \mapsto  \frac{1}{\sigma_v \sqrt{n}} \big( \ln Y_v^{\approx}(tn)
        + t  n L^* (1/v)  \big), \quad n \in \N,\\
      t \mapsto  \frac{1}{\sigma_v \sqrt{n}} \big( \ln Y_v^{<}(tn)
        + t  n L^* (1/v)  \big), \quad n\in \N,\\
      t \mapsto  \frac{1}{\sigma_v \sqrt{n}} \big( \ln Y_v(tn)
      + t  n L^* (1/v)  \big), \quad n\in \N,
    \end{split}
  \end{align}
  converges as $n\to\infty$ in $\mathbb P$-distribution to standard
  Brownian motion.
\end{proposition}

\begin{proof}
  Throughout the proof we assume that $n$ is large enough so that
  $\mathcal H_n$ occurs. To simplify the notation, we also omit the
  dependence of $\overbar\eta^\zeta_n$ and $\sigma^\zeta_n$ on the
  parameter $v$.

  Let
  $\widehat \tau_i := \tau_i - E^{\zeta, \overbar \eta_{n}^\zeta} [\tau_i]$.
  Using the definition of the tilted measure $P^{\zeta ,\eta }$ (see
    \eqref{eq:tiltedVarTime} and below) with \eqref{eq:Zzetaeta},  and
  the fact that
  $\sum_{i=1}^n E^{\zeta, \overbar \eta_{n}^\zeta} [\tau_i]
  =E^{\zeta, \overbar \eta_{n}^\zeta} [H_n]=n/v$,
  we can rewrite $Y^\approx_v(n)$ as
  \begin{equation}
    \label{eq:Ysimdecomp}
    \begin{split}
      Y^\approx_v(n)
      & = E^{\zeta, \overbar \eta_{n}^\zeta}
      \bigg[\exp \Big\{ - \overbar \eta_{n}^\zeta
          \sum_{i=1}^{n}\widehat \tau_i\Big\}
        ;\sum_{i=1}^{n} \tau_i \in \Big[\frac nv-K,\frac nv \Big]
        \bigg]
      e^ { - n \big( v^{-1}\overbar \eta_{n}^\zeta
          - \overbar{L}_{n}^\zeta (\overbar \eta_{n}^\zeta)\big) }
      \\
      & = E^{\zeta, \overbar \eta_{n}^\zeta}
      \bigg[
        \exp \Big\{ - \sigma_{n}^\zeta
          \frac{\overbar \eta_{n}^\zeta}{\sigma_{n}^\zeta}
          \sum_{i=1}^{n}  \widehat \tau_i \Big\} ;
        \frac{\overbar \eta_{n}^\zeta}{\sigma_{n}^\zeta}
        \sum_{i=1}^{n} \widehat \tau_i \in
        \Big[0, -\frac{K \overbar \eta_{n}^\zeta }{\sigma_{n}^\zeta} \Big]
        \bigg]
      e^{ - n (\overbar{L}_{n}^\zeta)^* (1/v) }.
    \end{split}
  \end{equation}
  Writing $\mu_{n}^\zeta$ for the distribution of
  $\frac{\overbar{\eta}_{n}^\zeta}{\sigma_{n }^\zeta}
  \sum_{i=1}^{n} \widehat \tau_i$
  under $P^{\zeta, \overbar \eta_{n}^\zeta}$ (depending implicitly on $v$),
  we obtain
  \begin{align}
    \label{eq:ProdBEupperBd}
    Y_v^\approx(n)
    &= e^{ - n (\overbar{L}_{n }^\zeta)^* (1/v) }
    \int_{ 0}^{-K\overbar \eta_{n}^\zeta /\sigma_{n}^\zeta}
    e^{-\sigma_{n}^\zeta x}\, \d  \mu_{n}^\zeta( x),
  \end{align}
  and, in a similar vein,
  \begin{align} \label{eq:ll}
    Y_v^<(n)
    &= e^{ - n (\overbar{L}_{n }^\zeta)^* (1/v) }
     \int_{-K\overbar \eta_{n}^\zeta /\sigma_{n}^\zeta}^\infty
    e^{-\sigma_{n}^\zeta x}\, \d \mu_{n}^\zeta(x).
  \end{align}
  The first factor in \eqref{eq:ProdBEupperBd}, \eqref{eq:ll} can be
  controlled by Proposition~\ref{prop:IEst} and
  Corollary~\ref{cl:explicitW}. The following lemma gives estimates for
  the second factors.

  \begin{lemma}
    \label{lem:IIEst}
    Let $V$ and $K$ be as in Proposition~\ref{prop:Donsker}. Then there
    exists $C\in (1,\infty)$ such that on $\mathcal H_n$, for all $v\in
    V$,
    \begin{equation}
      \label{eq:theta1First}
      \sigma_{n}^\zeta
      \int_{ 0}^{-K\overbar \eta_{n}^\zeta /\sigma_{n}^\zeta}
      e^{-\sigma_{n}^\zeta x}\, \d \mu_{n}^\zeta(x)
      \in [C^{-1},C],
    \end{equation}
    \begin{equation}
      \label{eq:theta1Second}
      \sigma_{n}^\zeta
      \int_{-K\overbar \eta_{n}^\zeta /\sigma_{n}^\zeta}^\infty
      e^{-\sigma_{n}^\zeta x}\, \d \mu_{n}^\zeta(x)
      \in [C^{-1},C].
    \end{equation}
  \end{lemma}
  In order not to hinder the flow of reading, we finish the proof of
  Proposition~\ref{prop:Donsker} first. Using Lemma~\ref{lem:IIEst},
  \eqref{eq:ProdBEupperBd}, and recalling the definition
  \eqref{eq:defWnt} of $W_n$ directly yields \eqref{eq:expAsymptotics}.
  From \eqref{eq:ll}, \eqref{eq:ProdBEupperBd}, and Lemma~\ref{lem:IIEst}
  we deduce \eqref{eq:comparability}. Finally, replacing $n$ by $nt$ in
  \eqref{eq:expAsymptotics}, observing that $\sqrt t W_{nt}(1)=W_n(t)$,
  and using \eqref{eq:comparability}, the fact that $\mathcal H_n$ occurs
  $\mathbb P$-a.s.~for $n$ large, in combination with and
  Proposition~\ref{prop:IEst}, yields the convergence of the three
  sequences in \eqref{eq:3seq} to standard Brownian motion.
\end{proof}

We now show Lemma~\ref{lem:IIEst} which was used in the last proof.

\begin{proof} [Proof of Lemma \ref{lem:IIEst}]
  We start with proving \eqref{eq:theta1First}. Throughout the proof we
  assume that $\mathcal H_n$ occurs. Observe that the  $\widehat \tau_i$,
  $1 \le i \le n$, are independent under
  $P^{\zeta, \overbar \eta_{n}^\zeta}$ and have small exponential
  moments uniformly in $n$ (cf.~\eqref{eq:tauexpmoments}).
  Moreover, recalling \eqref{eq:sigmaDef}, the variance of the
  distribution $\mu^\zeta_n$ is one by definition. A local central limit
  theorem for such independent normalized sequences, Theorem~13.3 (or
    formula~(13.43)) of \cite{BhRa-76}, thus yields
  \begin{equation}
    \label{eq:BEstein}
    \sup_{A} \vert \mu^\zeta_n(A) - \Phi(A) \vert \le
    C n^{-1/2},
  \end{equation}
  where the supremum runs over all intervals in $\mathbb R$, and $\Phi$
  denotes the standard Gaussian measure. Applying \eqref{eq:BEstein} to
  $A=[0,-K\overbar \eta^\zeta_n/\sigma_n^\zeta ]$ and bearing in mind
  \eqref{eq:sigmabound}, this implies that for all $K$ large enough,
  uniformly in $v\in V$,
  \begin{equation*}
    c^{-1}n^{-1/2}
    <\mu ^\zeta_n([0,-K\overbar \eta_n^\zeta/\sigma_n^\zeta])
    <cn^{-1/2}.
  \end{equation*}
  Since the function $e^{-\sigma_n^\zeta x }$ is uniformly bounded from
  above and below in this interval, \eqref{eq:theta1First} follows by
  another application of \eqref{eq:sigmabound}.

  In order to show \eqref{eq:theta1Second}, we observe that uniformly in
  $v\in V$,
  \begin{equation*}
    \sigma_{n}^\zeta
    \int_{-K\overbar \eta_{n}^\zeta /\sigma_{n}^\zeta}^\infty
    e^{-\sigma_{n}^\zeta x}\, \d \mu_{n}^\zeta(x)
    \ge
    \sigma_{n}^\zeta
    \int_{-K\overbar \eta_{n}^\zeta /\sigma_{n}^\zeta}
    ^{-2K\overbar \eta_{n}^\zeta /\sigma_{n}^\zeta}
    e^{-\sigma_{n}^\zeta x}\, \d \mu_{n}^\zeta(x) \ge C^{-1}
  \end{equation*}
  by the same arguments as in the proof of \eqref{eq:theta1First}. On the
  other hand, using \eqref{eq:sigmabound} and \eqref{eq:BEstein} again,
  writing $I_j= [-jK\overbar \eta_{n}^\zeta /\sigma_{n}^\zeta,
    -(j+1)K\overbar \eta_{n}^\zeta /\sigma_{n}^\zeta]$,
  \begin{equation}
    \begin{split}
      \label{eq:Ytailest}
      \sigma_{n}^\zeta
      \int_{-K\overbar \eta_{n}^\zeta /\sigma_{n}^\zeta}^\infty
      e^{-\sigma_{n}^\zeta x}\, \d \mu_{n}^\zeta(x)
      &\le
      \sigma_{n}^\zeta
      \sum_{j=1}^\infty
      \mu_n^\zeta (I_j)
      e^{-jK |\overbar\eta^\zeta_n|}
      \\&\le
      c \sigma_{n}^\zeta
      \sum_{j=1}^\infty n^{-1/2}
      e^{-jK |\overbar\eta^\zeta_n|}
      \le C,
    \end{split}
  \end{equation}
  uniformly in $v\in V$. This completes the proof of the lemma.
\end{proof}

\begin{remark}
  The arguments of the last proof can be used to show that for arbitrary
  $a\in [0,n/v]$, $n\in \mathbb N$, $v\in V$ on $\mathcal H_n$,
  \begin{equation}
    \label{eq:atail}
    \frac 1{Y_{v}^\approx(n)}
    \,E_{0} \Big[e^{ \int_0^{H_{n}} \zeta(X_s) \, \d s };
      H_{n}\le \frac nv -a\Big]
    \le C e^{-c a} .
  \end{equation}
  Indeed, the expectation on the left-hand side of \eqref{eq:atail} can
  be written as in \eqref{eq:ll} with $K$ replaced by $a$. Hence, with
  help of  \eqref{eq:theta1First}, the left-hand side of \eqref{eq:atail}
  is bounded by the left-hand side of \eqref{eq:Ytailest} with $K$
  replaced by $a$. Recalling the last-but-one expression in
  \eqref{eq:Ytailest}, inequality \eqref{eq:atail} easily follows.
\end{remark}

\begin{remark}
  The proof of Proposition~\ref{prop:Donsker} is the only occasion where
  the random tilting by $\overbar \eta^\zeta_n(v)$ is really necessary.
  The reason for this is the application of \eqref{eq:BEstein}, the local
  central limit theorem in spirit, which is useful only for events of
  sufficiently large probability. Deterministic tilting by
  $\overbar \eta (v)$, which would simplify the remaining parts of the
  paper, unfortunately requires dealing with events of much smaller
  probability.
\end{remark}

\subsection{The walk lingers in the bulk} 
\label{ssec:comparisonlemma}

We now show that the invariance principles of
Proposition~\ref{prop:Donsker} are useful in order to analyze the
Feynman-Kac representation \eqref{eq:FKforN} of
$\ttE^\xi_{u_0} [N^\ge (t,vt)]$. We explore the fact that, under the
considered distributions, conditioning on $X_{n/v}=n$ (as in the
  Feynman-Kac representation) implies that with high probability $H_n$ is
close to $n/v$, that is the `walk lingers in the bulk'.

\begin{lemma}
  \label{lem:fixedVsVariableTime}
  Let $K>0$ and $V$ be as in Proposition~\ref{prop:Donsker}. Then
  there exists a constant $c>0$ such that for all $n\in\N$ and $v\in V$,
  on $\mathcal H_n$,
  \begin{equation} \label{eq:fixedVarTime}
    \begin{split}
      c Y^\approx_v(n)
      &\le E_{0} \Big[ \exp \Big \{ \int_0^{n/v} \zeta(X_s) \, \d s \Big\}
        ;X_{n/v}=n \Big]
      \\&\le E_{0} \Big[ \exp \Big \{ \int_0^{n/v} \zeta(X_s) \, \d s \Big\}
        ;X_{n/v}\ge n \Big]
      \le c^{-1} Y_v^\approx(n).
    \end{split}
  \end{equation}
  In particular,
  \begin{equation}
    \label{eq:NYcomp}
    c e^{\es\, n/v} Y^\approx_v(n)
    \le  \ttE_0^\xi \Big[N\Big(\frac nv,n\Big)\Big]
    \le \ttE_0^\xi \Big[ N^{\ge}\Big(\frac nv,n\Big)\Big]
    \le c^{-1} e^{\es\, n/v} Y^\approx_v(n).
  \end{equation}
\end{lemma}

\begin{proof}
  The second claim of the lemma follows directly from the first one; it
  suffices to recall $\xi(x)=\zeta(x)+\es$ and
  \eqref{eq:FKforN}.

  To prove the first claim we define
  \begin{equation*}
    p_n^\zeta(s) :=   E_n
    \Big[ \exp \Big\{ \int_0^s \zeta(X_r) \, \d r \Big\};
      X_s = n \Big], \qquad n\in \mathbb Z, s\ge 0,
  \end{equation*}
  and set $t=n/v$, to simplify notation. Using the strong Markov
  property,
  \begin{equation*}
    \begin{split}
      E_{0} \Big[ &\exp \Big \{ \int_0^{t} \zeta(X_s) \, \d s \Big\};
        X_{t}=n \Big]
      \\&= E_{0} \Big[ \exp \Big \{ \int_0^{H_n} \zeta(X_s) \, \d s \Big\}
         p_n^\zeta(t-H_n); H_n\le t \Big]
      \\&\ge
      E_{0} \Big[ \exp \Big \{ \int_0^{H_n} \zeta(X_s) \, \d s \Big\}
        ;  H_n\in[ t-K,t] \Big] \inf_{s\le K} p_n^\zeta(s).
    \end{split}
  \end{equation*}
  Since the $\zeta(x)$'s are bounded from below by assumption
  \eqref{eq:xiAssumptions}, the infimum on the right-hand side can be
  bounded from below by a deterministic constant $c = c(K)>0$, implying
  the first inequality in \eqref{eq:fixedVarTime}.

  The second inequality of \eqref{eq:fixedVarTime} is obvious. For
  the third one, observe that
  $\{X_t \ge n\}\subset\{H_n\le t\}$. Therefore, decomposing the integral
  according to the value of $H_n$ and using the fact that $\zeta \le 0$,
  we obtain
  \begin{align*}
    E_{0}\Big[&\exp\Big\{ \int_0^{t} \zeta(X_s) \, \d s \Big\} ; X_t \ge n\Big]
    \\&=
    E_{0}\Big[\exp\Big\{ \int_0^{H_n} \zeta(X_s) \, \d s \Big\}\exp\Big\{
        \int_{H_n}^t \zeta(X_s) \, \d s \Big\}  ; X_t \ge n\Big]\\
    &\le
   E_{0} \Big[\exp\Big\{ \int_0^{H_n} \zeta(X_s) \, \d s \Big\} ; H_n\le
     t\Big] = Y_v(n).
  \end{align*}
  By Proposition~\ref{prop:Donsker}, $Y_v(n)$ and $Y^\approx_v(n)$ are
  comparable on $\mathcal H_n$, which proves the third inequality.
\end{proof}

\subsection{Initial condition stability} 
\label{ssec:initialcondition}

The next lemma shows that initial conditions $u_0$ satisfying
assumption \eqref{eq:inCond} are comparable to the `one-particle'
initial condition $u_0=\ind_{\{0\}}$.

\begin{lemma}
  \label{lem:inCond}
  Let $V$ be as in \eqref{eq:Vinterval}. There exists a finite constant
  $C$ such that for all $u_0$ as in \eqref{eq:inCond} and for all
  $n\in \mathbb N$, $t\ge 0$ such that $n/t\in V$, on $\mathcal H_n$,
  \begin{equation}
    \label{eq:sameOrder}
    1 \le \frac{\ttE^\xi_{u_0} \big[  N^{\ge}(t,n) \big] }
    {\ttE^\xi_{0} \big[  N(t,n) \big] }
    \le  C.
  \end{equation}
\end{lemma}

\begin{proof}
  The first inequality in \eqref{eq:sameOrder} is obvious, so we proceed
  to the second one. Moreover, since
  $\ttE^\xi_{u_0} \big[  N^{\ge}(t,n) \big]$ is an increasing function of
  $u_0(x)$ for every $x\in - \mathbb N_0$, we can assume that
  $u_0= c \ind_{-\N_0}$. Using the Feynman-Kac representation
  \eqref{eq:FKforN}, and replacing $\xi $ by $\zeta $, we see that
  \begin{equation}
    \label{eq:incondaa}
    \frac{\ttE^\xi_{c \ind_{-\N_0}} \big[  N^{\ge}(t,n) \big] }
    {\ttE^\xi_{0} \big[  N(t,n) \big] }
    = \frac{ c \sum_{x \le 0}
      E_x \Big[\exp \Big\{ \int_0^t \zeta(X_s) \, \d s \Big\};
        X_t \ge n \Big]}
    {E_{0} \Big[\exp \Big\{ \int_0^t \zeta(X_s) \, \d s \Big\};
        X_t = n \Big]}.
  \end{equation}
  Applying the strong Markov property on the numerator of the right-hand
  side, we obtain
  \begin{equation*}
    \begin{split}
      \sum_{x \le 0}
      E_x & \Big[e^{ \int_0^t \zeta(X_s) \, \d s };
        X_t \ge n \Big]
      \le
      \sum_{x \le 0}
      E_x \Big[e^{ \int_0^{H_n} \zeta(X_s) \, \d s };
        H_n \le t \Big]
      \\ &=
      \sum_{x \le 0}
      \int_{0}^t
      E_x \Big[e^{ \int_0^{H_{0}} \zeta(X_s) \, \d s };
        H_{0}\in \d a\Big]
      E_{0} \Big[e^{ \int_0^{H_{n}} \zeta(X_s) \, \d s };
        H_{n}\le t-a\Big].
    \end{split}
  \end{equation*}
  By \eqref{eq:atail}, on $\mathcal H_n$, the second factor
  on the right-hand side can be bounded from above by
  $ C e^{-c a}  Y_{n/t}^\approx(n)$ for all $n$ with $n/t\in V$ and
  $a\in [0,t]$. This implies that the right-hand side of the last display
  is bounded from above by
  \begin{equation*}
    CY^\approx_{n/t}(n)\sum_{x\le 0} \int_0^t P_x(H_{0}\in \d a) e^{-ca}
    \le C Y^\approx_{n/t}(n),
  \end{equation*}
  where for the last inequality we used that due to the stationarity of
  simple random walk we have
  $\sum_{x\le 0}P_x(H_{0}\in \d a) = \sum_{x\ge 0} P_0(H_x\in \d a)$,
  and the latter
  is the probability that an arbitrary point $x\ge 0$ is visited
  for the first time at time $\d a$, so it is bounded by $\d a$. By
  Lemma~\ref{lem:fixedVsVariableTime}, on $\mathcal H_n$,
  $Y^\approx_{n/t}(n)$ is comparable to the denominator of the
  right-hand side in \eqref{eq:incondaa}, which completes the proof.
\end{proof}

\subsection{Proof of Theorem~\ref{thm:PAMFCLT}
  (functional CLT for the PAM)} 
\label{ssec:PAMFCLT}
We have all ingredients to show our first main result, the invariance
principle for the PAM, Theorem~\ref{thm:PAMFCLT}.

\begin{proof}[Proof of Theorem~\ref{thm:PAMFCLT}]
  Recall that we have to show that the sequence of processes
  $\big(\ln u (nt,\floor{vnt}) - nt \lambda (v)\big)/(\sigma_v\sqrt{v n})$,
  with $u(t,x)= \ttE^\xi_{u_0}[N(t,x)]$ as in~\eqref{eq:uxiN}, satisfies
  the functional central limit theorem under $\P$ as $n\to\infty$.

  By Lemma~\ref{lem:inCond} we can assume without loss of generality that
  $u_0=\ind_{\{0\}}$. Moreover, by Lemma~\ref{lem:fixedVsVariableTime},
  $\mathbb P$-a.s.~for all large $t$,
  \begin{equation}
    \label{eq:PAMbd}
    c Y^\approx_v(vt) e^{t \es}
    \le u(t,\floor{vt})
    \le c^{-1} Y^\approx_v(vt) e^{t \es}.
  \end{equation}
  Replacing $t$ by $vt$ in Proposition~\ref{prop:Donsker}, we see that
  \begin{equation*}
    t \mapsto
    \frac 1{\sigma_v \sqrt {nv}} \big(\ln Y^\approx_v(tvn)+tvn L^*(1/v)\big)
  \end{equation*}
  converges as $n\to\infty$ to  standard Brownian motion. Combining this
  with \eqref{eq:PAMbd} easily implies the theorem; incidentally, it also
  shows that the Lyapunov exponent $\lambda (v)$ defined in
  \eqref{eq:lyapunov} satisfies $\lambda (v)= \es - vL^*(1/v)$ for $v>v_c$,
  as claimed in \eqref{eq:lambdaesL} of Proposition~\ref{prop:lyapExp}. This
  completes the proof of Theorem~\ref{thm:PAMFCLT}.
\end{proof}

\begin{remark}
  \label{rem:getoeq}
  Theorem~\ref{thm:PAMFCLT} remains valid when the function $u (t,x)$ is
  replaced by $\ttE_{u_0}^\xi [N^{\ge}(t,x)]$ with $u_0$ as in
  \eqref{eq:inCond}. This is a consequence of Lemma~\ref{lem:inCond} again.
\end{remark}

\begin{remark}
  \label{rem:explicitN}
  It will be useful to have a more explicit formula for
  $\ln \ttE_0^\xi [N^{\ge} (n/v_0,n)]$. Combining \eqref{eq:NYcomp}
  with Corollary~\ref{cl:explicitW}, Proposition~\ref{prop:Donsker} and
  \eqref{eq:lambdaesL} yields the existence of a constant $C<\infty$ and a
  $\mathbb P$-a.s.~finite random variable $\mathcal N_3$ such that
  $\P$-a.s.~for all $n\ge \mathcal N_3$,
  \begin{equation*}
    \Big | \ln \ttE_0^\xi [N^{\ge} (n/v_0,n)] -
     \sum_{i=1}^n L_i^\zeta (\overbar \eta (v_0)) + n L(\overbar\eta
       (v_0)) \Big | \le C \ln n.
  \end{equation*}
\end{remark}

\section{Breakpoint behavior}
\label{sec:breakpoint}

The goal of this section is to prove the functional central limit theorem
for the breakpoint, Theorem~\ref{thm:bpFCLT}. This is done in
Section~\ref{ssec:bpFCLT} after some additional preparations.

\subsection{Perturbation estimates} 
\label{ssec:perturbationestimates}

The results of Section~\ref{sec:firstmoment} provide a reasonably precise
description of the behavior of expectations of $N^\ge(t, vt)$. We are now
interested in how sensitive the expectation of $N^\ge(\cdot, \cdot)$ is
to perturbations in the space and time coordinate. The first lemma deals
with space perturbations:

\begin{lemma}
  \label{lem:upperTail}
  (a) Let $\widetilde \varepsilon (t)$ be a positive function with
  $\lim_{t\to\infty} \widetilde \varepsilon (t)t^\delta =0$ for some
  $\delta >0$. Then for each $\varepsilon > 0$ there exists a constant
  $C_0 =C_0(\varepsilon) < \infty$  such that $\mathbb P$-a.s.,
  \begin{equation}
    \label{eq:breakPtSlope}
    \limsup_{t\to\infty}
    \sup\Big\{ \Big|
    \frac 1 h
    \ln \frac{ \ttE_{u_0}^\xi [N^\ge(t, vt + h)]}
    {\ttE_{u_0}^\xi [N^\ge(t, vt)]}
    -  L(\overbar \eta(v)) \Big |:
    (h,v)\in \mathcal E_t
    \Big\} \le \varepsilon,
  \end{equation}
  where
  $\mathcal E_t=\{(h,v): C_0 \ln t \le  |h| \le t \widetilde \varepsilon (t),
    v\in V, v+\frac ht \in V\}$.

  (b) There exist constants $C,c\in (0,\infty)$ and a $\P$-a.s.~finite
  random variable~$\mathcal T_1$ such that $\P$-a.s.~for
  all $t\ge \mathcal T_1$, uniformly for $0\le h \le t^{1/3}$
  and $v,v+h/t\in V$,
  \begin{equation*}
    c e^{-C h} \ttE_{u_0}^\xi \big[N^\ge(t, vt)\big]
    \le \ttE_{u_0}^\xi \big[N^\ge(t, vt + h)\big]
    \le  C e^{-c h} \ttE_{u_0}^\xi \big[N^\ge(t, vt)\big].
  \end{equation*}
\end{lemma}
\begin{proof}
  (a) We set $v':=v+\frac ht$.
  Without loss of generality we can assume $t$ to be large enough so
  that the events $\mathcal H_{vt}$ and $\mathcal H_{v't}$ occur and thus
  $\overbar \eta^\zeta_{vt}(v)$  and $\overbar \eta^\zeta_{v't}(v')$
  exist and satisfy corresponding versions of \eqref{eq:baretazeta}.
  By Lemmas~\ref{lem:fixedVsVariableTime} and~\ref{lem:inCond}, the fraction
  in \eqref{eq:breakPtSlope} can be approximated, up to a multiplicative
  constant that is irrelevant in the limit, by
  \begin{equation*}
    \frac{Y_{v'}(v't)}{Y_v(vt)}
    = \frac{E_{0} \Big[
        \exp \Big\{ \int_0^{H_{v't}} \zeta (X_s) \, \d s \Big\}
        ;{H_{v't} \le t} \Big]}
    {E_{0} \Big[
        \exp \Big\{ \int_0^{H_{vt}} \zeta (X_s) \, \d s \Big\}
        ;{H_{vt} \le t} \Big]}.
  \end{equation*}
  Using the notation from \eqref{eq:SnDef}, this can be rewritten in the
  same vein as in \eqref{eq:Ysimdecomp} as
  \begin{equation}
    \label{eq:BigFrac}
      \frac
      { E^{\zeta, \overbar \eta_{v't}^\zeta(v')}
        \Big[e^{ - \overbar \eta_{v't}^\zeta(v')
            \sum_{i=1}^{v't}  \widetilde \tau_i };
          \sum_{i=1}^{v't} \widetilde \tau_i \in (-\infty,0] \Big]
        \cdot
        e^{ -  S_{v't}^{\zeta, v'}
          (\overbar \eta_{v't}^\zeta (v') ) }}
      { E^{\zeta, \overbar \eta_{vt}^\zeta(v)}
        \Big[e^{ - \overbar \eta_{vt}^\zeta (v)
            \sum_{i=1}^{vt}  \widehat \tau_i };
          \sum_{i=1}^{vt} \widehat \tau_i \in (-\infty,0] \Big]
        \cdot
        e^{ -  S_{vt}^{\zeta, v}  (\overbar \eta_{vt}^\zeta (v) )
          }},
  \end{equation}
  where, similarly as before
  $\widehat \tau_i := \tau_i - E^{\zeta, \overbar \eta_{vt}^\zeta(v)} [\tau_i]$
  and
  $\widetilde \tau_i :=
  \tau_i - E^{\zeta, \overbar \eta_{v't}^\zeta(v')} [\tau_i]$.
  By the same methods as
  in~\eqref{eq:ProdBEupperBd}--\eqref{eq:theta1Second}, the expectations
  in the numerator and denominator of \eqref{eq:BigFrac} are both of
  order $t^{-1/2}$. Their ratio is thus bounded from above and below by
  positive finite constants and can  be neglected in the limit taken in
  \eqref{eq:breakPtSlope}.

  The remaining terms in \eqref{eq:BigFrac} contribute to the minuend of
  \eqref{eq:breakPtSlope} as
  \begin{equation}
    \label{eq:Ses}
    \begin{split}
      \frac 1h
      \big(S_{vt}^{\zeta, v}  (\overbar \eta_{vt}^\zeta (v))
        - S_{vt}^{\zeta, v} (\overbar \eta_{v't}^\zeta (v') )\big)
      +\frac 1h
      \big(S_{vt}^{\zeta, v}  (\overbar \eta_{v't}^\zeta (v'))
        - S_{v't}^{\zeta, v'} (\overbar \eta_{v't}^\zeta (v') )\big).
    \end{split}
  \end{equation}
  In order to show that the first summand on the right-hand side of
  \eqref{eq:Ses} is negligible uniformly as $t \to \infty$, we write
  \begin{align}
    \label{eq:expandS}
    \begin{split}
      S_{vt}^{\zeta, v} (\overbar \eta_{v't}^\zeta (v') )
      &=
      S_{vt}^{\zeta, v} (\overbar \eta_{vt}^\zeta (v) )
      + (S_{vt}^{\zeta, v})' (\overbar \eta_{vt}^\zeta (v) )
      \, \big( \overbar \eta_{v't}^\zeta (v')
        - \overbar \eta_{vt}^\zeta (v) \big)
      \\&+ (S_{vt}^{\zeta, v})'' (\widetilde \eta )
      \, \big( \overbar \eta_{v't}^\zeta (v') - \overbar \eta_{vt}^\zeta
        (v) \big)^2,
    \end{split}
  \end{align}
  for some $\widetilde \eta \in \Delta$ with
  $\big \vert \widetilde \eta - \overbar \eta_{vt}^\zeta(v) \big \vert
  \le \big | \overbar \eta_{v't}^\zeta(v')
  - \overbar \eta_{vt}^\zeta(v) \big |$.
  As observed below \eqref{eq:Texpand}, one has
  $(S_{vt}^{\zeta, v})' (\overbar \eta_{vt}^\zeta (v) )=0$, so the
  second term vanishes. For the third one, note that by Lemma
  \ref{lem:etandep}, $\P$-a.s.~for all $t$ large enough,
  \begin{equation}
    \label{eq:etaCloseness}
    \vert \overbar \eta_{v't}^\zeta  (v')
    - \overbar \eta_{vt}^\zeta  (v') \vert
    \le \frac{Ch}{t}.
  \end{equation}
  Moreover, by the characterizing property \eqref{eq:qqqqa} of
  $\overbar \eta_{vt}^\zeta  (v )$, Lemma~\ref{lem:lambdaProps} and the
  implicit function theorem, we see that
  $v \mapsto \overbar \eta_{vt}^\zeta  (v )$ is differentiable on the
  interior of $V$, with uniformly bounded derivative. Therefore, on
  $\mathcal H_{vt}$, uniformly for $v \in V$ and $h$ as in
  \eqref{eq:breakPtSlope},
  \begin{equation}
    \label{eq:etaDiffBd}
    \vert  \overbar \eta_{vt}^\zeta  (v' )
    -\overbar \eta_{vt}^\zeta  (v )  \vert
    \le \frac{Ch}{t}.
  \end{equation}
  Recalling \eqref{eq:Ssecondbound}, we see that, $\mathbb P$-a.s.,
  $(S_{vt }^{\zeta, v} )'' (\widetilde \eta)  \le Ct$ uniformly in
  $v \in V$ and $t$ large. Combined with
  \eqref{eq:expandS} to \eqref{eq:etaDiffBd} we thus deduce that $\P$-a.s.,
  the first term on the right-hand side of \eqref{eq:Ses} satisfies
  \begin{equation}
    \label{eq:SDiffBd}
    \Big \vert
    \frac 1h
    \big(S_{vt}^{\zeta, v}  (\overbar \eta_{vt}^\zeta (v))
      - S_{vt}^{\zeta, v} (\overbar \eta_{v't}^\zeta (v') )\big)
    \Big \vert
    \le \frac{Ch}t
  \end{equation}
  which is negligible in the limit considered in \eqref{eq:breakPtSlope}.

  Plugging in the definitions, the second summand on the right-hand side
  of \eqref{eq:Ses} satisfies
  \begin{align}
    \label{eq:hLLN}
    \begin{split}
      \frac 1h
      \big(S_{vt}^{\zeta, v}  (\overbar \eta_{v't}^\zeta (v'))
        - S_{v't}^{\zeta, v'} (\overbar \eta_{v't}^\zeta (v') )\big)
      &=\frac 1h \sum_{i=vt+1}^{v't} L_i^\zeta
      \big(\overbar \eta_{v't}^\zeta (v')\big)
    \end{split}
  \end{align}
  (where the sum should be interpreted as $-\sum_{i=v't+1}^{vt}$ if $v'<v$).
  The right-hand side of \eqref{eq:hLLN} can be approximated with the
  help of the following claim.
  \begin{claim} \label{cl:perSecPart}
    For each $\varepsilon >0$ and each $q \in \N$
    there exists a constant $C=C(q, \varepsilon)<\infty$ such that
    for all $t$ large enough,
    \begin{equation*}
      \P \Big( \sup_{\substack{v \in V\\
            C(q, \varepsilon) \ln t \le  |h| \le t \widetilde \varepsilon
            (t)}} \Big|
        \frac1h \sum_{i=vt +1}^{v't}
        L_i^\zeta(\overbar \eta^\zeta_{v't}(v'))
        - L(\overbar \eta (v)) \Big| > \varepsilon, \mathcal
        H_{v't},\mathcal H_{vt} \Big)
      \le Ct^{-q}
    \end{equation*}
    with $\widetilde \varepsilon(t)$ and $v'$ as in
    Lemma~\ref{lem:upperTail}.
  \end{claim}
  We postpone the proof of this claim after the proof of
  Lemma~\ref{lem:upperTail}. Inequality \eqref{eq:SDiffBd} and
  Claim~\ref{cl:perSecPart} together imply that the left-hand side of
  \eqref{eq:Ses} $\mathbb P$-a.s.~satisfies
  \begin{equation*}
    \limsup_{t\to\infty}\sup_{(t,v)}\Big\{
      \Big|\frac 1 h
      \big(S_{vt}^{\zeta, v}  (\overbar \eta_{vt}^\zeta (v))
        - S_{v't}^{\zeta, v'} (\overbar \eta_{v't}^\zeta (v')
          )\big)-L(\overbar \eta (v))\Big|
      \Big\}
    \le \varepsilon,
  \end{equation*}
  where the supremum is taken over all $(t,v)$ satisfying
  $C(2,\varepsilon ) \ln t\le |h|\le t\widetilde \varepsilon (t)$ and
  $v \in V$. This is what is necessary to prove
  Lemma~\ref{lem:upperTail}(a).

  (b) Using the same arguments as in the proof of (a), it is sufficient
  show that the exponential factors in \eqref{eq:BigFrac} are bounded
  from above and below by exponential functions, that is the right-hand
  side of \eqref{eq:Ses} is bounded away from $0$ and $\infty$. However,
  for the second summand on the right-hand side this easily follows from
  \eqref{eq:hLLN}, because
  $c^{-1}<L_i^\zeta (\overbar \eta^\zeta_{v't}(v'))<c<0$ uniformly in
  $i\ge 0$,  $v \in V$, and $\xi$ satisfying \eqref{eq:xiAssumptions}.
  The first summand can be neglected for $t$ sufficiently large due to
  \eqref{eq:SDiffBd}.
\end{proof}

\begin{proof}[Proof of Claim \ref{cl:perSecPart}]
  We rewrite
  \begin{align}
    \label{eq:sumhsplit}
    \begin{split}
      & \sum_{i=vt +1}^{v't}
      L_i^\zeta(\overbar \eta^\zeta_{v't}(v'))
      - L(\overbar \eta (v))\\
      &=\sum_{i=vt +1}^{v't}
      \big(  L_i^\zeta(\overbar \eta^\zeta_{v't}(v'))
        -  L_i^\zeta(\overbar \eta(v)) \big)
      + \sum_{i=vt +1}^{v't}
      \big( L_i^\zeta(\overbar \eta(v))
        - L(\overbar \eta (v)) \big).
    \end{split}
  \end{align}
  Observing that the family of functions
  $\big(\eta\mapsto
    L_i^{\zeta}(\eta )\big)_{i
    \in \Z, -\es \le \zeta (j)\le 0 \, \forall j\in \mathbb Z}$
  is equicontinuous on $\Delta $, Proposition~\ref{prop:etaExist},
  \eqref{eq:etaCloseness} and \eqref{eq:etaDiffBd} yield that
  \begin{equation}
    \P \Big( \sup_{\substack{v \in V\\
          \ln t \le  |h| \le t \widetilde \varepsilon
          (t)}}  \Big |  \frac1h \sum_{i=vt +1}^{v't}
      \big(  L_i^\zeta(\overbar \eta^\zeta_{v't}(v'))
        -  L_i^\zeta(\overbar \eta(v)) \big)  \Big| \ge \frac{\varepsilon}{2},
      \mathcal H_{v't},\mathcal H_{vt}
      \Big)
    \le Ct^{-q}.
  \end{equation}
  Regarding the second summand on the right-hand side of
  \eqref{eq:sumhsplit}, it suffices to observe that for $C(q, \varepsilon)$
  large enough,
  \begin{equation}
    \label{eq:unifLDP}
    \P \Big( \sup_{\substack{x \in \Delta \\
          C(q, \varepsilon) \ln t \le  |h| \le t  \widetilde \varepsilon(t)}}
      \Big|  \frac1h \sum_{i=vt +1}^{v't}
      L_i^\zeta(x) - L(x) \Big| \ge \varepsilon / 2 \Big) \le Ct^{-q},
  \end{equation}
  which follows from the Hoeffding type bound (Lemma~\ref{lem:hoef})
  using the same steps as in the proof of Claim~\ref{cl:Lprimeconc}.
  Combining \eqref{eq:sumhsplit}--\eqref{eq:unifLDP}  with \eqref{eq:Hn}
  finishes the proof of the claim.
\end{proof}

We now deal with time perturbations, where it is possible and
useful to obtain more precise estimates.

\begin{lemma}
  \label{lem:timepert}
  (a) Let $\varepsilon (t)$ be a function such that
  $\lim_{t\to\infty}\varepsilon (t)=0$. Then there exists a constant
  $C\in (0,\infty)$ and a $\mathbb P$-a.s.~finite random variable
  $\mathcal T_2$ such that $\mathbb P$-a.s.~for all $t\ge \mathcal T_2$,
  \begin{equation}
    \label{eq:timepert}
    \sup_{(h,v)\in \mathcal E_t}\bigg|
    \ln \frac{\ttE_{u_0}^\xi [N^\ge(t+h, v t)]}
    {\ttE_{u_0}^\xi [N^\ge(t, vt)]}
    - h \big(\es - \overbar \eta(v)\big)  \bigg |
    \le C+ C |h| \bigg(\sqrt {\frac {\ln t}{t}} + \frac {|h|}t\bigg),
  \end{equation}
  where
  $\mathcal E_t=\{(h,v):  |h| \le t \varepsilon (t),
    v\in V,vt/(t+h)\in V\}$.

  (b) In particular, there exist constants $C,c\in (0,\infty)$
   such that for $\mathcal T_2$ as in $(a)$ the following holds true:
  $\P$-a.s.~for all
  $t\ge \mathcal T_2$, uniformly in
  $0\le h\le t \varepsilon (t)$  and
  $v,vt/(t+h) \in V$,
  \begin{equation*}
    ce^{c h}  \ttE_{u_0}^\xi \big[N^\ge(t, vt)\big]
    \le \ttE_{u_0}^\xi \big[N^\ge(t+h, vt )\big]
    \le Ce^{C h} \ttE_{u_0}^\xi \big[N^\ge(t, vt)\big].
  \end{equation*}
\end{lemma}

\begin{remark}
  \label{rem:hnegative}
  By interchanging the roles of $vt$ and $vt+h$ as well as of $t$ and
  $t+h$ in the claims~(b) of Lemmas~\ref{lem:upperTail}
  and~\ref{lem:timepert}, respectively, it follows that they hold also
  for $h\in [-t^\frac13,0]$ and $h\in [-t\varepsilon(t),0]$,
  respectively, with minimal modifications: in Lemma~\ref{lem:timepert},
  the prefactors $c e^{ch}$ and $C e^{Ch}$ should be replaced by
  $c e^{Ch}$ and $C e^{ch}$, respectively; a similar replacement applies
  also in  Lemma~\ref{lem:upperTail}.
\end{remark}
\begin{proof}[Proof of Lemma~\ref{lem:timepert}]
  (a) Let  $v':=vt/(t+h)$. Through the proof we assume assume $t$ to be
  large enough such that $\mathcal H_{vt}$ and $\mathcal H_{v't}$ hold
  true. Using Proposition~\ref{prop:Donsker} and the same arguments as in
  the proof of Lemma~\ref{lem:upperTail}, the fraction in
  \eqref{eq:timepert} satisfies, for some $c\in (1,\infty)$,
  \begin{equation*}
    c^{-1}
    e^{h \es} \cdot \frac{Y^\approx_{v'}(vt)}{Y^\approx_{v}(vt)}
    \le
    \frac{\ttE_{u_0}^\xi [N^\ge(t+h, v t)]}
    {\ttE_{u_0}^\xi [N^\ge(t, vt)]}
    \le c
    e^{h \es} \cdot \frac{Y^\approx_{v'}(vt)}{Y^\approx_{v}(vt)}.
  \end{equation*}
  In addition, similarly to \eqref{eq:Ysimdecomp},
  \begin{equation*}
    \frac{Y^\approx_{v'}(vt)}{Y^\approx_{v}(vt)}=
      e^{h \es} \cdot \frac
      { E^{\zeta, \overbar \eta_{vt}^\zeta(v')}
        \Big[e^{ - \overbar \eta_{vt}^\zeta(v')
            \sum_{i=1}^{vt}  \widetilde \tau_i };
          \sum_{i=1}^{vt} \widetilde \tau_i \in [-K,0] \Big]
        \cdot
        e^{ -  S_{vt}^{\zeta, v'}
          (\overbar \eta_{vt}^\zeta (v') ) }}
      {E^{\zeta, \overbar \eta_{vt}^\zeta(v)}
        \Big[e^{ - \overbar \eta_{vt}^\zeta (v)
            \sum_{i=1}^{vt}  \widehat \tau_i };
          \sum_{i=1}^{vt} \widehat \tau_i \in [-K,0] \Big]
        \cdot
        e^{ -  S_{vt}^{\zeta, v}  (\overbar \eta_{vt}^\zeta
            (v) )
          }},
  \end{equation*}
  where again
  $\widehat \tau_i := \tau_i - E^{\zeta, \overbar \eta_{vt}^\zeta(v)} [\tau_i]$
  and
  $\widetilde \tau_i := \tau_i - E^{\zeta, \overbar \eta_{vt}^\zeta(v')} [\tau_i]$.
  As in the proof of Lemma \ref{lem:upperTail}, the ratio of the
  expectations in the numerator and denominator is asymptotically bounded
  from above and below. It follows that the expression in the supremum of
  \eqref{eq:timepert} is bounded from above by
  \begin{equation}
    \label{eq:time3Sum}
    \begin{split}
      &C + \big | h\overbar \eta (v) +
      S_{vt}^{\zeta, v}  (\overbar \eta_{vt}^\zeta (v))
        - S_{vt}^{\zeta, v'} (\overbar \eta_{vt}^\zeta (v') )\big|
        \\&\le C+ \big|
      h\overbar \eta (v) +
      \big(S_{vt}^{\zeta, v}  (\overbar \eta_{vt}^\zeta (v'))
        - S_{vt}^{\zeta, v'} (\overbar \eta_{vt}^\zeta (v') )\big)\big|
      \\&\quad+
      \big|\big(S_{vt}^{\zeta, v}  (\overbar \eta_{vt}^\zeta (v))
        - S_{vt}^{\zeta, v} (\overbar \eta_{vt}^\zeta (v') )\big)\big|.
    \end{split}
  \end{equation}
  Plugging in the definitions \eqref{eq:SnDef} and \eqref{eq:Vdef}, the
  second summand on the right-hand side satisfies
  \begin{equation*}
    \begin{split}
      \big|h \overbar \eta (v)+\big(S_{vt}^{\zeta, v}  (\overbar \eta_{vt}^\zeta (v'))
        - S_{vt}^{\zeta, v'} (\overbar \eta_{vt}^\zeta (v') )\big)\big|
      &= \Big|h \overbar \eta (v)+{vt \overbar \eta_{vt}^\zeta (v')}
      \Big( \frac 1v - \frac 1{v'}\Big)\Big|
      \\&=|h|\,|\overbar \eta (v) -\overbar \eta_{vt}^\zeta (v')|.
    \end{split}
  \end{equation*}
  Using \eqref{eq:etaDiffBd} and \eqref{eq:etaEmpEta} of
  Proposition~\ref{prop:etaExist} in combination with Borel-Cantelli
  lemma,  this can then be bounded by the right-hand side of
  \eqref{eq:timepert}. The last summand on the right-hand side of
  \eqref{eq:time3Sum} can be shown to be smaller than $C h^2/t$ using the
  same steps as in \eqref{eq:expandS}--\eqref{eq:SDiffBd} of the proof of
  Lemma~\ref{lem:upperTail}, completing the proof of (a). Claim (b)
  directly follows from (a).
\end{proof}

\subsection{Proof of Theorem~\ref{thm:bpFCLT}
  (functional CLT for the breakpoint)} 
\label{ssec:bpFCLT}

We now have all the ingredients to show our second main result, the
invariance principle for the breakpoint, Theorem~\ref{thm:bpFCLT}.

\begin{proof}[Proof of Theorem~\ref{thm:bpFCLT}]
  We must show that the sequence of processes
  $\frac 1{\overbar\sigma_v\sqrt n} \big(\overbar m_v(nt)-vnt\big)$
  converges to standard Brownian motion,
  where
  \begin{equation*}
    \overbar m_v(t) =
    \sup
    \Big\{ n \in \N  :  \ttE^\xi_{u_0} \big[  N^\ge(t,n) \big]
      \ge \frac 12 e^{t\lambda(v)}  \Big\}
  \end{equation*}
  was defined in \eqref{eq:mvt}.

  We assume that $u_0=\ind_{\{0\}}$ first. Let
  $u^\ge(t,x):= \ttE_0^\xi [N^{\ge}(t,x)]$, $t\ge 0$, $x\in \mathbb Z$,
  and extend it to $x\in \mathbb R$ by linear interpolation. Furthermore,
  set
  \begin{equation}
    \label{eq:Wv}
    U_v(t): = t \lambda (v) - \ln u^\ge(t,vt)-\ln 2.
  \end{equation}
  Recalling the definition of $\sigma_v^2$ from \eqref{eq:limitSigma}, by
  Remark~\ref{rem:getoeq},
  \begin{equation}
    \label{eq:WvFCLT}
    \bigg(t\mapsto \frac {U_v(nt)} { \sqrt { \sigma_v^2 v n}} \bigg)_{n\in \N}
    \ \text{converges as $n\to\infty$ to Brownian motion.}
  \end{equation}
  Obviously, $u^\ge(t,x)$ is decreasing in $x$ with
  $\lim_{t\to\infty} \frac 1t \ln u^\ge(t,0)=\lambda (0)>\lambda (v)$ and
  $\lim_{x\to\infty} u^\ge(t,x)= 0$,
  see Proposition \ref{prop:lyapExp}. Let $r=r(t)$ be the largest solution
  of the equation
  \begin{equation}
    \label{eq:rt}
    u^\ge(t,vt+r) = \frac 12 e^{t \lambda (v)},
  \end{equation}
  which exists $\mathbb P$-a.s.~for $t$ large enough by the previous
  considerations. Moreover, by the definition of $\overbar m_v(t)$,
  \begin{equation}
    \label{eq:raproxmv}
    r(t)-1 < \overbar m_v(t)-vt \le r(t).
  \end{equation}
  Combining equations \eqref{eq:Wv} and \eqref{eq:rt}, we see that $r(t)$
  is the largest solution to
  \begin{equation*}
    \ln \frac{u^\ge(t,vt+r(t))}{u^\ge(t,vt)} = U_v(t).
  \end{equation*}
  Let $\widetilde \varepsilon(t)$ be an arbitrary positive function with
  $\widetilde  \varepsilon(t) t^\frac14 \to 0$ and
  $\widetilde \varepsilon(t) t^\frac12 \to \infty$ as $t \to \infty$.  By
  the space perturbation Lemma~\ref{lem:upperTail}, using also the
  monotonicity of $u^\ge(t,\cdot)$ and the fact that
  $L(\overbar \eta (v))<0$, we obtain that for every
  \begin{equation} \label{eq:deltaChoice}
    \delta \in (0,|L(\overbar \eta (v))|),
  \end{equation}
  $\P$-a.s.~for all $t$ large enough,
  \begin{equation*}
    \overbar \varphi_t (r(t)) L(\overbar \eta (v)) -\delta |r(t)|
    \le \ln \frac{u^\ge(t,vt+r(t))}{u^\ge(t,vt)}
    \le \underline \varphi_t (r(t)) L(\overbar \eta (v)) + \delta |r(t)|;
  \end{equation*}
  here, for $C_0=C_0(\delta )$, the functions $\overbar \varphi_t $ and
  $\underline \varphi_t $ are given by
  \begin{equation*}
    \begin{split}
      \underline \varphi_t(r)
      &= \sup \big \{s  : s\le r \text{ and }C_0 \ln t \le |s| \le
        t\widetilde  \varepsilon (t)\big \},
      \\
      \overbar \varphi_t(r)
      &= \inf \big \{s  : s\ge r \text{ and }C_0 \ln t \le |s| \le
        t\widetilde  \varepsilon (t) \big \},
    \end{split}
  \end{equation*}
  and satisfy $\underline \varphi_t(r) = \overbar \varphi_t (r) = r$ for
  $C_0\ln t \le |r| \le t\varepsilon (t)$ and
  $\underline \varphi_t\le \overbar \varphi_t$. This implies that whenever
  \begin{equation} \label{eq:Uin}
    |U_v(t)|\in
    \big[C_0  (|L(\overbar \eta (v))|+\delta)\ln t,
      t \widetilde  \varepsilon (t)(|L(\overbar \eta (v))|-\delta) \big],
  \end{equation}
  then, due to \eqref{eq:deltaChoice},
  \begin{equation*}
    r(t) \in \Big[
      \frac{U_v(t)}{L(\overbar \eta (v))\pm \delta },
      \frac{U_v(t)}{L(\overbar \eta (v))\mp \delta }
      \Big],
  \end{equation*}
  where the upper signs correspond to $U_v(t)>0$ and the lower signs to
  $U_v(t)<0$. In particular, since $U_v$ satisfies the invariance principle
  \eqref{eq:WvFCLT}, property \eqref{eq:Uin} is satisfied with probability
  tending to 1 as $t\to\infty$. Since $\delta $ is arbitrary, it thus
  follows that in $\P$-distribution
  \begin{equation*}
    \lim_{n\to\infty} \frac  1{\sqrt n} r(n\cdot) = \lim_{n\to\infty}
    \frac1{\sqrt n} \cdot \frac{U_v(n\cdot)}{L(\overbar \eta (v))}
  \end{equation*}
  as processes defined on $[0,\infty)$, which together with
  \eqref{eq:raproxmv} and \eqref{eq:WvFCLT} implies the claim of the
  theorem for
  \begin{equation}
    \label{eq:barsigmav}
    \overbar \sigma_v = \frac{ \sqrt{\sigma_v^2 v}}{ |L(\overbar \eta (v))|}.
  \end{equation}

  The case of general $u_0$ satisfying \eqref{eq:inCond} then follows from
  Lemmas~\ref{lem:inCond} and~\ref{lem:upperTail}.
  This completes the proof.
\end{proof}

\subsection{Invariance principle for the breakpoint inverse} 

We will later on need the following invariance principle for a generalized
inverse of the breakpoint defined by $T_0=0$ and, for $n\ge  1$,
\begin{equation}
  \label{eq:Tndef}
  T_n:=\inf\Big\{t\ge 0: \ttE^\xi _{u_0}[N^{\ge}(t,n)]\ge \frac 12\Big\}=
  \inf\{t\ge 0: \overbar m(t)\ge n\}.
\end{equation}
Observe, that by definition
\begin{equation}
  \label{eq:Tineqt}
  T_{\overbar m(t)}\le t.
\end{equation}

\begin{theorem}
  \label{thm:bpinverse}
  There exists a $\P$-a.s.~finite
  random variable $\mathcal C =\mathcal C(\xi )$ and a constant
  $C_1<\infty$ such that $\mathbb P$-a.s.~for
  all $n\ge 1$,
  \begin{equation}
    \label{eq:Tnseries}
    \bigg | T_n -  \bigg(\frac n {v_0}
      + \frac 1 {v_0 L(\overbar \eta (v_0))}
    \sum_{i=1}^n \big(L_i^\zeta (\overbar \eta (v_0))-L(\overbar \eta
        (v_0))\big)\bigg)\bigg| \le \mathcal C + C_1 \ln  n.
  \end{equation}
  In particular,
  \begin{equation} \label{eq:LLNTn}
    \lim_{n\to\infty}\frac{T_n}{n} = \frac{1}{v_0}, \quad \P\text{-a.s.},
  \end{equation}
  and, a fortiori, the sequence
  \begin{equation*}
    t\mapsto \frac{v_0L(\overbar \eta (v_0))}{ \sqrt{\sigma_{v_0} n}}
    \Big(T_{nt}-\frac {nt}{v_0}\Big), \qquad n\ge 0,
  \end{equation*}
  converges  as $n\to\infty$ in $\mathbb P$-distribution to
  standard Brownian motion.
\end{theorem}

\begin{proof}
  To show \eqref{eq:Tnseries}, we set
  \begin{equation*}
    h_n =  \frac 1 {v_0 L(\overbar \eta (v_0))}
    \sum_{i=1}^n \big(L_i^\zeta (\overbar \eta (v_0))-L(\overbar \eta
        (v_0))\big).
  \end{equation*}
  Observe that $\mathbb P$-a.s.~for all $n$ large enough
  \begin{equation}
    \label{eq:hlil}
    |h_n|\le C \sqrt {n \ln \ln n}.
  \end{equation}
  Indeed, the random variables
  $L_i^\zeta (\overbar\eta(v_0))-L(\overbar \eta(v_0))$ are centered and
  mixing as in \eqref{eq:LcondBd}. We can thus apply Azuma's inequality
  for mixing sequences, Lemma~\ref{lem:hoef}, which can be turned into a
  maximal inequality using \cite[Theorem~1]{KeMa-11} to deduce that for
  all $a\ge 0$
  \begin{equation*}
    \mathbb P(\max_{k\le n}|h_k| \ge a) \le C e^{-c a^2/n}.
  \end{equation*}
  The usual steps of the proof of the upper bound in the classical law of the
  iterated logarithm then provide us with \eqref{eq:hlil}.

  We now fix $\alpha \in \mathbb R$ and estimate
  $\ln \ttE_{u_0}^\xi[N^{\ge}(n/v_0+h+\alpha \ln n,n)]$. To this end we
  use the time perturbation
  Lemma~\ref{lem:timepert} which can be applied due to \eqref{eq:hlil}.
  Combining this with Remarks~\ref{rem:getoeq} and \ref{rem:explicitN} in order to rewrite
  $\ln \ttE_{u_0}^\xi [N^{\ge}(n/v_0,n)]$, we obtain that
  \begin{equation}
    \begin{split}
      \label{eq:zcvzcz}
      &\ln \ttE_{u_0}^\xi[N^{\ge}(n/v_0+h_n+\alpha \ln n,n)]
      \\&=
      \sum_{i=1}^n \big(L_i^\zeta (\overbar \eta (v_0)) -  L(\overbar\eta
          (v_0))\big) + (h_n+\alpha \ln n)\big(\es - \overbar \eta (v_0)\big)
      + \varepsilon(\alpha ,n)
      \\&= \alpha (\es - \overbar \eta (v_0)) \ln n  +  \varepsilon(\alpha, n),
    \end{split}
  \end{equation}
  where the last equality follows from
  $\es - \overbar \eta (v_0)+v_0L(\overbar \eta (v_0))=0$,
  cf.~\eqref{eq:lambdaesL}. Furthermore, the error term $\varepsilon(\alpha, n)$ satisfies
  \begin{equation*}
    |\varepsilon(\alpha, n)|\le C + C(|h_n| + (|\alpha| \vee 1)\ln n)\bigg(\sqrt {\frac{\ln n}{n}} +
      \frac{|h_n|+ (|\alpha| \vee 1)\ln n}n\bigg) + C \ln n,
  \end{equation*}
  with $C$ depending on neither $\alpha$ nor $n$.
  Choosing $\alpha $ sufficiently large positive (respectively negative) the
  right-hand side of \eqref{eq:zcvzcz} converges to $+\infty$
  (respectively $-\infty$). Recalling the definition \eqref{eq:Tndef} of $T_n$,
  claim \eqref{eq:Tnseries} follows for all $n$ sufficiently large.
  Adjusting $\mathcal C$ then deals with the remaining $n$'s.

  The law of large numbers \eqref{eq:LLNTn} directly follows from
  \eqref{eq:Tnseries} in
  combination with the ergodic theorem and the definitions
  from~\eqref{eq:empLav} and~\eqref{eq:Ldef}. The invariance principle is
  then again a consequence of this formula and Remark~\ref{rem:getoeq}.
\end{proof}

Theorem~\ref{thm:bpinverse} can be used to deduce a strong law of numbers
for the breakpoint which does not follow easily from the previous
argumentation.
\begin{corollary}
  \label{cor:breakLLN}
  Under \eqref{eq:xiAssumptions}, \eqref{eq:inCond} and  \eqref{eq:vAssumptions},
  \begin{equation*}
    \lim_{t\to\infty}\frac {\overbar m(t)}{t} = v_0, \qquad \P\text{-a.s.}
  \end{equation*}
\end{corollary}
\begin{proof}
  Consider first the case $u_0=\delta_0$. Then for $\varepsilon >0$, by
  \eqref{eq:lyapunov},  $\liminf \overbar m(t)/t\ge (1-\varepsilon )v_0$,
  $\mathbb P$-a.s.
  On the other hand, since $\overbar m(t)$ diverges
  $\P$-a.s.~and $t\ge T_{\overbar m(t)}$ due to \eqref{eq:Tineqt},
  \begin{equation*}
    \liminf \frac t {\overbar m(t)}
    \ge \liminf \frac{T_{\overbar m(t)}}{\overbar m(t)}=\frac 1 {v_0},
    \qquad\P\text{-a.s.},
  \end{equation*}
  by \eqref{eq:LLNTn},
  completing the proof for $u_0=\delta_0$. General $u_0$ satisfying \eqref{eq:inCond} can then
  be handled using Lemmas~\ref{lem:inCond} and~\ref{lem:upperTail}.
\end{proof}

\section{The breakpoint approximates the maximum}
\label{sec:maximumproofs}

In this section we prove the main results about the position of the
rightmost particle $M(t)$ and its median $m(t)$. We will see that those
are well approximated by the breakpoint, and thus satisfy the same
invariance principles.

It is elementary to obtain upper tail estimates for $M(t)$ and an upper
bound on $m(t)$: the definition of $\overbar m(t)$ and the Markov inequality
imply directly that
\begin{equation}
  \label{eq:barmm}
  \overbar m(t)\ge m(t).
\end{equation}
In addition, by
Lemma~\ref{lem:upperTail}(b), $\mathbb P$-a.s.~for $t$ large enough,
\begin{equation}
  \label{eq:Mtupperbound}
  \ttP^\xi_{u_0}(M(t)\ge \overbar m(t)+h )
  \le \ttE^\xi_{u_0}[N(t,\overbar m(t)+h)]
  \le C e^{-c h },
  \quad h\in ( 0,t^{1/3}).
\end{equation}
Note that these estimates are rather coarse. One expects
\eqref{eq:Mtupperbound} to hold with $m(t)$ instead of $\overbar m(t)$
and $\overbar m(t)-m(t)\asymp \ln t$. These bounds, however, are more than
sufficient to show the stated functional limit theorems.

As usual in the branching random walk literature, the lower bounds are
more difficult, and  are obtained via second moment estimates on the
so-called leading particles. Since $m(t)$ and $M(t)$
are stochastically increasing in the initial condition, we will assume,
without loss of generality, that ${u_0=\ind_{\{0\}}}$ throughout this section.

\subsection{Leading particles} 
\label{ssec:leadingparticles}

We consider a special class of particles $Y\in N(t)$ with
trajectories satisfying
\begin{equation}
  \label{eq:leading}
  \begin{split}
    &Y_t\ge \overbar m(t), Y_{T_{\overbar m(t)}}\ge \overbar m(t),
    \\&\text{and} \quad
    H^Y_k \ge T_k - \alpha \psi^\xi (k) \text{ for all }1\le k<\overbar m(t),
  \end{split}
\end{equation}
where $H^Y_k=\inf\{s\ge 0:Y_s=k\}$, $\alpha >2$ is a fixed constant, $T_k$
is the breakpoint inverse introduced in \eqref{eq:Tndef} of
Theorem~\ref{thm:bpinverse}, and $\psi^\xi $ is defined by
\begin{equation}
  \label{eq:psi}
  \psi^\xi (k)= \mathcal C(\xi ) + C_1 (1 \vee \ln k),
\end{equation}
where $\mathcal C(\xi )$ and $C_1$ are as in \eqref{eq:Tnseries}.
Analogously to the literature on homogeneous branching random walk, we
will call such particles \emph{leading at time $t$}. We further set
\begin{equation*}
  N^{\mathcal L}_t = \big|\{Y\in N(t): Y \text{ is leading at time $t$}\}\big|.
\end{equation*}
The probability of finding a leading particle at time $t$ is bounded from
below in the following proposition.

\begin{proposition}
  \label{prop:leadingparticles}
  There exists a constant $\gamma >0$ such that $\mathbb P$-a.s.~for all
  $t$ large enough
  \begin{equation*}
    \ttP^\xi_0 ( N^{\mathcal L}_t\ge 1) \ge  t^{-\gamma }.
  \end{equation*}
\end{proposition}
The proof of this proposition will be based on the classical Paley-Zygmund
inequality
\begin{equation}
  \label{eq:PZ}
  \ttP_0^\xi (N^{\mathcal L}_t \ge 1)
  \ge \frac{\ttE_0^\xi [N^{\mathcal L}_t]^2}
  {\ttE_0^\xi[(N^{\mathcal L}_t)^2]}.
\end{equation}
Estimates for the expectations on the right-hand side are provided in the
following two subsections. Since we do not strive to find the optimal
constant $\gamma$ in this paper, we use $\gamma$ to denote a generic
large constant whose value can change during the computations.

\subsubsection{First moment for the leading particles}
\label{ssec:leadingparticlesfirst}

\begin{lemma}
  \label{lem:lowerleading}
  There exists a constant $\gamma >0$ such that $\mathbb P$-a.s.~for all
  $t$ large enough
  \begin{equation*}
    \ttE_0^\xi [N^{\mathcal L}_t] \ge t^{-\gamma }.
  \end{equation*}
\end{lemma}
\begin{proof}
  Let $\bar t= T_{\overbar m(t)}$. Then by \eqref{eq:Tineqt},
  $\bar t \le t$. Hence, every particle
  satisfying $Y_{\bar t}\ge \overbar m(t)$ has probability
  at least 1/2 to satisfy also $Y_t\ge \overbar m(t)$. Further, by the
  definitions of $\bar t$ and $\overbar m(t)$, we have
  $\ttE_0^\xi [N^\ge(\bar t,\overbar m(t))]\ge 1/2$. Therefore,
  \begin{equation*}
    \ttE_0^\xi[N^{\mathcal L}_t]
    \ge
    \frac{\ttE_0^\xi \big[\big|\big\{Y\in N(\bar t)
          :Y_{\bar t}\ge \overbar m(t),
          H_k\ge T_k - \alpha \psi^\xi (k)\,  \forall k<\overbar
          m(t)\big\}\big|\big]}
    {4 \ttE_0^\xi [N^\ge(\bar t,\overbar m(t))]}.
  \end{equation*}
  Using the Feynman-Kac representation (Proposition~\ref{prop:FK}) this
  implies that
  \begin{equation}
    \label{eq:ENLaa}
    \ttE_0^\xi[N^{\mathcal L}_t]
    \ge
    \frac
    {E_0\Big[e^{\int_0^{\bar t} \zeta (X_s) \,\d s}
        ;X_{\bar t} \ge \overbar m(t),
        H_k\ge T_k - \alpha \psi^\xi (k) \, \forall k<\overbar m(t)\Big]}
    {4E_0\Big[e^{\int_0^{\bar t} \zeta (X_s) \,\d s}
        ;X_{\bar t} \ge \overbar m(t)\Big]}.
  \end{equation}
  Following the same steps as in the proof of the lower bound in
  Lemma~\ref{lem:fixedVsVariableTime}, the numerator in \eqref{eq:ENLaa}
  satisfies
  \begin{equation*}
    \begin{split}
      &E_0\Big[e^{\int_0^{\bar t} \zeta (X_s) \,\d s}
        ;X_{\bar t} \ge \overbar m(t),
        H_k\ge T_k - \alpha \psi^\xi (k) \, \forall k<\overbar m(t)\Big]
      \\&\ge
      E_0\bigg[e^{\int_0^{H_{\overbar m(t)}} \zeta (X_s) \,\d s}
        \times E_{\overbar m(t)}\Big[e^{\int_{0}^{r}\zeta (X_s) \,\d s};X_{r}\ge
          \overbar m(t)\Big]{\Big|_{r=\bar t-H_{\overbar m(t)}}}
        ;\\&\qquad\qquad H_{\overbar m(t)}\in[\bar t-K,\bar t],
        H_k\ge T_k - \alpha \psi^\xi(k) \, \forall k<\overbar m(t)\bigg]
      \\&\ge c
      E_0\Big[e^{\int_0^{H_{\overbar m(t)}} \zeta (X_s) \,\d s}
        ;H_{\overbar m(t)}\in [\bar t-K,\bar t],
          H_k\ge T_k-\alpha \psi^\xi (k) \, \forall k<\overbar m(t)\Big],
    \end{split}
  \end{equation*}
  where in the last step we used
  $\essinf_{\zeta ,r\le K, \, x \in \Z}
  E_x[e^{\int_0^r\zeta (X_s)\, \d s};{X_r\ge x}]\ge c>0$,
  due to \eqref{eq:xiAssumptions}. On the other hand, by
  Lemma~\ref{lem:fixedVsVariableTime}, the denominator of
  \eqref{eq:ENLaa} is bounded from above by
  $ C E_0\big[e^{\int_0^{H_{\overbar m(t)}} \zeta (X_s) \,\d s} ;
    H_{\overbar m(t)}\in [\bar t-K,\bar t]\big]$.
  Replacing now $\overbar m(t)$ by $n$, and thus $\bar t$ by $T_n$, and
  using the law of large numbers \eqref{eq:LLNTn} for $T_n$,
  we observe from the previous reasoning that in order to show the lemma, it
  is sufficient to prove that $\mathbb P$-a.s., for all $n$ large enough,
  \begin{equation}
    \begin{split}
      \label{eq:ENLbb}
      \frac {E_0\Big[e^{\int_0^{H_n} \zeta (X_s) \,\d s}
          ;H_n\in [T_n-K,T_n],
          H_k\ge T_k-\alpha \psi^\xi(k) \, \forall k<n\Big]}
      {E_0\Big[e^{\int_0^{H_n} \zeta (X_s) \,\d s}
          ;H_n\in [T_n-K,T_n] \Big]} \ge n^{-\gamma }.
    \end{split}
  \end{equation}
  To prove \eqref{eq:ENLbb}, we set
  $\eta = \overbar \eta (v_0)$ below and rewrite its left-hand side as
  \begin{equation*}
    \begin{split}
      &\frac
      {E^{\zeta ,\eta }
        \big[e^{-\eta H_n};
          H_n\in [T_n-K,T_n],
          H_k\ge T_k-\alpha \psi^\xi (k) \,\forall k<n\big]}
      { E^{\zeta ,\eta}
        \big[e^{-\eta H_n}; H_n\in [T_n-K,T_n] \big]}
      \\&\qquad\ge c \cdot
      \frac
      {  P^{\zeta ,\eta}
        \big(H_n\in [T_n-K,T_n],
          H_k\ge T_k-\alpha \psi^\xi (k) \,\forall k<n\big)}
      { P^{\zeta ,\eta}\big(H_n\in [T_n-K,T_n] \big)}
      \\&\qquad\ge c \cdot P^{\zeta ,\eta}
      \big(H_n\in [T_n-K,T_n],
        H_k\ge T_k -\alpha \psi^\xi(k) \,\forall k<n\big).
    \end{split}
  \end{equation*}
  Setting $\widehat H_n := H_n - E^{\zeta ,\eta }[ H_n]$ and
  $R_n := T_n - E^{\zeta ,\eta }[H_n]$ in the last formula, we thus see that
  \eqref{eq:ENLbb} is equivalent to
  \begin{equation}
    \label{eq:PHoverR}
    P^{\zeta ,\eta }
    \Big(\widehat H_n \in [R_n-K, R_n],
      \widehat H_k \ge R_k-\alpha \psi^\xi(k)\,\forall k<n\Big) \ge
    n^{-\gamma }
  \end{equation}
  $\mathbb P$-a.s.~for all $n$ large enough. Note that $R_n$ depends
  on the random environment $\xi$ only, so it is not random under
  $P^{\zeta ,\eta }$.

  The next two claims will show that, after rescaling, the processes $R_n$
  and $\widehat H_n$ behave like Brownian motions. To approximate  $R_n$,
  whose increments are not stationary, we introduce an auxiliary
  process with stationary increments
  \begin{equation*}
    R'_n := \sum_{i=1}^n \rho_i, \qquad n\ge 1,
  \end{equation*}
  where
  \begin{equation}
    \label{eq:rhoi}
    \rho_i :=
    \frac 1 {v_0 L(\eta)}
    \big(L_i^\zeta (\eta)-L(\eta)\big)
    - (E^{\zeta ,\eta }[ \tau_i]-\mathbb
      E[E^{\zeta ,\eta }[
      \tau_i]]), \qquad i\ge 1.
  \end{equation}

  \begin{lemma}
    \label{lem:RtoBM}
    The random variables $R'_n$ are adapted to the filtration
    $\mathcal F_n=\sigma (\xi(i):i\le n)$, and $R'_n$ approximates $R_n$
    in the sense that $\P$-a.s.,
    \begin{equation}
      \label{eq:RRapprox}
      |R_n-R'_n|\le \psi^\xi(n), \qquad \text{for all $n\ge 0$},
    \end{equation}
    with $\psi^\xi$ as in \eqref{eq:psi}. Moreover, the sequence of
    increments $(\rho_n)$ is bounded, stationary and there exist some
    constants $c,C \in (0,\infty)$ such that
    \begin{equation}
      \label{eq:Rprimecorrelations}
      \big|\mathbb E[\rho_{n+m} \, |\, \mathcal F_n]\big|\le C e^{-cm}.
    \end{equation}
    Finally, there is $\sigma_1^2 \in (0,\infty)$ such that both
    processes, $[0,\infty) \ni t \mapsto n^{-1/2} R_{n t }$ and
    $[0,\infty) \ni t \mapsto  n^{-1/2}R'_{nt}$, converge as $n\to\infty$
    in $\mathbb P$-distribution to a Brownian motion with variance
    $\sigma_1^2$.
  \end{lemma}
  \begin{proof}
    The adaptedness of  $(R'_n)$  to $(\mathcal F_n)$, as well as
    the stationarity and the boundedness of $(\rho_n)$ follow directly
    from their definitions, recalling the  assumption
    \eqref{eq:xiAssumptions}.  The estimate \eqref{eq:RRapprox} follows
    from \eqref{eq:Tnseries} of Theorem~\ref{thm:bpinverse} after
    a straightforward computation.
    Furthermore, Lemma~\ref{lem:covBd} yields
    \begin{equation*}
      \big| \mathbb E [L_{n+m}^\zeta (\eta )-L(\eta ) \,|\, \mathcal F_n] \big|
      \le C e^{-c m},
    \end{equation*}
    and  analogically, bearing in mind that
    $(L_{i}^\zeta)' (\eta ) = E^{\zeta, \eta}[ \tau_i]$,
    \begin{equation*}
      \big|  \E \big[E^{\zeta ,\eta } [\tau_{n+m}]-\mathbb
        E[E^{\zeta ,\eta }
          [\tau_{n+m}]]  \,\big|\, \mathcal F_n \big] \big| \le C e^{-c m},
    \end{equation*}
    proving \eqref{eq:Rprimecorrelations}.

    Finally, observing that the increments of $R'_n$ are centered, the functional
    central limit theorem for $n^{-1/2} R'_{n \cdot }$ follows directly
    from a functional central limit theorem for stationary mixing
    sequences, see e.g.~Theorem~11 and Corollary~12 of \cite{MePeUt-06},
    the assumptions of which can be checked easily from
    \eqref{eq:Rprimecorrelations}. The functional central limit theorem
    for $n^{-1/2}R_{n\cdot}$ then follows from~\eqref{eq:RRapprox}.
  \end{proof}

  \begin{claim}
    \label{cl:HtoBM}
    There is $\sigma_2^2\in (0,\infty)$ such that $\mathbb P$-a.s., under
    $P^{\zeta ,\eta }$, $n^{-1/2} \widehat H_{n\cdot}$ converges to a
    Brownian motion with variance $\sigma_2^2$.
  \end{claim}
  \begin{proof}
    Since the $\tau_i$'s are independent under $P^{\zeta ,\eta }$,
    $\widehat H_n$ is a sum of independent and centered random variables,
    which have uniformly exponential tails. Moreover, the sequence of the
    variances of the increments is stationary under $\mathbb P$. The
    claim then follows easily by a functional version of the Lindeberg-Feller
    central limit theorem (see e.g.~\cite[Theorem 9.3.1]{GiSk-69}).
  \end{proof}

  \begin{remark}
    \label{rem:MM}
    In view of the last claim and Lemma \ref{lem:RtoBM}, the probability
    in \eqref{eq:PHoverR} can approximatively be viewed as the
    probability that one Brownian motion stays above another, quenched,
    Brownian motion. This problem was recently studied in \cite{MaMi-15}
    for the case of two independent Brownian motions, where it was shown
    that this probability behaves like $n^{-\gamma }$ with $\gamma $
    depending on the variances of the Brownian motions. More importantly,
    it was proved there that $\gamma >1/2$ whenever the variance of the
    quenched Brownian motion is positive. That implies that the price for
    a particle to be leading should be larger than in the homogeneous
    case, resulting thus in a larger backlog of $m(t)$ behind
    $ \overbar  m(t)$.

    In this paper, the situation is more intricate due to the
    dependencies of the random variables involved. Hence, we do not
    strive for the optimal~$\gamma $. Nevertheless, our proof partially
    builds on certain ideas appearing in \cite{MaMi-15}.
  \end{remark}

  We proceed by showing \eqref{eq:PHoverR}. In view of \eqref{eq:RRapprox},
  \begin{equation}
    \begin{split}
      \label{eq:newbeta}
      &P^{\zeta ,\eta }
      \Big(\widehat H_n \in [R_n-K, R_n],
        \widehat H_k \ge R_k-\alpha \psi^\xi (k)\,\forall 1 \le k<n\Big)
      \\&\ge
      P^{\zeta ,\eta }
      \Big(\widehat H_n-R'_n \in I_n,
        \widehat H_k-R'_k \ge -(\alpha -1)\psi^\xi(k)
        \,\forall 1 \le k<n\Big),
    \end{split}
  \end{equation}
  where $I_n=[R_n-R'_n-K,R_n-R'_n]$. Note that since $\alpha >2,$ we have
  that $R_n-R_n'-K\ge -(\alpha-1)\psi^\xi(n)$ for $n$ large enough.

  On the right-hand side of \eqref{eq:newbeta}, we require the process
  $\widehat H_n-R'_n$ to stay above the barrier between times $0$ and $n$
  and to be (almost) fixed at times $0$ and $n$. It turns out useful to
  split the problem into two parts: distancing the barrier at $0$, and
  distancing the barrier at $n$. Thus, we will consider two independent
  copies $X^1$ and $X^2$ of $X$ under the same measure $P^{\zeta ,\eta }$,
  and write $\widehat H^i_k$, $i=1,2$, for the associated hitting times.
  We further consider a random variable $\Sigma_n$ independent of $X^1$,
  $X^2$, which under $P^{\zeta ,\eta }$ is uniformly distributed on
  $\{1,\dots, n-1\}$. We introduce
  \begin{equation}
    \label{eq:beta}
    \beta^i_k = \widehat H^i_k-R'_k, \qquad k\ge 0, i=1,2,
  \end{equation}
  as a convenient abbreviation---mind, however, that $\beta^i_k$ has a
  part, $R'_k$, that depends only on the random environment $\xi $, and
  another part, $\widehat H^i_k$, that depends on both, $\xi $ and the random
  walk $X^i$.
  Furthermore, define
  \begin{equation}
    \label{eq:betaDist}
    \beta_k =
    \begin{cases}
      \beta^1_k, & \text{for }1 \le  k\le \Sigma_n,\\
      \beta^1_{\Sigma_n}+(\beta^2_k-\beta^2_{\Sigma_n}),\qquad
      &\text{for }\Sigma_n<k\le n.
    \end{cases}
  \end{equation}
  The process $\beta $ has the increments of $\beta^1$ before $\Sigma_n$
  and the increments of $\beta^2$ after $\Sigma_n$. Since,  under
  $P^{\zeta ,\eta }$, the processes $\widehat H^1$ and $\widehat H^2$ are
  independent and have independent increments, it follows that the
  process $\beta$ has, under $P^{\zeta ,\eta }$, the same distribution as
  $\widehat H_\cdot- R'_\cdot$. Hence, \eqref{eq:PHoverR} will follow if
  we show that, $\P$-a.s.~for all $n$ large enough,
  \begin{equation}
    \label{eq:Pbeta}
    P^{\zeta ,\eta }\big(\beta_k\ge -(\alpha -1)\psi^\xi(k) \, \forall 1 \le  k<n,
      \beta_n\in I_n\big) \ge n^{-\gamma }.
  \end{equation}
  Finally, we write $\overbar \beta^1_k = \beta^1_{k+1}-\beta^1_1$, and
  $\overbar \beta^2_k = \beta^2_{n-k}-\beta^2_n$, $k=0,\dots,n$, for
  $\beta^1$ shifted by one, and `$\beta^2$ running backwards from $n$',
  respectively. Due to the independence of the increments of $\beta^1$ under
  $P^{\zeta ,\eta }$, $\beta^1_1$ is independent of $\overbar \beta^1$.
  We then decompose $\beta_n$ as
  \begin{equation} \label{eq:betanEq}
    \beta_n = \beta^1_{\Sigma_n}+ (\beta^2_n-\beta^2_{\Sigma_n})
    =  \beta^{1}_1+\overbar \beta^1_{\Sigma_n-1} -
    \overbar \beta^2_{n-\Sigma_n}.
  \end{equation}

  The following lemma, the proof of which is postponed to the end of this
  subsection, provides a control on the processes $\beta^1 $ and
  $\bar \beta $.
  \begin{lemma}
    \label{lem:BMoverBM}
    (a) There is $\gamma '>0$ such that
    $\mathbb P$-a.s.~for all $n$ large enough
    \begin{equation*}
      \begin{split}
        &P^{\zeta ,\eta }\big(\overbar \beta^1_k\ge 0 \,\forall 1 \le k\le n,
          \overbar\beta^1_n\ge n^{1/4}\big) \ge n^{-\gamma '},\\
        &P^{\zeta ,\eta }\big(\overbar \beta^2_k\ge 0 \,\forall 1 \le k\le n,
          \overbar\beta^2_n\ge n^{1/4}\big) \ge n^{-\gamma '}.
      \end{split}
    \end{equation*}

    (b) There is $C_2>0$ such that $\mathbb P$-a.s.~for $n$ large enough,
    \begin{equation*}
      P^{\zeta ,\eta }\Big(\max_{1\le k\le n} \max_{i=1,2}
        |\overbar\beta^i_k-\overbar\beta^i_{k-1}|
        \le C_2\ln n\Big)\ge 1-n^{-3\gamma '}.
    \end{equation*}

    (c) Let $\delta\in(0,1)$.
    There is $c>0$ such that $\mathbb P$-a.s., for all $x>0$,
    \begin{equation*}
      P^{\zeta ,\eta }\big(\beta_1 \in [x,x+\delta ]\big)
      \ge  c \delta  e^{-x/c}.
    \end{equation*}
  \end{lemma}

  We now complete the proof of \eqref{eq:Pbeta}. With
  \eqref{eq:betaDist} and \eqref{eq:betanEq} we get
  \begin{equation*}
    \begin{split}
      &\{\beta_k\ge -(\alpha -1)\psi^\xi(k)  \, \forall 1 \le k<n, \beta_n\in I_n\}
      \\&\supset \Big (
      \{\overbar\beta^1_k\ge 0 \, \forall 1\le k \le n, \,
        \overbar\beta_n^1 \ge n^\frac14\}
      \cap\{\max_{1\le k\le n} |\overbar\beta^1_k-\overbar\beta^1_{k-1}| \le
        C_2\ln n\}
      \\&\quad\cap\{\overbar \beta^2_k\ge 0 \, \forall 0\le k \le n, \,
        \overbar \beta^2_n \ge n^\frac14 \}
      \cap\{\max_{1\le k\le n} |\overbar\beta^2_k-\overbar\beta^2_{k-1}| \le
        C_2\ln n\}
      \\&\quad\cap\{\beta^1_1 \in
        (I_n- \overbar\beta^1_{\Sigma_n-1}
        +\overbar \beta^2_{n-\Sigma_n})\cap [0,\infty)\}\Big).
      \end{split}
  \end{equation*}
  Indeed, the first and third event on the right-hand side ensure that
  the trajectories of $\overbar \beta^1$ and $\overbar \beta^2$ cross as
  on Figure~\ref{fig:ballot}, and stay above the barrier, which together
  with $\beta^1_1\ge 0$ of the fifth event ensures that $\beta_k$ stays
  above the barrier as required. The second and
  the fourth event then ensure that at the time of crossing they are
  `sufficiently close' (which is not necessary for the inclusion to hold,
    but will be useful later). The fifth event in addition ensures that
  $\beta_n\in I_n$, cf.~\eqref{eq:betanEq}.
  \begin{figure}
    \includegraphics[width=12cm]{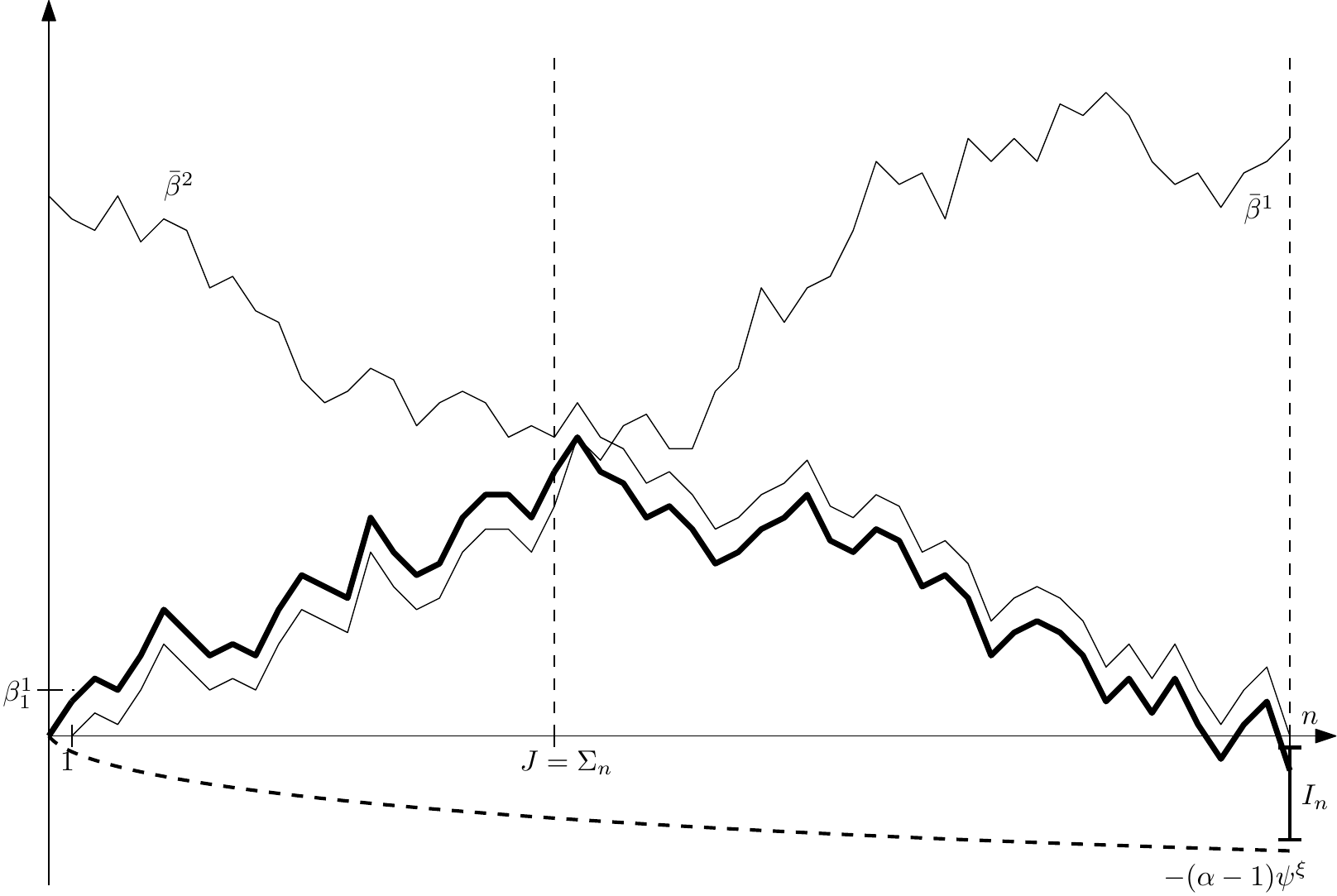}
    \caption{
      \label{fig:ballot}
      Construction of $\beta$ (thick line) from $\overbar \beta^1$,
      $\overbar\beta^2$ (thin lines) and $\beta^1_1$. For clarity of
      presentation, $\overbar \beta^1$ is drawn starting from 1, and
      $\overbar\beta^2$ is drawn running backwards from $n$.}
  \end{figure}

  By Lemma~\ref{lem:BMoverBM}(a,b) and the independence of
  $\overbar \beta^1$, $\overbar \beta^2$ under $P^{\zeta,\eta}$, the
  probability of the intersection of the first four events on the
  right-hand side of the last display is at least $n^{-2\gamma '}$. In
  addition, if these four events occur, then  there is
  $J \in \{1, \ldots,  n\}$ such that
  $I_n- \overbar\beta^1_{J-1} +\overbar \beta^2_{n-J}\subset [0, 2C_2\ln n]$.
  Moreover, $\P(\Sigma_n = J)= 1/(n-1)$. Hence, using
  Lemma~\ref{lem:BMoverBM}(c), conditionally on the occurrence of the
  first four events and $\Sigma_n=J$, we can bound the probability of the
  fifth event on the right-hand side from below by
  $c' n^{-1} e^{-C_2\ln n/c}\ge n^{-\gamma ''}$. Combining these
  estimates proves that $\mathbb P$-a.s.~for $n$ large, \eqref{eq:Pbeta},
  and thus also \eqref{eq:PHoverR}, is larger than $n^{-\gamma }$ with
  $\gamma > 1+2\gamma ' + \gamma ''$. This completes the proof of
  \eqref{eq:ENLbb} and thus of Lemma~\ref{lem:lowerleading}.
\end{proof}

We proceed by proving Lemma~\ref{lem:BMoverBM} which we used in the last
proof.
\begin{proof}[Proof of Lemma~\ref{lem:BMoverBM}]
  Throughout the proof we use the fact that the three processes $\beta $,
  $\beta^1$ and $\beta^2$ have, under $P^{\zeta ,\eta }$, the same
  distribution as $(\widehat H_k-R'_k)_{k\ge 0}$. Since the statements of
  the lemma depend only on the respective distribution, we can and will
  therefore assume that these three processes are equal to
  $(\widehat H_k-R'_k)_{k\ge 0}$. In particular, their increments satisfy
  \begin{equation}
    \label{eq:betainc}
    \beta_k-\beta_{k-1}= \widehat H_k - R'_k - (\widehat H_{k-1}-R'_{k-1}) =
    \tau_k +\Big(- \frac{L^\zeta_k(\eta )}{v_0 L(\eta )}\Big),
  \end{equation}
  where the last equality follows from definitions of $\widehat H$, $R'$
  and $\rho$. It is also useful to observe that, due to
  \eqref{eq:xiAssumptions}, the second summand on the
  right-hand side of \eqref{eq:betainc} satisfies
  \begin{equation}
    \label{eq:betacor}
   -\frac{1}{C} \ge - \frac{L^\zeta_k(\eta )}{v_0 L(\eta )} > -C,
    \qquad \text{for all $k\in \mathbb N$, $\P$-a.s.}
  \end{equation}
  for some constant $C \in (0,\infty)$,
  and that the $\tau_i$ are unbounded non-negative random variables with
  uniform exponential tail (cf.~\eqref{eq:tauexpmoments}), i.e., there
  exists $c>0$ such that for all $k\ge 0$, $\P$-a.s.,
  \begin{equation}
    \label{eq:betatail}
    P^{\zeta ,\eta }(\tau_k\ge u)\le e^{-cu} \qquad
    \text{for all  $u\ge 0$}.
  \end{equation}
  In particular, in combination with \eqref{eq:betacor} we infer that
  there is a small constant $c>0$ such that $\mathbb P$-a.s.,
  \begin{equation}
    \label{eq:betatwoside}
    P^{\zeta ,\eta }\big(\beta_k-\beta_{k-1}>c\big)>c \quad \text{and}\quad
    P^{\zeta ,\eta }\big(\beta_k-\beta_{k-1}<-c\big)>c.
  \end{equation}

  The claim (b) of the lemma then readily follows from
  \eqref{eq:betainc}--\eqref{eq:betatail}, using a union bound and the
  fact that the increments of $\overbar \beta^i$ correspond directly to
  increments of $\beta^i$, $i=1,2$.

  To prove (c), we write
  $\beta^1_1 = \tau_1-L^\zeta _1(\eta )/(v_0 L(\eta ))$, by
  \eqref{eq:betainc}. Recalling \eqref{eq:betacor}, it is sufficient to
  show that there exists $c >0$ such that, $\P$-a.s.,~we have
  $P^{\zeta ,\eta }(\tau_1\in [x+y,x+y+\delta] )\ge c \delta e^{-x/c}$,
  uniformly over $y\in [0,C]$. To see this, recall that under
  $P^{\zeta ,\eta }$, $X$ is a Markov chain whose jump rate from $0$ is
  bounded uniformly in $\zeta $, again by \eqref{eq:xiAssumptions}. If
  the waiting time of $X$ at $0$ is in the required interval, and the
  first jump of $X$ is to the right, then the required event is realized,
  proving (c).

  Claim (a) is the most difficult. We first prove it for
  $\overbar \beta^1$, and explain the modification required to show it
  for $\overbar \beta^2$ at the end of the proof. To simplify
  notation, we consider $\beta $ instead of $\overbar\beta^1$. This is
  possible since $\overbar \beta^1$ has the same distribution as $\beta $
  in the environment shifted by one.

  In the proof we often split the random environment $\xi $ into two
  parts $\underline \xi (j) = (\xi(k))_{k\le j}$ and
  $\overbar \xi (j)=(\xi(k))_{k>j}$.  Set $t_0=t_{-1}=0$ and $t_i=2^i$,
  for $i\ge 1$. Fix $a\in(0,\infty)$ and for $i\ge 1$ define random
  variables $Z_i$ by
  \begin{equation*}
    \begin{split}
      Z_i &:= \essinf_{\underline \xi (t_{i-2})}
      \inf_{x\ge a t_{i-1}^{1/2}}
      P^{\zeta ,\eta }\big(\beta_{t_i}\ge a t_i^{1/2},
        \beta_k\ge t_i^{1/4} \,\forall k\in \{t_{i-1},\ldots, t_i\} \, \big| \,
        \beta_{t_{i-1}} = x\big)
      \\&=
      \essinf_{\underline \xi (t_{i-2})}
      P^{\zeta ,\eta }\big(\beta_{t_i}\ge a t_i^{1/2},
        \beta_k\ge t_i^{1/4} \,\forall k\in \{t_{i-1},\ldots, t_i\} \, \big| \,
        \beta_{t_{i-1}} = at_{i-1}^{1/2}\big).
    \end{split}
  \end{equation*}
  Here,  $\essinf_{\underline \xi (t_{i-2})}$  means taking the essential
  infimum with respect to $\underline \xi (t_{i-2})$ and leaving the
  remaining $\xi$ random. The second equality then follows from the obvious
  monotonicity of the considered event in the starting position. Observe
  that the random variable $Z_i$ is $\sigma (\xi(k),t_{i-2}<k\le t_i)$
  measurable, that is the sequence $(Z_i)$ is 1-dependent.

  Setting $i(n)=\lceil \log_2 n \rceil$, using the Markov property, for
  $n$ large enough,
  \begin{equation*}
    P^{\zeta ,\eta }(\beta_k\ge 0 \,\forall k\le n, \beta_n\ge n^{1/4})\ge
    \prod_{i=1}^{i(n)} Z_i = \exp\Big\{\sum_{i=1}^{i(n)}\ln Z_i\Big\}.
  \end{equation*}
  If we show that $\P$-a.s.,
  \begin{equation}
    \label{eq:ZLLN}
    \limsup_{k\to\infty} \frac 1 k \sum_{i=1}^k  (-\ln Z_i) \le c
    <\infty,
  \end{equation}
  then the first half of claim (a) will follow with $\gamma' > c/\ln 2$.

  First, we claim that, $\mathbb P$-a.s., $-\ln Z_i<\infty$.  Indeed,
  recalling \eqref{eq:betatwoside}, it is easy to see that the
  probability in the definition of $Z_i$ is always positive. If we show
  that the $(-\ln Z_i)$'s have uniformly small exponential moments, then
  \eqref{eq:ZLLN} will follow by standard arguments, using also the
  1-dependence of the sequence $Z_i$.

  To finish the proof, it is therefore sufficient to show that there
  exists some small $\theta >0$ such that for all $i$ large enough
  \begin{equation}
    \label{eq:Zexpmom}
    \mathbb E[\exp\{-\theta \ln Z_i\}]\le c < \infty.
  \end{equation}
  Throughout the proof of this inequality, $i$ is considered fixed and we
  often omit it from the notation. To gain more independence, again, we
  introduce $\overbar \rho_k^{(j)}$, $j<k$, by
  \begin{equation}
    \label{eq:barrho}
    \overbar \rho_k^{(j)} := \esssup_{\underline \xi(j)} \rho_k.
  \end{equation}
  Note that $\rho_k$ is a $\sigma (\xi(n),n\le k)$-measurable random
  variable and thus $\overbar \rho_k^{(j)}$ is
  $\sigma (\xi(n),j<n\le k)$-measurable.
  We further write
  $\overbar R_n^{(j)} = \sum_{k=1}^n \overbar \rho_k^{(j)}$ and note that
  the increments of $\overbar R^{(j)}$ provide upper bounds for the
  increments of~$R'$.

  Let $M_R$
  be the essential supremum of the absolute value of the
  increments $\rho_k$ of $R'$, which is finite by Lemma~\ref{lem:RtoBM}.
  Set
  \begin{equation}
    \label{eq:LtempDef}
    L:=a t_i^{1/2},
  \end{equation}
  $r_0:=t_{i-1}$, and define
  \begin{equation*}
    s_0 := \inf \Big\{k\ge r_0: R'_k-R'_{r_0}\ge \frac L8  \Big\}
    \wedge t_i.
  \end{equation*}
  Further, recursively for $j\ge 1$, we define
  \begin{equation}
    \label{eq:sj}
    \begin{split}
      r_{j+1} & :=s_j +  \Big \lceil \frac L {8M_R} \Big \rceil,\\
      s_{j+1} &:= \inf \Big\{k\ge r_{j+1}: \overbar R_k^{(s_j)}
        -\overbar R_{r_{j+1}}^{(s_j)}\ge \frac L8  \Big\}
      \wedge \big(r_{j+1} + (t_i-t_{i-1})\big).
    \end{split}
  \end{equation}
  Heuristically, $s_j$ is the first time when $\overbar R$ (and thus
    possibly also $R'$) ``increases considerably after time $r_j$'';
  due to the definition of $\beta_k$ in \eqref{eq:beta}, such a behavior
  of $R'$ is potentially dangerous for the event in $Z_i$, in that it
  might lead to $Z_i$ being very small. By definition, $s_{j+1}$ depends
  only on $\xi(l)$ with $l>s_j$, so the increments $s_j-r_j$ are
  independent under $\mathbb P$ and bounded by $t_i-t_{i-1}$.

  For $j\ge 0$ consider the events
  \begin{align*}
    \mathcal G_j&=\Big\{\beta_{s_j}\ge 2 L,
      \inf_{r_j\le l \le s_j} \widehat H_l- \widehat H_{r_j}\ge
      -\frac L8 \Big\},\\
    \mathcal G'_j&=\Big\{\inf_{s_j\le l \le r_{j+1}}
      \widehat H_l- \widehat H_{s_j}\ge -\frac L8 \Big\},
  \end{align*}
  and define
  \begin{equation*}
    J=\inf\{j:s_j-r_j\ge t_{i}-t_{i-1}\}.
  \end{equation*}
  Finally, set
  $\mathcal G=\bigcap_{j=0}^J \mathcal G_j \cap
  \bigcap_{j=0}^{J-1}\mathcal G'_j$.

  We claim that this construction ensures that
  \begin{equation}
    \label{eq:pathconst}
    Z_i \ge P^{\zeta ,\eta }\big(\mathcal G \, | \, \beta_{t_{i-1}}=a
      t_{i-1}^{1/2}\big).
  \end{equation}
  To see this, observe that in each of the time intervals $[r_j,s_j]$ and
  $[s_j,r_{j+1}]$, the process $\overbar R$ (and thus also $R'$) moves
  upwards by at most $L/8+M_R$
   by definition of these intervals. On the other
  hand, on $\mathcal G$ the process $\widehat H$ moves downwards by at
  most  $L/8$
   in any of these intervals. Since in the probability defining
  $Z_i$ we condition on $\beta_{r_0}=L/\sqrt 2>L/2$ and, on $\mathcal G$,
  $\beta_{s_j}\ge 2L$, this ensures that $\beta_k\ge c L^{1/2}\ge t_{i}^{1/4}$ for
  $k\in [r_0,s_0]$  and $\beta_k \ge L $ for $k\in[s_0,s_J]$. Moreover,
  on $ \mathcal G$, $s_J\ge t_i$, proving \eqref{eq:pathconst}.

  Using the independence of the increments of $\widehat H$ under the
  measure $P^{\zeta ,\eta }$, the monotonicity of
  $x\mapsto P^{\zeta ,\eta }(\mathcal G_j \, |\, \beta_{r_j}=x)$, and the
  fact that $J$ is $\sigma (\xi (x):x\in \mathbb Z)$-measurable, we get
  \begin{equation}
    \label{eq:aa}
    \begin{split}
      P^{\zeta ,\eta }&\big(\mathcal G \, |\,
        \beta_{t_{i-1}}=a t_{i-1}^{1/2}\big)
      \\&\ge
      P^{\zeta ,\eta }\big(\mathcal G_0 \, |\,  \beta_{r_0}=L/\sqrt 2\big)
      \prod_{j=1}^J P^{\zeta ,\eta }(\mathcal G_j \, |\, \beta_{r_j}=2L)
      \prod_{j=0}^{J-1} P^{\zeta ,\eta }(\mathcal G'_j).
    \end{split}
  \end{equation}
  It is not difficult to show, using the independence and the uniform
  exponential tail of the increments of
  $\widehat H$ as well as the fact that they are centered, that if $i$ is large enough,
  $P^{\zeta ,\eta }(\mathcal G'_j)\ge \frac 12$ for all $j$.
  On the other hand, for $j\ge 1$,
  \begin{equation}
    \begin{split}
      P^{\zeta ,\eta }&(\mathcal  G_j\, |\, \beta_{r_j}=2L)
      \\&\ge
      P^{\zeta ,\eta }\Big(\widehat H_{s_j}-\widehat H_{r_j}\ge \frac{5L}2,
        \inf_{r_j\le l \le s_j} \widehat H_l-\widehat H_{r_j}\ge - \frac
        L8\Big).
    \end{split}
  \end{equation}
  Observing that the increments of $\widehat H$ are independent under
  $P^{\zeta,\eta}$ and the considered events are both increasing in those
  increments, we can use the Harris-FKG inequality to bound this from
  below by
  \begin{equation}
    \label{eq:yyy}
    \begin{split}
       & P^{\zeta ,\eta }\Big(\widehat H_{s_j}-\widehat H_{r_j}\ge \frac{5L}2 \Big)
        \times P^{\zeta ,\eta } \Big(\inf_{r_j\le l \le s_j} \widehat H_l-\widehat H_{r_j}\ge - \frac
        L8\Big)
      \\&\ge
      c \exp\Big\{-\frac {L^2}{c(s_j-r_j)}\Big\},
    \end{split}
  \end{equation}
  with a sufficiently small constant $c>0$.  To obtain the last
  inequality, we used Azuma's inequality (together with the fact that
    $\widehat H$ is a martingale under $P^{\zeta,\eta}$ and the variances of
    its increments are uniformly bounded, by \eqref{eq:xiAssumptions}),
  and as well as Gaussian scaling to
  infer that the second factor is bounded from below by a constant (since
    $s_j-t_j \le c L^2$). By changing the constant
  $c$, the same lower bound holds for the first term on the right-hand
  side of \eqref{eq:aa} as well.

  Coming back to \eqref{eq:Zexpmom}, using
  \eqref{eq:pathconst}--\eqref{eq:yyy}, recalling that the increments of
  $R'$ are exponentially mixing (cf.~\eqref{eq:Rprimecorrelations}) and
  that the intervals $[r_j,s_j]$ are separated by spaces of length
  $\frac{L}{8M_R}$, we obtain
  \begin{equation}
    \begin{split}
      \label{eq:qqc}
      \mathbb E&[\exp\{-\theta \ln Z_i\}]
      \le C \,
      \mathbb E\Big[\exp\Big\{
          -\theta J\ln \frac c2
          + \theta \sum_{j=0}^J\frac{L^2}{c(s_j-r_j)}\Big\}\Big]
      \\&\le C \sum_{k=1}^\infty
      \Big(\frac 2 c \Big)^{\theta k}
      \mathbb E\Big[
        e^{ \frac{\theta L^2}{c(s_k-r_k)}}
        \ind_{s_k-r_k=t_{i}-t_{i-1}}
        \prod_{j=0}^{k-1}
        e^{
          \frac{\theta L^2}{c(s_j-r_j)}}
        \ind_{s_j-r_j<t_{i}-t_{i-1}}\Big]
      \\&= C\sum_{k=1}^\infty
      \Big(\frac 2 c \Big)^{\theta k}
      e^{
        \frac{2 \theta L^2}{ct_i}}
      \prod_{j=0}^{k-1}
      \mathbb E\Big[
        e^{
          \frac{\theta L^2}{c(s_j-r_j)}}
        \ind_{s_j-r_j<t_{i}/2}\Big],
      \end{split}
  \end{equation}
  where in the equality we used the independence of the $s_j-r_j$'s
  under $\mathbb P$. To upper bound the last expectation, we rewrite it as
  \begin{equation}
    \begin{split}
      \label{eq:qqb}
      &\int_{0}^\infty
      \mathbb P \Big(
        e^{
          \frac{\theta L^2}{c(s_j-r_j)}}
        \ind_{s_j-r_j<t_{i}/2}> a \Big) \, \d a
      \\&\le
      e^{\frac{2\theta L^2}{ct_i}}
      \mathbb P(s_j-r_j< t_i/2)
      + \int_{e^{\frac{2\theta L^2}{ct_i}}}^\infty
      \mathbb P\Big(e^{ \frac{\theta L^2}{c(s_j-r_j)}}>a\Big)
      \,\d a.
    \end{split}
  \end{equation}
  Substituting $a=\exp\big\{\frac{\theta L^2}{cy}\big\}$, the second summand can
  be written as
  \begin{equation}
    \label{eq:qqa}
    \int_{0}^{t_i/2}\mathbb P(s_j-r_j<y)\frac{\theta L^2}
    {cy^2}e^{\frac{\theta L^2}{cy}}\,\d y.
  \end{equation}
  Recalling the definition \eqref{eq:sj} of $s_j$,  for $i$ sufficiently
  large we have for $0 \le y \le t_{i}/2 = \frac{L^2}{2a^2}$ that
  \begin{align*}
      \mathbb P(s_j-r_j<y)
      &= \mathbb P\Big(\max_{0\le m\le y}
        \sum_{k=1}^{m}\overbar \rho_{r_j+k}^{(s_{j-1})}\ge \frac L8\Big)\\
      &\le \mathbb P\Big(\max_{1\le m\le y}
        \sum_{k=1}^{m}\rho_{r_j+k}\ge \frac L9\Big);
  \end{align*}
  here, to obtain the inequality one takes advantage of the estimates
  \eqref{eq:LesssBd} and \eqref{eq:L'esssBd}, which yield that uniformly
  in $0 \le j \le k,$
  \begin{equation*}
    0 \le \overbar \rho_k^{(j)}  - \rho_k \le C_\Delta e^{- (k-j)/C_\Delta},
  \end{equation*}
  and thus, using  $r_j - s_{j-1} \ge cL$, that
  \begin{equation*}
    0 \le  \overbar \rho_{r_j+k}^{(s_{j-1})} - \rho_{r_j+k} \le Ce^{- cL},
    \quad \forall k \in \Big\{0, \ldots, \frac{L^2}{2a^2} \Big\}.
  \end{equation*}
  Inequality \eqref{eq:Rprimecorrelations} can then be used
  to verify the assumption of Azuma's inequality for mixing sequences of
  Lemma \ref{lem:hoef} for the sequence $\rho_k$, and thus
  \begin{equation*}
    \mathbb P\Big(
      \sum_{k=1}^{m}\rho_{r_j+k}\ge \frac L9\Big)\le C e^{-cL^2/m},
  \end{equation*}
  for some constants $C$ and $c$ and all admissible $m$. This inequality
  extends to a maximal inequality, as follows from
  \cite[Theorem~1]{KeMa-11},
  \begin{equation*}
    \mathbb P\Big( \max_{0\le m\le y}
      \sum_{k=1}^{m}\rho_{r_j+k}\ge \frac L9\Big)\le C e^{-cL^2/y}.
  \end{equation*}
  Inserting these back into \eqref{eq:qqa} implies that the second summand
  in \eqref{eq:qqb} is smaller than
  \begin{equation*}
    \int_{0}^{t_i/2}\frac{C\theta L^2}{y^2}
    e^{\frac{(\theta - c) L^2}{c'y}}\,\d y
    =\int_{0}^{\frac{1}{2a^2}}
    \frac{C\theta}{z^2} e^{\frac{(\theta -c)}{c'z}} \, \d z,
  \end{equation*}
  which can be made arbitrarily small by choosing $\theta $ small. In
  addition, the first summand in \eqref{eq:qqb} is strictly smaller than
  $1$ by the functional central limit theorem from Lemma~\ref{lem:RtoBM},
  hence the right-hand side of \eqref{eq:qqb} is strictly smaller than
  one for all $\theta >0$ sufficiently small. Therefore, for $\theta $
  small enough the sum in \eqref{eq:qqc} converges, which implies
  \eqref{eq:Zexpmom} and completes the proof of the first claim in
  Lemma~\ref{lem:BMoverBM}(a).

  The proof of the second claim is very similar, so we only explain the
  modifications which need to be introduced due to the fact that
  $\overbar \beta^2$ is `running backwards', and thus its dependence on the
  environment $\zeta $ is different.
  The first modification
  involves the definition of $Z_i$ where the $\essinf$ should be taken
  over $\underline \xi (n-t_{i+1})$. This makes $Z_i$ measurable with
  respect to $\sigma (\xi (k):n-t_{i+1}<k\le n-t_{i-1})$, and thus
  $(Z_i)$ still is a 1-dependent sequence.
  Furthermore, the definitions of $s_j$, $r_j$ should be replaced by
  $r_0= t_{i-1}$, and
  \begin{equation*}
    \begin{split}
      s_0 &:= \inf \Big\{k\ge r_0:
        \overbar R_{n-k}^{(n-k-\ell)}-\overbar R_{n-r_0}^{(n-k-\ell)}
        \ge \frac L8  \Big\}
    \wedge t_i,\\
      r_{j+1} & :=s_j +\ell  ,\\
      s_{j+1} &:= \inf \Big\{k\ge r_{j+1}:
        \overbar R_{n-k}^{(n-k-\ell)} -\overbar R_{n-r_{j+1}}^{(n-k-\ell)}
        \ge \frac L8  \Big\}
      \wedge \big(r_{j+1} + (t_i-t_{i-1})\big),
    \end{split}
  \end{equation*}
  with $\ell :=  \lceil  L/ {8M_R} \rceil$, which again makes the
  increments $(s_j-r_j)$ independent under $\mathbb P$.

  With these modifications, the second claim in (a) can be shown almost
  exactly as the first one, which completes the proof of the lemma.
\end{proof}

\subsubsection{Second moment for the leading particles}

We now estimate the second moment of the number of leading particles
needed for the application of~\eqref{eq:PZ}. The proof is relatively
short because we do not try to get the optimal power $\gamma $ below.

\begin{lemma}
  \label{lem:upperleading}
  There exists a constant $\gamma <\infty$ such that $\mathbb P$-a.s.~for
  all $t$ large enough,
  \begin{equation*}
    \ttE_0^\xi [ (N^{\mathcal L}_t)^2] \le t^\gamma.
  \end{equation*}
\end{lemma}

\begin{proof}
  Recall the definition \eqref{eq:leading} of leading particles. Since we
  look for an upper bound, we can ignore the condition
  $Y_{\overbar m(t)}\ge \overbar m(t)$ there.  We define
  a random function $\varphi^\xi: \mathbb R^+\to \mathbb R^+$ by
  \begin{equation}
    \label{eq:varphiDef}
    \varphi^\xi(s):=k \qquad \text{for all }s\in [T_{k} - \alpha  \psi^\xi(k),
      T_{k+1}-\alpha \psi^\xi(k+1)), \quad k \in \N_0,
  \end{equation}
  where $\psi^\xi$ as in \eqref{eq:psi} and
  $T_0-\alpha \psi^\xi(0):=-\infty$, by convention. By
  \eqref{eq:secMomFormula} of Proposition~\ref{prop:FK} we then have
  \begin{equation}
    \begin{split}
      \label{eq:ENsqbb}
      &\ttE_0^\xi [ (N^{\mathcal L}_t)^2]
      \\&\le
      \ttE_0^\xi [  N^{\mathcal L}_t]
      + 2 \int_0^t  E_0\Big[
        \exp\Big\{\int_0^s \xi (X_r) \,\d r\Big\}
        \xi (X_s) \ind_{X_r\le \varphi^\xi(r)\forall r\in [0,s]}
        \\&\times
        \Big(E_{X_s}\Big[
            \exp\Big\{\int_0^{t-s} \xi (X_r) \,\d r\Big\}
            \ind_{X_r\le \varphi^\xi(s+r) \forall r\in [0,t-s],X_{t-s}\ge \overbar
              m(t)}\Big]\Big)^2
        \Big] \, \d s.
    \end{split}
  \end{equation}

  In the
  upper bound of \eqref{eq:ENsqbb} we will repeatedly use the
  perturbation Lemmas~\ref{lem:upperTail} and \ref{lem:timepert} in a
  neighborhood  $(t,\overbar m(t))$. This is always justified $\mathbb P$-a.s.~for
  $t$ large enough, observing also that $\overbar m(t)/t\to v_0\in V$,
  $\mathbb P$-a.s., by Corollary~\ref{cor:breakLLN}. In particular, by
  Lemma~\ref{lem:upperTail}(b) and the definition of $\overbar m(t)$,
  $\P$-a.s.~for $t$ large enough,
  \begin{equation}
    \label{eq:aroundbreak}
    \ttE^\xi_0\big[N^\ge (t,\overbar m(t))\big]
    \le C \ttE^\xi_0\big[N^\ge (t,\overbar m(t)+1)\big]
    \le C/2.
  \end{equation}
  The first summand on the right-hand side of \eqref{eq:ENsqbb} then satisfies
  \begin{equation*}
    \ttE_0^\xi [  N^{\mathcal L}_t]
    \le \ttE_0^\xi[N^{\ge}(t,\overbar m(t))]
    \le C.
  \end{equation*}
  Since $\xi (X_s)\le \es$, the second summand on the right-hand side of
  \eqref{eq:ENsqbb} is bounded from above by
  \begin{equation}
    \begin{split}
      \label{eq:sma}
      & 2\es \int_0^t \sum_{k=-\infty}^{\varphi^\xi(s)}
      E_0\bigg[
        e^{\int_0^s \xi (X_r) \, \d r}
        \ind_{X_s=k}
        \Big(E_{k}\Big[
            e^{\int_0^{t-s} \xi (X_r) \,\d r};
            X_{t-s}\ge \overbar m(t)\Big]\Big)^2
        \bigg] \, \d s
      \\&=
      2\es \int_0^t \sum_{k=-\infty}^{\varphi^\xi(s)}
      \ttE^\xi_0[N(s,k)]
      \ttE^\xi_k[N^{\ge}(t-s,\overbar m(t))]^2
      \, \d s.
    \end{split}
  \end{equation}
  To find an upper bound for the integral on the
  right-hand side of \eqref{eq:sma}, we remark that,
  by the first moment formula \eqref{eq:firMomFormula} of
  Proposition~\ref{prop:FK}, the Markov property, and
  \eqref{eq:aroundbreak}, $\P$-a.s. for $t$ large enough,
  \begin{equation*}
    \begin{split}
      \ttE_0^\xi [N(s,k)] \ttE_k^\xi [N^\ge(t-s,\overbar m(t))]
      &= \ttE^\xi_0\big[|\{Y\in N(t), Y_t\ge \overbar m(t),
          Y_s=k\}|\big]
      \\& \le \ttE^\xi_0\big[N^\ge (t,\overbar m(t))\big]\le C.
    \end{split}
  \end{equation*}
  Hence, $\P$-a.s.~for $t$ large enough, uniformly in
  $s\in [0,t]$, $k\le \varphi^\xi (s)$,
  \begin{equation}
    \label{eq:NgeleN}
    \ttE_k^\xi [N^\ge(t-s,\overbar m(t))]\le C/\ttE_0^\xi [N(s,k)].
  \end{equation}
  In order to take advantage of
  \eqref{eq:NgeleN}, we treat separately four ranges of
  parameters $s\in [0,t]$ and $k\le \varphi^\xi(s)$ in \eqref{eq:sma}.

  (A) We start with considering the range
  \begin{equation}
    \label{eq:rangeA}
    \mathcal N_1(\xi) \le k \le \varphi^\xi(s), s \ge \mathcal S(\xi ),
    \text{ such that } k/s\in V,
  \end{equation}
  where $\mathcal N_1(\xi )$ and $\varphi^\xi $ are defined
  in~\eqref{eq:Vinterval}, \eqref{eq:Hn}, and \eqref{eq:varphiDef},
  respectively, and $\mathcal S(\xi )$ is a $\sigma (\xi) $-measurable
  random variable which is a.s.~finite and which will be specified
  below.  In this case, by Lemma~\ref{lem:inCond},
  \begin{equation*}
    \ttE^\xi_0[ N (s,k)] \ge
    c \, \ttE^\xi_0[ N^\ge (s,k)]
    \ge c \, \ttE^\xi_0[ N^\ge (s,\varphi^\xi(s))],
  \end{equation*}
  where the last inequality follows from $k\le \varphi^\xi(s)$. Let
  $l=l(s)$  be such that
  $s \in [T_l - \alpha \psi^\xi(l),T_{l+1} - \alpha \psi^\xi(l+1))$; note
  that $\varphi^\xi(s) =l$. Let $\mathcal S(\xi )$ be a $\P$-a.s.~finite
  random variable such $s\ge S(\xi)$ implies that $l/T_l \in V$,
  $l / (T_{l}-\alpha \psi^\xi(l)) \in V$ and
  $s\ge T_l-\alpha \psi^\xi (l)\ge \mathcal T_1\vee \mathcal T_2$, where
  $\mathcal T_1$ and $\mathcal T_2$ are as in
  Lemmas~\ref{lem:upperTail}(b) and~\ref{lem:timepert}(b).
  The existence of such $\mathcal S$ is implied by the law of large
  numbers \eqref{eq:LLNTn} for $T_n$. Using then repeatedly
  Lemma~\ref{lem:timepert}(b) and
  Remark~\ref{rem:hnegative}, the
  right-hand side of the previous display can be bounded from below by
  \begin{equation*}
    c\,
    \ttE^\xi_0[ N^\ge (T_{l}-\alpha \psi^\xi(l),l)]
    \le c\,
    \ttE^\xi_0[ N^\ge (T_{l},l)] e^{-c \psi^\xi (l)}
    \ge \mathcal C'(\xi ) t^{-\gamma},
  \end{equation*}
  for some $\gamma \in (0,\infty)$ and a positive random variable
  $\mathcal C'(\xi )$, where in the last inequality we used
  $\ttE^\xi_0[ N^\ge (T_l,l)] = 1/2$, and
  $\psi^\xi (l)=\mathcal C(\xi ) + C_1(1\vee \ln l)$,
   the need for which
  emanates from the randomness of $\psi^\xi$. Thus, combining the last
  two displays with \eqref{eq:NgeleN} we infer that  $\mathbb P$-a.s~for $t$ large enough,
  uniformly for $k$, $s$ as in \eqref{eq:rangeA}
  \begin{equation}
    \label{eq:domaina}
    \ttE^\xi_0[N(s,k)]
    \ttE^\xi_k[N^{\ge}(t-s,\overbar m(t))]^2
    \le  C \, \ttE^\xi_0[ N (s,k)]^{-1}
    \le \mathcal C''(\xi ) t^\gamma.
  \end{equation}

  (B) We now consider the ranges
  \begin{equation}
    \label{eq:rangeB}
    s\in [0,t] \text{ and $k$ such that }
    |k/s|\le \bar v/2,
  \end{equation}
  where $\bar v>0$ is the asymptotic speed of the maximal particle in the
  homogeneous branching random walk with branching rate $\ei$
  (cf.~\eqref{eq:xiAssumptions}). We assume without loss of generality
  that $V$ is fixed so that it contains $\bar v/2$ in its interior. Since
  $\xi (x)\ge \ei$, by a straightforward comparison argument and
  properties of the homogeneous branching random walk, we infer the
  existence of some constant $c>0$ such that $\ttE^\xi_0[N(s,k)]\ge c$
  for all $s$ and $k$ as in~\eqref{eq:rangeB}. Therefore, by
  \eqref{eq:NgeleN}, $\P$-a.s.~for $t$ large enough, uniformly for $s$,
  $k$ as in \eqref{eq:rangeB},
  \begin{equation}
    \label{eq:domainb}
    \ttE^\xi_0[N(s,k)]
    \ttE^\xi_k[N^{\ge}(t-s,\overbar m(t))]^2
    \le   c \, \ttE^\xi_0[ N (s,k)]^{-1} \le C.
  \end{equation}

  (C) Now let
  \begin{equation*}
    s\in [0,t] \text{ and }k\le 0.
  \end{equation*}
  By the Feynman-Kac formula, using also $\essinf \xi \ge 0$,
  \begin{align*}
    2 \ttE^\xi_k[N^{\ge}(t,\overbar m(t))]
    &\ge 2 E_k\Big[ \exp \Big\{\int_0^{t-s} \xi(X_r) \, \d r \Big\}
      \ind_{X_{t-s} \ge \overbar m(t)} \ind_{X_{t} \ge \overbar m(t)} \Big]\\
    &\ge 2 \ttE^\xi_k[N^{\ge}(t-s,\overbar m(t))] P_0(X_s \ge 0)
    \ge  \ttE^\xi_k[N^{\ge}(t-s,\overbar m(t))].
  \end{align*}
  Therefore,
  \begin{align*}
    \begin{split}
      \ttE^\xi_0&[N(s,k)]
      \ttE^\xi_k[N^{\ge}(t-s,\overbar m(t))]^2
      \\& \le  2 \ttE^\xi_0[N(s,k)]
      \ttE^\xi_k[N^{\ge}(t-s,\overbar m(t))]
      \ttE^\xi_{k} [N^{\ge}(t,\overbar m(t))]
    \end{split}
  \end{align*}
  For $k\le 0$, by the monotonicity in the initial condition, taking
  advantage of Lemma~\ref{lem:inCond}, $\P$-a.s.~for $t$~large enough,
  \begin{equation*}
    \ttE^\xi_{k} [N^{\ge}(t,\overbar m(t))]
    \le \ttE^\xi_{\ind_{-\N_0}} [N^{\ge}(t,\overbar m(t))]
    \le C \ttE^\xi_{0} [N^{\ge}(t,\overbar m(t))] \le C.
  \end{equation*}
  Combining the last two inequalities, applying also Markov property and
  \eqref{eq:aroundbreak},
  $\mathbb P$-a.s. for $t$ large enough, uniformly in $s\in [0,t]$,
  \begin{equation}
    \label{eq:domainc}
    \begin{split}
      \sum_{k=-\infty}^{0}&
      \ttE^\xi_0[N(s,k)]
      \ttE^\xi_k[N^{\ge}(t-s,\overbar m(t))]^2
      \\&\le C \sum_{k=\infty}^0
      \ttE^\xi_0[N(s,k)]
      \ttE^\xi_k[N^{\ge}(t-s,\overbar m(t))]
      \\&\le C \ttE^\xi_0[N^{\ge}(t,\overbar m(t))]
      \le C.
    \end{split}
  \end{equation}

  (D) The remaining part of the range of parameters relevant in
  \eqref{eq:sma}, which is not controlled by (A)--(C), is a subset of
  \begin{equation*}
    \mathcal B^\xi =\{(s,k)\in [0,\infty)\times \mathbb N:
      \bar v s/ 2\le k \le \varphi^\xi (s), s+k\le \tilde {\mathcal C}(\xi)\}
  \end{equation*}
  for some finite random variable $\tilde{\mathcal C}(\xi)$ depending on
  $\mathcal N_1$ and $\mathcal S$. Observe that $\mathcal B^\xi$ is a
  bounded set for $\P$-a.e.~$\xi$.

  We start with observing that there is a constant $L>0$ such that
  \begin{equation}
    \label{eq:mtgrowth}
    \text{for $t$ large enough, $\P$-a.s.,} \quad
    \overbar m(t+1)-\overbar m(t)\le L.
  \end{equation}
  Indeed, by the perturbation Lemmas~\ref{lem:upperTail}(b)
  and~\ref{lem:timepert}(b), and \eqref{eq:aroundbreak},
  \begin{equation*}
    \begin{split}
      \ttE^\xi_0[N^\ge(t+1,\overbar m(t)+L)]
      &\le C e^{C} \ttE^\xi_0[N^\ge(t,\overbar m(t)+L)]
      \\&\le C' e^{-cL} \ttE^\xi_0[N^\ge(t,\overbar m(t))]
      \le C'' e^{-cL}.
    \end{split}
  \end{equation*}
  Choosing $L$ to make the right-hand side smaller than $1$ then yields
  \eqref{eq:mtgrowth}.

  The definition of $\overbar m(t)$, inequality \eqref{eq:mtgrowth},
  and Lemma~\ref{lem:upperTail}(b) imply together that,  $\P$-a.s.~for $t$
  large enough,
  \begin{equation*}
    \begin{split}
      1/2&\ge \ttE^\xi_0[N^{\ge}(t+1,\overbar m(t+1)+1)]
      \\&\ge c e^{-c (\overbar m(t+1)+1-\overbar m(t)}
      \ttE^\xi_0[N^{\ge}(t+1,\overbar m(t)]
      \\&\ge c e^{-c L}
      \ttE^\xi_0[N^{\ge}(t+1,\overbar m(t)].
      \end{split}
  \end{equation*}
  Hence with $C = e^{cL}/(2c)$,
  $\P$-a.s.~for $t$ large enough, using also the Markov property,
  \begin{equation*}
    \begin{split}
      C&\ge \ttE^\xi_0[N^{\ge}(t+1,\overbar m(t))]
      \\&\ge  \ttE^\xi_0[N(s+1,k)]
      \ttE^\xi_k \big[N^\ge \big(t-s, \overbar m(t)\big)\big].
    \end{split}
  \end{equation*}
  By the boundedness of $\mathcal B^\xi$ and \eqref{eq:xiAssumptions}, there
  is a random variable $\mathcal C(\xi )\in (0,\infty)$ such that, for all
  $(s,k)\in \mathcal B^\xi$,
  \begin{equation*}
    \mathcal C(\xi)\ge \ttE^\xi_0[N(s+1,k)]\ge P_0[X_{s+1}=k]\ge \mathcal C(\xi)^{-1}.
  \end{equation*}
  Combining the last two inequalities then yields
  \begin{equation*}
    \ttE^\xi_k \big[N^\ge \big(t-s, \overbar m(t)\big)\big]\le \mathcal C '(\xi )
  \end{equation*}
  for all $(s,k)\in \mathcal B^\xi$, $\P$-a.s.~for $t$ large enough.
  Hence, following the same arguments as before, using \eqref{eq:aroundbreak},
  we obtain
  \begin{equation}
    \label{eq:domaind}
    \begin{split}
      \ttE^\xi_0[N(s,k)]
      \ttE^\xi_k[N^{\ge}&(t-s,\overbar m(t))]^2
      \\&\le \mathcal C'(\xi )\ttE^\xi_0[N(s,k)]
      \ttE^\xi_k[N^{\ge}(t-s,\overbar m(t))]
      \\&\le  \mathcal C'(\xi)
      \ttE^\xi_0[N^{\ge}(t,\overbar m(t))]
      \le C \mathcal C'(\xi),
    \end{split}
  \end{equation}
  uniformly for $(s,k)\in \mathcal B^\xi$, $\P$-a.s.~for $t$ large enough.

  Using inequalities \eqref{eq:domaina}, \eqref{eq:domainb},
  \eqref{eq:domainc}, and \eqref{eq:domaind} in their respective domains
  in the summation and integration in \eqref{eq:sma} (recalling that $V$
    contains $\bar v/2$ in its interior), we can, $\P$-a.s.~for $t$ large
  enough, bound the second summand on the right-hand side of
  \eqref{eq:ENsqbb} from above by
  \begin{align*}
    \mathcal C''(\xi) t^{\gamma + 2} + Ct^2 + C t + C\mathcal C'(\xi),
  \end{align*}
  where the summands correspond to cases (A)--(D) above. This completes
  the proof of the lemma.
\end{proof}

Combining Lemmas~\ref{lem:lowerleading} and~\ref{lem:upperleading} with
\eqref{eq:PZ} completes the proof of
Proposition~\ref{prop:leadingparticles}.

\subsection{Proof of Theorem \ref{thm:logCorrection} and Proposition~\ref{prop:Mclosetom}} 
\label{ssec:maxproofs}

By inserting the estimates from Lemmas~\ref{lem:lowerleading}
and~\ref{lem:upperleading} into the Paley-Zygmund
inequality~\eqref{eq:PZ} we obtain
$\ttP^\xi_0(N^{\mathcal L}_t\ge 1)\ge t^{-\gamma }$ for all large $t$,
$\mathbb P$-a.s. To complete the proof of the lower bound in
Theorem~\ref{thm:logCorrection} we need to amplify this estimate, using a
technique adapted from the homogeneous branching random walk literature
(see e.g.~\cite{McD-95}). The first step is the following lemma
guaranteeing  that with very high probability the
number of particles in the origin grows exponentially in time.

\begin{lemma}
  \label{lem:expgrowthzero}
  There exists $C_3 > 1$ and $t_0<\infty$ such that such that for all
  $t \ge t_0$, and $\mathbb P$-a.e.~$\xi$,
  \begin{equation*}
    \ttP_0^\xi (  N(t,0)  \le C_3^t)\le C_3^{-t}.
  \end{equation*}
\end{lemma}

\begin{proof}
  Recall from \eqref{eq:xiAssumptions} that the essential infimum $\ei$
  of the $\xi$ is strictly positive. By the monotonicity of $N(t,0)$ in
  $\xi$ which can be ensured by a straightforward coupling, it suffices
  to show the claim for the homogeneous branching random walk with
  branching rate $\ei$. We write $\ttP^{\ei}_0$ for the law of this
  process starting in $0$.

  For $t \ge 0$ and $\varepsilon >0$,
  let $D_\varepsilon (t/3)$ be the set of
  \emph{direct} offsprings of the initial particle
  until time $t/3$ which are at sites $[-\varepsilon t,\varepsilon t]$
  at time $t/3$. Then, for any
  $\varepsilon > 0$  there exists $\delta > 0$ such that
  \begin{equation}
    \label{eq:ampaa}
    \ttP^\ei_0\big( \vert D_\varepsilon (t/3)
      \vert \le \delta t/3\big)
    \le e^{-\delta t/3}.
  \end{equation}
  Indeed, the probability that the initial particle%
  \footnote{For
      simplicity of redaction, in a slight abuse of notation we
      reformulate the original branching mechanism which replaces a
      particle by two new particles, by the equivalent branching
      mechanism where instead particles are not replaced and give birth
      to one more particle.}
  leaves $[-\varepsilon t/2, \varepsilon t/2]$ before $t/3$ is smaller
  than $e^{-c(\varepsilon )t}$. If it stays in this interval, it produces
  more than $t\, \ei /4$ direct offsprings with probability larger than
  $1-e^{-ct}$, by large deviations for the Poisson distribution with
  parameter $t\, \ei /3$, and every of these offsprings stays in
  $[-\varepsilon t,\varepsilon t]$ with probability at least
  $1-e^{-c(\varepsilon )t}$ again.

  For a particle $Y\in D_\varepsilon (t/3)\subset N(t/3)$, we denote by
  $A_Y(2t/3)$ the set of \emph{all} offsprings it produced between times
  $t/3$ and $2t/3$ and which are at the site $Y_{t/3}$ at time $2t/3$.
  We claim that
  there exists $c> 1$ such that
  \begin{equation}
    \label{eq:ampbb}
    \ttP^\ei_0 \Big( \big| A_Y(2t/3)\big|
      \ge c^t  \Big)
    >0
  \end{equation}
  Indeed, under $\ttP^\ei$, it is well-known (e.g., it follows from the
    Feynman-Kac formula) that the expected number of particles in $0$
  grows exponentially. Hence, we can fix $r>0$ such that
  $\ttE^\ei[N(r,0)]=:\mu >1$, and consider an auxiliary process evolving
  as follows
  \begin{itemize}
    \item
    start with one particle at an arbitrary site $x\in \mathbb Z$ at time $0$,

    \item
    particles evolve independently as a continuous time simple random walk
    and split into two at rate $\ei$,

    \item
    and at each time $r n$, $n\in \mathbb N$,
     all the particles not at $x$ are killed.
  \end{itemize}
  Let $Z_n$ be the number of particles at $x$ at time $rn$ in this
  auxiliary process. It is easy to see that $Z_n$ is a supercritical
  Galton-Watson process and thus it survives with a positive probability,
  $\ttP^\ei_0(Z_n>0 \, \forall n \ge 0)\ge p>0$, and on the event of
  survival it grows exponentially,
  $\ttP^\ei_0(Z_n\ge c^n\, |\, Z_n >0 \, \forall n \ge 0)\ge 1/2$ for
  some $c>1$. Hence, for every $Y\in D_\varepsilon (t/3)$,
  $\ttP^\ei_0(|A_Y(2t/3)|\ge c^{t})\ge p/2$ at all times such that
  $2t/3=rn$ for some $n\in \mathbb N$. A straightforward extension to all
  times then yields \eqref{eq:ampbb}.

  Combining \eqref{eq:ampaa} and \eqref{eq:ampbb} implies that
  \begin{equation*}
    \ttP^\ei_0 \big (N(2t/3,[-\varepsilon t,\varepsilon t])\ge c^t \big)\ge
    1-e^{-c't}.
  \end{equation*}
  Moreover, the constant $c$ does not depend on $\varepsilon $.
  As a consequence, choosing $\varepsilon >0$
  small enough such that
  \begin{equation*}
    P_{\varepsilon t}(X_{t/3} = 0) \ge \Big(\frac{1+c}2\Big)^{-t/3},
  \end{equation*}
  and using an easy large deviation argument we obtain that
  \begin{equation*}
    \ttP^\ei_0\Big(N(t,0)<\frac{c^t}2
      \Big(\frac{1+c}2\Big)^{-t/3}\ \Big|\
      N(2t/3,[-\varepsilon t,\varepsilon t])\ge
      c^t\Big)\le e^{-c''t}.
  \end{equation*}
  This completes the proof of the lemma.
\end{proof}

We now obtain a lower bound on $M(t)$.

\begin{proposition}
  \label{prop:MaxPartLag}
  For any $q \in \N$ there exists a constant $C^{(q)} < \infty$ such that
  for $\P$-a.a. $\xi$, for all $t$ large enough
  \begin{equation*}
    \ttP_0^\xi(M(t) \ge \overbar m(t) - C^{(q)}\ln t) \ge 1 - 2 t^{-q}.
  \end{equation*}
\end{proposition}

\begin{proof}
  Without loss of generality we assume that $q>\gamma $ for $\gamma $ as
  in Proposition~\ref{prop:leadingparticles}. We fix $r=c_1 \ln t$ where
  $c_1$ is chosen so that for $C_3$ of Lemma~\ref{lem:expgrowthzero} we
  have $C_3^{-r}= t^{-q}$, and further choose $C^{(q)}$ large enough so
  that $\overbar m (t-r)\ge \overbar m(t)-C^{(q)} \ln t$.
  To see that this is possible, observe that by the perturbation
  Lemmas~\ref{lem:upperTail} and~\ref{lem:timepert} we have for some
  $c,c'\in (0,\infty)$
  \begin{equation*}
    \begin{split}
      \ttE^\xi_0[N^\ge(t-r,\overbar m(t)-C^{(q)}\ln t)]
      &\ge e^{-cr}
      \ttE^\xi_0[N^\ge(t,\overbar m(t)-C^{(q)}\ln t)]
      \\&\ge e^{-cr}e^{c'C^{(q)}\ln t}
      \ttE^\xi_0[N^\ge(t,\overbar m(t))]
      \\&\ge \frac 12 e^{-cr}e^{c'C^{(q)}\ln t},
    \end{split}
  \end{equation*}
  and fix $C^{(q)}$ so that the right-hand side is smaller than $1/2$.

  Set $x:=\overbar m(t)-C^{(q)} \ln t$ and observe that by considering
  separately the events $\{N(r,0)<C_3^r\}$, $\{N(r,0)\ge C_3^r\}$ and using
  the Markov property and the independence of the particles in the second
  case
  \begin{equation*}
    \begin{split}
    \ttP_0^\xi (M(t) \ge x) &= \ttP_0^\xi (N^{\ge}(t,x)  \ge 1)
      \\&\ge 1 -  \ttP_0^\xi ( N(r,0) \le C_3^{r})
      - \big( \ttP_0^\xi (N^\ge(t-r,x) < 1)\big)^{C_3^{r}}.
      \\&\ge 1 -  t^{-q}
      - \big( \ttP_0^\xi (N^{\mathcal L}_{t-r} < 1)\big)^{C_3^{r}}.
    \end{split}
  \end{equation*}
  Here, for the last inequality we used Lemma \ref{lem:expgrowthzero} as
  well as $x\le \overbar m(t-r)$ and so
  $N^\ge(t-r,x)\ge N^{\mathcal L}_{t-r}$.
  Proposition~\ref{prop:leadingparticles} then implies
  \begin{equation*}
      \big( \ttP_0^\xi (N^{\mathcal L}_{t-r} < 1)\big)^{C_3^{r}}
      \le   (1-t^{-\gamma })^{t^q} \le t^{-q}
  \end{equation*}
  for $t$ large enough. This completes the proof.
\end{proof}

\begin{proof}[Proof of Theorem~\ref{thm:logCorrection} and
    Proposition~\ref{prop:Mclosetom}]
  \ Proposition~\ref{prop:MaxPartLag} and Borel-Cantelli lemma
  (controlling non-integer $t$ by standard estimates) imply that
  $M(t) \ge \overbar m(t) - C^{(2)} \ln t$,
  $\mathbb P\times \ttP^\xi_0$-a.s.~for
  all $t$ large enough, and thus  $\mathbb P$-a.s.,
  $m(t)\ge \overbar m(t)-C^{(2)}\ln t$, for such $t$ as well. By the
  monotonicity and the independence of the particles, these lower bounds
  hold for an arbitrary initial condition satisfying \eqref{eq:inCond}.
  These facts combined with \eqref{eq:Mtupperbound} and
  $m(t)\le \overbar m(t)$ (cf.~\eqref{eq:barmm}), complete the proof
  of Theorem~\ref{thm:logCorrection} and Proposition~\ref{prop:Mclosetom}.
\end{proof}

\section{BRWRE and the randomized Fisher-KPP} 
\label{sec:FKPPproofs}

In this section we prove the central limit theorem for the front of the
randomized Fisher-KPP equation. We begin by establishing the connection
between the BRWRE and this Fisher-KPP equation. Its proof is a
straightforward adaptation of \cite{MK-75}, who proved the corresponding
result in the case of the homogeneous BBM (see also
  \cite{IkNaWa-68a,IkNaWa-68b,IkNaWa-69}).

\begin{proposition}
  \label{prop:KPPBRW}
  For a bounded $(\xi(x))_{x \in \Z}$  and $f: \Z \to [0,1]$
  \begin{equation*}
    w(t,x) := 1- {\tt E}^\xi_{x} \Big[ \prod_{Y \in N(t)} f(Y_t) \Big]
  \end{equation*}
  solves
  \begin{equation*}
    \frac{\partial w}{\partial t} = \Delta_{{\mathrm d}} w + \xi(x) w(1-w)
  \end{equation*}
  with initial condition $w(0, \cdot) =1- f$.
  In particular, $w(t,x)= \ttP^\xi_x(M(t)\ge 0)$ solves this equation with
  $f=\ind_{-\mathbb N}$, i.e.~$w(0,\cdot)=\ind_{\N_0}$.
\end{proposition}

\begin{proof}
  Actually, we show that $v:= 1-w$ solves
  \begin{equation*}
    \frac{\partial v}{\partial t} = \Delta_{{\mathrm d}} v - \xi(x) v(1-v)
  \end{equation*}
  with initial condition $v(0,\cdot) = f$, which will establish the claim.

  According to whether or not the original particle has split into two
  before time $t$, the Feynman-Kac formula in combination with the Markov
  property at time $s$ supplies us with
  \begin{equation*}
      v(t,x) =
      E_x \Big [ e^{-\int_0^t \xi(X_r) \,\d r} f(X_t) \Big]+
      \int_0^t
      E_x \Big[ \xi(X_s) e^{-\int_0^s \xi(X_r) \, \d r}
        v^2(t-s,X_s) \Big] \, \d s
  \end{equation*}
  Using the reversibility of the random walk, and substituting $s$ by $t-s$,
  this can be written as
  \begin{equation}
    \label{eq:vDecomp}
    \begin{split}
      v(t,x)&=\sum_{y\in \mathbb Z}f(y)
      E_y \Big [ e^{-\int_0^t \xi(X_r) \, \d r} \ind_x(X_t) \Big]
      \\&+
      \int_0^t \sum_{y\in \mathbb Z} v^2(s,y)
      E_y \Big [ \xi(y) e^{-\int_0^{t-s} \xi(X_r) \, \d r}
        \ind_{x}(X_{t-s}) \Big ] \, \d s.
    \end{split}
  \end{equation}
  Differentiation then yields
  \begin{align*}
    \frac{\partial v}{\partial t}(t,x)
    &=
    \sum_{y\in \mathbb Z}f(y)
    E_y \Big  [ -\xi(x) e^{-\int_0^t \xi(X_r) \, \d r} \ind_x (X_t) \Big]
    \\&+ \sum_{y\in \mathbb Z} f(y)
    E_y \Big [ e^{-\int_0^t \xi(X_r) \,  \d r}  (\Delta_{{\mathrm d}} \ind_x)(X_t) \Big ]
    \\&+ \xi(x) v^2(t, x)
    \\&+\int_0^t \sum_{y\in \mathbb Z} v^2(s,y)
    E_y\Big  [ -\xi(y)\xi (x)
      e^{-\int_0^{t-s} \xi(X_r) \, \d r} \ind_x(X_{t-s})\Big] \, \d s
    \\&+ \int_0^t\sum_{y\in \mathbb Z} v^2(s,y)
    E_y \Big [ \xi(y)  e^{-\int_0^{t-s} \xi(X_r) \, \d r}
      (\Delta_{{\mathrm d}} \ind_x)(X_{t-s})\Big] \, \d s.
  \end{align*}
  Comparing this expression with the representation~\eqref{eq:vDecomp},
  the second and fifth summands together yield $\Delta_{{\mathrm d}} v$, the third is
  $\xi(x) v^2$, and the first and fourth together supply us with
  $-\xi(x) v$, which finishes the proof of the first claim. The second
  claim is a straightforward consequence of the first one.
\end{proof}

\begin{proof}[Proof of Theorem~\ref{thm:FKPPCLT}]
  Observe that the initial conditions in Theorem~\ref{thm:FKPPCLT}
  and the second claim of Proposition~\ref{prop:KPPBRW} are related by the
  reflection $x\mapsto -x$. Hence, setting $\tilde \xi (x)=\xi(-x)$,
  it is easy to see from the last proposition that the front
  $\widehat m(t)$ of the Fisher-KPP
  equation defined in \eqref{eq:dKPPfront} can be represented as
  \begin{equation*}
    \widehat m(t) =  \sup \Big \{x \in \Z  :
      \ttP_{-x}^{\tilde\xi} (M(t)\ge 0 ) \ge \frac12 \Big\}.
  \end{equation*}
  Comparing this to the definition \eqref{eq:median} of $m(t)$, we see
  that the role of $x$ and the origin is reversed, and the environment is
  reflected. This
  complication is easy to resolve. By the translation and reflection
  invariance of the environment $\xi$, for every $x\in \mathbb Z$,
  \begin{equation*}
    \mathbb P\Big(\ttP_{-x}^{\tilde \xi} (M(t)\ge 0)\ge \frac 12\Big) =
    \mathbb P\Big(\ttP_{0}^\xi (M(t)\ge x)\ge \frac 12\Big).
  \end{equation*}
  The central limit theorem for $\widehat m(t)$ then follows
  from the one for~$m(t)$.
\end{proof}

\section{Open questions} 
\label{sec:openquestions}

We collect here some open questions which naturally arise from the
investigations of this article.

\begin{enumerate}
  \item
  Can we say that $m(t)$ lags at least $\Omega(\ln t)$ behind
  $\overbar m(t)$?

  \item
  For $x \in \Z$ fixed, is the function $[0,\infty) \ni t \mapsto u(t,x)$
  increasing? It is not hard to see that generally  this is not the case
  on $[0,\infty)$; however, is it true for $t$ large enough?

  \item
  Is the family $M(t) - m(t)$, $t \ge 0$,  tight? In the case of
  homogeneous BBM, it already follows from the convergence to a traveling
  wave solution (see \cite{KoPePi-37}) that this is the case. In the case
  of spatially random branching rates this remains an open question.

\end{enumerate}

We expect our results to transfer to the continuum setting where the
space $\Z$ is replaced by $\R$ under suitable regularity and mixing
assumptions on $\xi$.

\appendix
\section{Auxiliary results}
\label{sec:auxProofs}

We prove here several auxiliary results that are used through the text.
Most of them use rather standard techniques, but we did not find any
suitable reference for them.

\subsection{Properties of logarithmic moment generating functions} 

\begin{lemma} \label{lem:lambdaProps}
  The functions $L$,  $L_i^\zeta$, and $\overbar L^\zeta_n$ defined in
  \eqref{eq:empL}--\eqref{eq:Ldef} are infinitely
  differentiable on $(-\infty, 0)$ and satisfy for $\eta \in (-\infty,0]$
  \begin{align}
    \label{eq:lambdaDer}
    L'(\eta)
    &= \E \Big[
      \frac{ E^\zeta[H_1 e^{\eta H_1}] }
      {E^\zeta[e^{\eta H_1}]} \Big]
    = \E \big [ E^{\zeta, \eta} [H_1] \big],\\
    \label{eq:lambdaQuenchedDer}
    (L_i^\zeta)'(\eta)
    &=  \frac{ E^\zeta[\tau_i e^{\eta \tau_i}] }
    {E^\zeta[e^{\eta \tau_i}]}
    =  E^{\zeta, \eta} [\tau_i],
  \end{align}
  (where the derivative in $0$ should be interpreted as the derivative from
    the left)
  and thus also
  $(\overbar L_n^\zeta)' (\eta ) = \frac 1n E^{\zeta ,\eta }[H_n]$.
  Further
  \begin{align}
    \label{eq:lambdaSecDer}
    \begin{split}
      L''(\eta)
      &= \E \big[ E^{\zeta, \eta}[H_1^2]  - E^{\zeta, \eta}[H_1]^2 \big] >0,
    \end{split}
    \\
    \label{eq:lambdaQuenchedSecDer}
    \begin{split}
      ( L_i^\zeta)''(\eta)
      &= \big(  E^{\zeta, \eta}[\tau_i^2]
        - E^{\zeta, \eta}[\tau_i]^2 \big) >0,
    \end{split}
  \end{align}
  and thus also
  $(\overbar L_n^\zeta)'' (\eta )
  = \frac 1n \big(  E^{\zeta, \eta}[H_n^2]  - E^{\zeta, \eta}[H_n]^2 \big)$.
  Moreover, for every $\Delta \subset (-\infty,0)$ compact, there is
  $ c(\Delta )\in (0,\infty)$ such that
  \begin{equation*}
    \sup_{\eta\in \Delta  }\esssup |L_i^\zeta (\eta )| \le  c(\Delta),
  \end{equation*}
  and analogous statements hold for $(L_i^\zeta)' $ and $(L_i^\zeta)''$ as
  well.
\end{lemma}

\begin{proof}
  The fact that $L $ and $\overbar L^\zeta_n$ are infinitely
  differentiable follows easily from the dominated convergence theorem
  which allows to interchange the differentiation with the expectations.
  The  equalities \eqref{eq:lambdaDer} and \eqref{eq:lambdaQuenchedDer}
  can  then be obtained by a direct computation from the definitions of
  the corresponding functions. The  equalities \eqref{eq:lambdaSecDer}
  and \eqref{eq:lambdaQuenchedSecDer}  follow from the definition
  \eqref{eq:tiltedVarTime} of $P^{\zeta, \eta}$. The strict inequalities
  in \eqref{eq:lambdaSecDer} and \eqref{eq:lambdaQuenchedSecDer} follow
  from the fact that, as $H_1$, $\tau_i$ are non-degenerate random
  variables, Jensen's inequality provides us with a strict inequality.

  To prove the last claim, we start with observing that
  $\zeta \mapsto L^\zeta_i$ and $\zeta \mapsto (L^\zeta_i)'$
  are increasing as mappings from $[\ei-\es,0]^\Z$ to $(-\infty, 0]$, and
  $\zeta(x) \in [\ei-\es,0]$ $\P$-a.s. Therefore,
  \begin{equation*}
    \begin{split}
      -\infty&<\ln E_{i-1}[e^{H_i (\ei-\es+\min \Delta) }]\le
      \inf_{\eta\in \Delta } \essinf L_i^\zeta (\eta )
      \\&\le
      \sup_{\eta\in \Delta } \esssup L_i^\zeta (\eta )  \le
      \ln E_{i-1}[e^{H_i \max \Delta }]<\infty,
    \end{split}
  \end{equation*}
  where the finiteness of the expectations on the left- and right-hand
  side follows easily since both $\ei-\es+\min \Delta$ and $\max \Delta$
  are negative.
  The proof for
  $(L_i^\zeta)'$  is similar. For  $(L_i^\zeta)''$ we use
  \eqref{eq:lambdaQuenchedSecDer} as well as the definition of
  $E^{\zeta,\eta}$ to observe that for $\eta \in \Delta$,
  \begin{align*}
    \sup_{\eta\in \Delta}\esssup |(L_i^\zeta)'' (\eta )|
    &\le   \sup_{\eta\in \Delta}\esssup  E^{\zeta, \eta} [\tau_i^2]  \\
    & \le  \frac{E_{i-1}[H_i^2 e^{H_i \max \Delta}]}
    {E_{i-1}[e^{H_i (\min \Delta +\ei - \es)}]}  < \infty,
  \end{align*}
  which finishes the proof.
\end{proof}

\begin{lemma} \label{lem:covBd}
  Let $\mathcal F_k=\sigma (\xi(i):i\le k)$ and $\Delta $ be a compact
  interval in $(-\infty,0)$. Then there
  exists a constant $C_\Delta  < \infty$ such that for all $0 \le i < j$
  and $\eta \in \Delta $, $\P$-a.s.,
  \begin{align}
    \label{eq:LcondBd}
    \big | \E[  L_j^\zeta(\eta) \, | \, \mathcal F_i] - L(\eta) \big |
    &\le C_\Delta e^{- (j-i)/C_\Delta},\\
    \label{eq:LesssBd}
    0 \le \Big( \esssup_{\zeta(k) \, : \, k\le i} L_j^\zeta(\eta) \Big) -   L_j^\zeta(\eta)
    &\le C_\Delta e^{- (j-i)/C_\Delta},
  \end{align}
  and similarly
  \begin{align}
    \label{eq:L'condBd}
    \big | \E[  (L_j^\zeta)'(\eta) \, | \, \mathcal F_i] - L'(\eta) \big |
    &\le C_\Delta e^{- (j-i)/C_\Delta},\\
    \label{eq:L'esssBd}
    0 \le \Big( \esssup_{\zeta(k) \, : \, k\le i} (L_j^\zeta)'(\eta) \Big) -   (L_j^\zeta)'(\eta)
    &\le C_\Delta e^{- (j-i)/C_\Delta}.
  \end{align}
\end{lemma}

\begin{proof}
  We only prove the  inequalities \eqref{eq:LcondBd} and
  \eqref{eq:LesssBd}, the remaining ones being derived in a similar manner.

  By translation invariance we may assume without loss of generality
  $0 = i < j$. Write $L_j^\zeta (\eta ) = \ln (A+B)$ where
  \begin{align*}
    A &= E_{j-1} \Big [ \exp \Big\{ \int_0^{H_{j}} (\zeta(X_s) + \eta) \, \d s
        \Big\}, {\inf_{0 \le s \le H_{j}}  X_s > 0} \Big],\\
    B &=  E_{j-1} \Big [ \exp \Big\{ \int_{0}^{H_{j}} (\zeta(X_s) + \eta)
        \, \d s \Big\}, {\inf_{0 \le s \le H_{j}}  X_s \le 0} \Big].
  \end{align*}
  Let $K_t$ denote the number of jumps that the random walk $(X_n)$ has
  made up to time $t >0$, which has Poisson distribution with
  parameter $t$. Then, since $\esssup \zeta =0$, for $\delta>0 $
  sufficiently small, uniformly over $\eta \in \Delta $,
  \begin{equation*}
    B \le E_{j-1}[e^{\eta H_{j}};H_{j} \ge \delta j]
    + P_{j-1}(K_{ \delta j} \ge j )
    \le  c e^{-j/c},
  \end{equation*}
  where the last inequality follows from standard large deviations for
  the Poisson random variable. On the other hand, due to
  \eqref{eq:zetaBd},  there is $c' \in (0, 1)$ such that   $\P$-a.s.~one
  has $A \in (c', 1).$ Therefore, since $\ln (1+x)\le x$, we infer that
  $\P$-a.s.
  \begin{equation*}
    \ln (A) \le  L_j^\zeta (\eta) \le \ln (A) + \ln (1+\tfrac BA)
    \le \ln (A) + ce^{-j/c}.
  \end{equation*}
  Since $\ln (A)$ is independent of $\mathcal F_0$ by definition, taking
  the essential supremum over $\zeta(k)$, $k\le 0$, this implies the
  second inequality of \eqref{eq:LesssBd}. The first inequality of
  \eqref{eq:LesssBd} follows from the definitions. Using the independence
  of $\ln (A)$ and $\mathcal F_0$ once again, we also infer that
  \begin{align*}
    \big | \E[  L_j^\zeta(\eta) - L(\eta)\, | \, \mathcal F_0] \big |
    &\le \big | \E[  \ln(A) - \E[\ln (A)] \, | \, \mathcal F_0] \big |
    + 2ce^{-j/c}=2ce^{-j/c},
  \end{align*}
  which finishes the proof of the lemma.
\end{proof}

\subsection{Basic properties of the Lyapunov exponent} 
\label{ssec:lyapExp}

We prove here various properties of the Lyapunov exponent $\lambda $
defined in~\eqref{eq:lyapunov} that are used throughout the paper.
Some of these properties are standard, but for some
of them we did not find any reference. In particular, the proof of $v_c>0$
is presumably new and of independent interest.

\begin{proposition}
  \label{prop:lyapExp}
  Assume \eqref{eq:xiAssumptions}.

  (a) The function $\lambda: \R \to \R$ is well-defined, non-random, even,
  and concave. It satisfies $\lambda(0) = \es$, $\lambda (v)<\es$ for
  every $v\neq 0$, and $\lim_{v \to \infty} \lambda (v)/ v = -\infty$. In
  particular, there exists a unique $v_0\in (0,\infty)$ such that
  $\lambda({v_0}) = 0$.

  (b) There is $v_c\in (0,\infty)$ given by
  $v_c = (L'(0))^{-1}$ (where the derivative is taken from the left only)
  such that $\lambda $ is linear on $[0,v_c]$, and strictly concave on
  $(v_c,\infty)$.  In addition, for every $v\in [0,\infty)$,
  \begin{equation}
    \label{eq:lambdaesL}
    \lambda (v)= \es - vL^*(1/v),
  \end{equation}
  where for $v=0$ the right-hand side is defined as $\es$.
\end{proposition}

\begin{proof}
  (a) For $\alpha \in (0,1)$ and $v_1,v_2 \in \R$ the
  Markov property yields
  \begin{align}
    \label{eq:concaveDist}
    \begin{split}
      \ln{} &E_{0}
      \Big[ e^{ \int_0^{t} \xi(X_s) \, \d s }
        ;  X_t = \floor{(\alpha v_1 + (1-\alpha) v_2)t} \Big]
      \\ &\ge
            \ln E_{0}
      \Big[ e^{ \int_0^{(1-\alpha) t} \xi(X_s) \, \d s }
        ;X_{(1-\alpha)t} = \floor{(1-\alpha) v_2 t} \Big]\\
      &+
      \ln E_{\floor{(1-\alpha) v_2t}}
      \Big[ e^{ \int_0^{\alpha t} \xi(X_s) \, \d s }
        ; X_{\alpha t} =  \floor{(\alpha v_1 + (1-\alpha) v_2)t} \Big].
    \end{split}
  \end{align}
  Hence, choosing $ v_1 := v_2:= v$ and using Kingman's subadditive
  ergodic theorem \cite{Li-85} as well as the Feynman-Kac formula
  (Proposition \ref{prop:FK}), we obtain that for each $v \in \R$, the
  limit $\lambda(v)$ exists  and is non-random. In addition, $\lambda$ is
  an even function since  $X$ is symmetric simple random walk and the
  $(\xi(x))_{x \in \Z}$ are i.i.d.~by assumption.

  Dividing both sides of inequality \eqref{eq:concaveDist} by $t$ and
  taking the limit ${t\to\infty}$, the left-hand side converges $\P$-a.s.~to
  $\lambda (\alpha v_1 + (1-\alpha )v_2)$, and the first summand on the
  right-hand side converges to  $(1-\alpha) \lambda (v_2)$. Further, the
  second summand converges to $\alpha \lambda (v_1)$ in distribution, since
  it has the same distribution (up to possibly a small error  introduced by
    the use of the floor function, and which is irrelevant in the limit) as
  \begin{equation*}
    \frac 1t \ln E_{0} \Big[ \exp \Big\{ \int_0^{\alpha t} \xi(X_s) \,
        \d s \Big\} \ind_{X_{\alpha t} = \floor{\alpha v_1t}} \Big],
  \end{equation*}
  which converges $\P$-a.s.~to $\alpha \lambda(v_1)$. The concavity of
  $\lambda $ then follows.

  The proof of $\lambda(0) = \es$ is standard but we include it for the sake
  of completeness. By the Feynman-Kac formula and \eqref{eq:zetaDef},
  \begin{equation}
    \label{eq:lambdazero}
    \lambda (0)=\es + \lim_{t\to \infty} \frac 1t \ln
    E_0\Big[\exp\Big\{\int_0^t \zeta (X_s)\,\d
        s\Big\}; X_t=0\Big].
  \end{equation}
  Since $\zeta (x)\le 0$, the upper bound $\lambda (0)\le \es$ follows
  trivially. To show the lower bound,  fix $\varepsilon >0$ arbitrarily
  and note that by standard i.i.d.~properties of $\zeta$'s, there is
  $c(\varepsilon )>0$ such that $\mathbb P$-a.s.~for $t$ large enough,
  there is an interval $I_t\subset[-t^{1/4},  t^{1/4}]$ of length at
  least $c(\varepsilon )\ln t$ such  $\zeta(j)\ge -\varepsilon $ for all
  $j\in I_t \cap \Z$. Consider now the event
  $\mathcal A_t=\{X_0=X_t=0,X_s\in I_t\,\forall x\in[t^{1/2},t-t^{1/2}]\}$.
  By a local  central limit theorem,
  $P_0(X_{t^{1/2}}\in I_t)\ge c t^{-1/4}$. By standard spectral estimates
  for the simple random walk, for any $m\in I_t$,
  \begin{equation*}
    P_m(X_s\in I_t\,\forall s\le t-2t^{1/2})\ge e^{-ct/\ln t},
  \end{equation*}
  and, by a local central limit theorem again,
  $P_m(X_{t^{1/2}}=0)\ge c t^{-1/4}$. The Markov property thus yields
  $P_0(\mathcal A_t) \ge e^{-ct/\ln t}$. Going back to
  \eqref{eq:lambdazero}, restricting the expectation to $\mathcal A_t$,
  \begin{equation*}
    \lambda (0)\ge \es + \frac 1t \limsup_{t\to \infty } P_0(\mathcal A_t)
    e^{-2 \es  t^{1/2}} e^{-\varepsilon (t-2t^{1/2})}\ge \es -\varepsilon.
  \end{equation*}
  Since $\varepsilon >0$ was chosen arbitrarily, $\lambda(0) = \es$ follows.

  The fact
  $\lim_{v \to \infty} \lambda (v)/ |v| = -\infty$ follows from
  \eqref{eq:xiAssumptions} and large deviation properties of the
  continuous time simple random walk $X$.

  \smallskip

  (b) The strict concavity of $\lambda(v) $ on $(v_c,\infty)$ is a
  consequence of the strict convexity of $L^*(1/v)$ on this interval,
  which in turn follows from definition \eqref{eq:critVel} of $v_c$, the
  strict convexity of $L$ on $(-\infty,0)$ and standard properties of the
  Legendre transform. Also, for $v \in ( v_c, \infty)$, claim
  \eqref{eq:lambdaesL} is shown in the proof of Theorem~\ref{thm:PAMFCLT}
  in Section~\ref{ssec:PAMFCLT}. By the continuity of $\lambda$ (which
    follows from concavity and finiteness) and the monotonicity and
  lower-semicontinuity of $L^*$ (which entails its left-continuity in
    $1/v_c$), \eqref{eq:lambdaesL} also holds for $v=v_c$. We thus only
  need to show the linearity of $\lambda$ and \eqref{eq:lambdaesL} on
  $[0,v_c)$, and then $v_c\in (0,\infty)$.

  To show the linearity, observe that by the Feynman-Kac representation,
  \begin{equation*}
    \begin{split}
      u(t,vt)&=e^{t \es}E_{0} \Big[ \exp \Big\{ \int_0^{t} \zeta(X_s) \,
          \d s \Big\}; X_{t} = vt\Big]
      \\&\le e^{t \es}E_{0} \Big[ \prod_{i=1}^{vt-1}\exp \Big\{
          \int_{H_{i-1}}^{H_i}\zeta(X_s) \,
          \d s \Big\}\Big]
      = e^{t \es}e^{\sum_{i=1}^{vt} L^\zeta _i(0)}.
    \end{split}
  \end{equation*}
  Taking logarithms and letting $t\to\infty$, it follows that
  \begin{equation}
    \label{eq:linUBLyap}
    \lambda (v)\le \es + vL(0) = \es - vL^*(1/v_c), \quad v \in [0,v_c],
  \end{equation}
  where for the inequality we used \eqref{eq:Birkhoff}, and for the
  equality we took advantage of \eqref{eq:critVel} again. The concavity
  of $\lambda$, the linear upper bound \eqref{eq:linUBLyap}, and the fact
  that $\lambda$ coincides with this linear upper bound for $v = 0$
  (cf.~part (a)) and $v=v_c$ (by \eqref{eq:lambdaesL} for $v=v_c$) then
  imply the matching lower bound, proving the claimed linearity on
  $[0,v_c]$. Claim \eqref{eq:lambdaesL} for $v\in [0,v_c)$
  then follows directly, since we know that the inequality in
  \eqref{eq:linUBLyap} is an equality, and
  $L(0) = - L^*(1/v) = L^*(1/v_c)$ for $v \in [0,v_c]$, by standard
  properties of Legendre transform.

  It remains to show that $ (L'(0))^{-1}\in(0,\infty)$. We recall that by
  \eqref{eq:lambdaDer},
  $L'(0)=\mathbb E [E^{\zeta ,0}[H_1]]$. Taking
  advantage of the boundedness from below of $\zeta$ and the definition
  of $E^{\zeta,0}$, it is sufficient to show that
  \begin{equation}
    \label{eq:expUBbd}
    \E \big[ E_0 \big[H_1 e^{\int_0^{H_1} \zeta(X_s) \, \d s} \big] \big]
    \in (0,\infty).
  \end{equation}
  The lower bound follows easily, as $H_1$ is a non-trivial non-negative
  random variable. For the upper bound, for arbitrary
  fixed $h \in (0, \es -\ei)$ we introduce an
  auxiliary random environment
  \begin{align}
    \label{eq:zetastar}
    \zeta^*(x):=
    \begin{cases}
      0, \qquad & \text{if }\zeta(x) \in (-h, 0],\\
      -h,  & \text{if }\zeta(x) \le -h.
    \end{cases}
    \qquad x\in \mathbb Z.
  \end{align}
  In particular, note that $\zeta \le \zeta^*$ and
  \begin{equation} \label{eq:nonTrivProb}
    p:= \P(\zeta^*(0) = -h) = 1- \P(\zeta^*(0) = 0)  \in (0,1).
  \end{equation}
  Furthermore, defining for $n\in \N$ the events
  \begin{equation*}
    G_1(n) :=\Big\{ \inf_{s \in [0,H_1)}
      X_s \in (-n^{\frac23},-n^{\frac13}) \Big\},
  \end{equation*}
  we infer by standard large deviation estimates for simple random walk
  that there exist constants $c,C \in (0,\infty)$ such that
  \begin{equation}
    \label{eq:rangeLD}
    P_0(H_1 \in [n,n+1),  G_1(n)^c)
    \le Ce^{-cn^\frac13}, \quad \forall n \in \N_0.
  \end{equation}
  The next ingredient is \cite[Theorem 1.3]{AtDrSu-16}, which implies
  that thin points for the random walk are rare in the following sense:
  Write $\ell$ for the local time process
  \begin{equation} \label{eq:localTime}
    \ell_t(x) := \int_0^t \delta_x(X_s) \, \d s,
    \qquad x \in \Z, t \in [0,\infty),
  \end{equation}
  and, for $M \in (0,\infty)$, introduce the set of \emph{thin points} by
  \begin{equation*}
    \mathcal T_{t,M} := \big \{x \in \Z  :  \ell_t(x) \in (0,M] \big\}.
  \end{equation*}
  Then \cite[Theorem 1.3]{AtDrSu-16} entails that setting
  $G_2(t) := \big\{ | \mathcal T_{t,M} | \le t^\frac16 \big\}$,
  there exist  constants $c,C \in (0,\infty)$ such that
  \begin{equation}
    \label{eq:thinPts}
    P_0\big( G_2(H_1)^c \, | \, H_1 \in [n,n+1) \big)
    \le C e^{-c n^\frac17}, \quad \forall n \in \N_0.
  \end{equation}
  The last ingredient is a simple large deviation bound for
  i.i.d.~Bernoulli variables: recalling $p$ from  \eqref{eq:nonTrivProb}
  and setting
  \begin{equation*}
    G_3(n):= \Big\{ \big| \big\{x \in \{ -n,  \dots, 0\}
        :  \zeta^*(x) = -h \big \} \big|
      \ge \frac{p n}{2} \Big\},
  \end{equation*}
  we have due to \eqref{eq:xiAssumptions} that
  for arbitrary $\varepsilon  >0$
  \begin{equation} \label{eq:potLD}
    \P\big( G_3(n)^c \big) \le Ce^{-n(I_p(p/2)-\varepsilon )} \le C e^{-c_pn}, \quad
    \forall n \in \N,
  \end{equation}
  where $I_p$ is the usual rate function of the Bernoulli$(p)$ distribution.
  Combining \eqref{eq:rangeLD} to \eqref{eq:potLD}, we infer that
  \begin{align}
    \label{eq:H1expBd}
    \begin{split}
      &\E \Big[ E_0 \Big[H_1 e^{\int_0^{H_1} \zeta(X_s)\,  \d s} \Big] \Big]
      \\&\le \sum_{n =0}^\infty
      \E \Big[ E_0\Big[H_1 e^{\int_0^{H_1} \zeta(X_s) \,  \d s},
          H_1 \in [n,n+1) \Big],G_3(n^\frac13)\Big]
      + C(n+1)e^{-cn^\frac13}
      \\&\le \sum_{n=0}^\infty \E \bigg[
        E_0\Big[H_1 e^{\int_0^{H_1} \zeta(X_s)\,  \d s},
          H_1 \in [n,n+1), G_1(n), G_2(H_1) \Big],
        G_3(n^\frac13) \bigg]
      \\&\qquad + C(p)e^{-cn^\frac17}
      \\&\le C(p)  +  \sum_{n =0}^\infty  e^{-p h n^\frac13/2} < \infty,
  \end{split}
\end{align}
and the upper bound for \eqref{eq:expUBbd} follows.
\end{proof}

Finally, we shortly discuss the existence of random environments which
satisfy the condition \eqref{eq:vAssumptions} requiring that the speed of
the maximum particle, $v_0$ is strictly larger than the critical speed~
$v_c$. The simple proof of the following result reveals that this is the
case for a very rich family of environments, which heuristically can be
interpreted as exhibiting sufficiently strong branching.
\begin{lemma}
  \label{lem:qualLyap}
 There exist environments $\xi$ such
  that~\eqref{eq:xiAssumptions} and~$\eqref{eq:vAssumptions}$ hold true.
\end{lemma}
\begin{proof}
  Choose an arbitrary random environment $\xi$ fulfilling
  \eqref{eq:xiAssumptions}, and consider a family of environments
  $\xi^h:=(\xi (x)+h)_{x\in \mathbb Z}$, $h\ge 0$.  Writing $\lambda^h$
  for the Lyapunov exponent associated to $\xi^h$, the Feynman-Kac
  representation (Proposition~\ref{prop:FK}) and the definition
  \eqref{eq:lyapunov} of $\lambda $ yield that $\lambda^h(v)=\lambda(v)+h$.
  Hence, by Proposition~\ref{prop:lyapExp}, the value of $v_c$ does not
  change with $h$, and, on the other hand, $v_0 \to \infty$ as
  $h\to \infty$, which entails the desired statement.
\end{proof}

\subsection{Hoeffding type inequality for mixing sequences} 
We repeatedly make use of the following concentration inequality for
mixing sequences. We state it here for reader's convenience.

\begin{lemma}[{\cite[Theorem~2.4]{Ri-13}}]
  \label{lem:hoef}
  Let $(X_i)_{i\in \mathbb Z}$ be a sequence of real valued bounded
  random variables on some $(\Omega ,\mathcal F,\mathbb P)$ and let
  $\mathcal F_i=\sigma (X_j,j\le i)$. Suppose that there are real numbers
  $m_i>0$, $i\in \{1,\dots,n\}$ such that
  \begin{equation*}
    \sup_{j\in \{i+1,\dots,n\}}
    \Big(\|X_i^2\|_\infty
      + 2 \Big\|X_i\sum_{k=i+1}^j \mathbb E[X_k \, |\, \mathcal F_i]\Big\|_\infty\Big)
    \le m_i,\qquad \text{for all $i\le n$.}
  \end{equation*}
  Then for every $a>0$,
  \begin{equation*}
    \P\big(|X_1+\dots+X_n|\ge a\big)
    \le \sqrt e \exp\Big\{-\frac{a^2}{2\sum_{i=1}^n m_i}\Big\}.
  \end{equation*}
\end{lemma}

\bibliographystyle{imsart-number}
\bibliography{mbrwre}

\end{document}